\documentclass[10pt]{amsart}
\usepackage{amsmath}
\usepackage{amsfonts}
\usepackage{amsthm}
\usepackage{enumerate}
\usepackage{amssymb}
\usepackage{amscd}
\usepackage{mathrsfs}\usepackage[all]{xy}\DeclareMathOperator{\gr}{gr}
\DeclareMathOperator{\Gr}{Gr}
\DeclareMathOperator{\Hom}{Hom}
\DeclareMathOperator{\Ext}{Ext}

\DeclareMathOperator{\End}{End}
\DeclareMathOperator{\Supp}{Supp}
\DeclareMathOperator{\Aut}{Aut}
\DeclareMathOperator{\Out}{Out}
\DeclareMathOperator{\Ann}{Ann}

\DeclareMathOperator{\ann}{ann}
\DeclareMathOperator{\Sat}{Sat}

\DeclareMathOperator{\Loc}{Loc}

\DeclareMathOperator{\id}{id}
\DeclareMathOperator{\gldim}{gld}
\DeclareMathOperator{\im}{Im}
\DeclareMathOperator{\coker}{coker}
\DeclareMathOperator{\coh}{coh}
\DeclareMathOperator{\Ch}{Ch}
\DeclareMathOperator{\Spec}{Spec}

\DeclareMathOperator{\Sym}{Sym}
\DeclareMathOperator{\Gal}{Gal}
\DeclareMathOperator{\tr}{tr}
\DeclareMathOperator{\GL}{GL}
\DeclareMathOperator{\SL}{SL}
\begin{document}
\newtheorem{MainThm}{Theorem}
\renewcommand{\theMainThm}{\Alph{MainThm}}
\theoremstyle{definition}
\newtheorem*{defn}{Definition}
\newtheorem*{examp}{Example}
\newtheorem*{examps}{Examples}
\theoremstyle{remark}
\newtheorem*{rmk}{Remark}
\newtheorem*{rmks}{Remarks}
\theoremstyle{plain}
\newtheorem*{lem}{Lemma}
\newtheorem*{prop}{Proposition}
\newtheorem*{thm}{Theorem}
\newtheorem*{example}{Example}
\newtheorem*{examples}{Examples}
\newtheorem*{cor}{Corollary}
\newtheorem*{conj}{Conjecture}
\newtheorem*{hyp}{Hypothesis}
\newtheorem*{thrm}{Theorem}
\newtheorem*{quest}{Question}
\theoremstyle{remark}

\newcommand{\Grass}{\mathcal{G}}
\newcommand{\Zp}{{\mathbb{Z}_p}}
\newcommand{\Qp}{{\mathbb{Q}_p}}
\newcommand{\Fp}{{\mathbb{F}_p}}
\newcommand{\Fl}{\mathcal{B}}
\newcommand{\bFl}{\Fl_k}
\newcommand{\cFl}{\widehat{\Fl}}
\newcommand{\Ex}[1]{\widetilde{#1}}
\newcommand{\bM}{\overline{\mathcal{M}}}
\newcommand{\F}{\mathcal{F}}
\newcommand{\M}{\mathcal{M}}
\newcommand{\A}{\mathcal{A}}
\newcommand{\C}{\mathcal{C}}
\newcommand{\hM}{\widehat{\M}}
\newcommand{\hA}{\widehat{\A}}
\newcommand{\hAK}{\widehat{\A_K}}
\newcommand{\N}{\mathcal{N}}
\newcommand{\D}{\mathcal{D}}
\newcommand{\T}{\mathcal{T}}
\newcommand{\dnl}{\D^\lambda_n}
\newcommand{\dnt}{\Ex{\D}_n}
\newcommand{\h}[1]{\widehat{#1}}
\newcommand{\hK}[1]{\h{#1_K}}
\newcommand{\hc}[1]{\widehat{\mathscr{#1}}}
\newcommand{\hn}[1]{\h{#1_n}}
\newcommand{\hnK}[1]{\h{#1_{n,K}}}
\newcommand{\hdnl}{\widehat{\dnl}}
\newcommand{\hdnt}{\widehat{\dnt}}
\newcommand{\msD}{\mathscr{D}}
\newcommand{\sDa}{\mathscr{D}_a}
\newcommand{\hD}{\widehat{D}}
\newcommand{\hdnlK}{\widehat{\D^\lambda_{n,K}}}
\newcommand{\hdntK}{\widehat{\Ex{\D}_{n,K}}}
\newcommand{\sdnl}{\mathscr{D}^\lambda}
\newcommand{\sdnlK}{\mathscr{D}^\lambda_K}
\newcommand{\dlK}{\D^\lambda_K}
\newcommand{\unl}{\mathcal{U}^\lambda_n}
\newcommand{\uhn}{\U{h}_n}
\newcommand{\unG}{U(\fr{g})_n^{\mb{G}}}
\newcommand{\hunl}{\widehat{\mathcal{U}^\lambda_n}}
\newcommand{\hunlK}{\widehat{\mathcal{U}^\lambda_{n,K}}}
\newcommand{\hunt}{\widehat{\U{g}_n}}
\newcommand{\hunG}{\widehat{U(\fr{g})_n^{\mb{G}}}}
\newcommand{\GhuntK}{\mathcal{U}_n \ast H_n}
\newcommand{\hugnK}{\widehat{\U{g}_{n,K}}}
\newcommand{\ugK}{U(\fr{g}_K)}
\newcommand{\hugGnK}{\widehat{\U{g}^{\mb{G}}_{n,K}}}
\newcommand{\huhnK}{\widehat{\U{h}_{n,K}}}
\newcommand{\hutnK}{\widehat{\U{t}_{n,K}}}
\newcommand{\huhn}{\h{\uhn}}
\newcommand{\bscD}{\overline{\hdnl}}
\renewcommand{\O}{\mathcal{O}}
\newcommand{\GK}{\mathrm{GKdim}}
\newcommand{\invlim}{\lim\limits_{\longleftarrow}}
\newcommand{\od}{\overline{\deg}}
\newcommand{\md}[1]{#1-\mathbf{mod}}
\newcommand{\Dhatn}{\widehat{\widetilde{\D}_n}}
\newcommand{\injdim}{\mathrm{inj.dim}}
\newcommand{\lmod}[1]{\mathrm{Pmod}(#1)}
\newcommand{\rmod}[1]{\mathrm{Pmod-} #1}
\newcommand{\fr}[1]{\mathfrak{{#1}}}
\newcommand{\frk}[1]{\mathfrak{{#1}}_k}
\newcommand{\mb}[1]{\mathbf{{#1}}}
\newcommand{\U}[1]{U(\fr{#1})}
\newcommand{\Un}{\U{g}_n}
\newcommand{\Sy}[1]{S(\fr{#1})}
\newcommand{\UG}{(\Ui{g}{G})_n}
\newcommand{\Uk}[1]{U(\fr{#1}_k)}
\newcommand{\Syk}[1]{S(\fr{#1}_k)}
\newcommand{\Ui}[2]{\U{#1}^{\mb{#2}}}
\newcommand{\Si}[2]{\Sy{#1}^{\mb{#2}}}
\newcommand{\Uik}[2]{\Uk{#1}^{\mb{#2}_k}}
\newcommand{\Sik}[2]{\Syk{#1}^{\mb{#2}_k}}
\newcommand{\Env}[1]{R_\omega[{#1}]}
\newcommand{\AEn}[1]{K_\omega\langle{#1}\rangle_n}
\newcommand{\tocong}{\stackrel{\cong}{\longrightarrow}}
\newcommand{\st}{\mid}
\newcommand{\be}{\begin{enumerate}[{(}a{)}]}
\newcommand{\ee}{\end{enumerate}}
\let\le=\leqslant  \let\leq=\leqslant
\let\ge=\geqslant  \let\geq=\geqslant\title{On irreducible representations of compact $p$-adic analytic groups}
\author{Konstantin Ardakov}
\address{Queen Mary University of London, London E1 4NS}
\author{Simon Wadsley}
\address{Homerton College, Cambridge, CB2 8PQ}
\begin{abstract} We prove that the canonical dimension of a coadmissible representation of a semisimple $p$-adic Lie group in a $p$-adic Banach space is either zero or at least half the dimension of a non-zero coadjoint orbit. To do this we establish analogues for $p$-adically completed enveloping algebras of Bernstein's inequality for modules over Weyl algebras, the Beilinson-Bernstein localisation theorem and Quillen's Lemma about the endomorphism ring of a simple module over an enveloping algebra.
\end{abstract}
\setcounter{tocdepth}{1}
\maketitle
\tableofcontents

\section{Introduction}

\subsection{Coadmissible $KG$-modules}
Let $p$ be a prime, let $G$ be a compact $p$-adic analytic group, and let $K$ be a finite extension of $\Qp$. Continuous representations of $G$ in $K$-Banach spaces are of interest in many parts of modern arithmetic geometry and the Langlands programme. Following Schneider and Teitelbaum \cite{ST}, we only consider the \emph{coadmissible} representations of $G$, which by definition are finitely generated modules over the completed group ring $KG$ of $G$ with coefficients in $K$, defined below in $\S \ref{CongKern}$. These group rings, under the name of Iwasawa algebras, play a central role in non-commutative Iwasawa theory --- see, for example, \cite{CFKSV} for more details.

The category $\mathcal{M}$ of coadmissible $KG$-modules is abelian and each $M \in \mathcal{M}$ has a \emph{canonical dimension} $d(M)$, which gives rise to a natural dimension filtration
\[ \mathcal{M} = \mathcal{M}_d \supset \mathcal{M}_{d-1} \supset \cdots \supset \mathcal{M}_1 \supset \mathcal{M}_0\]
by Serre subcategories, where $d = \dim G$ is the dimension of the $p$-adic analytic group $G$ and $M \in \mathcal{M}_i$ if and only if $d(M) \leq i$. For example, $d(M) < d$ if and only if $M$ is a torsion $KG$-module, and $d(M) = 0$ if and only if $M$ is finite dimensional as a $K$-vector space. So in this precise sense, $d(M)$ measures the `size' of the underlying vector space of the representation $M$.

\subsection{Semisimple groups}
The structure of these module categories is the most intricate when the Lie algebra of the group $G$ is semisimple, so we focus on this case. Here is our main result.

\begin{MainThm}\label{main}  Let $G$ be a compact $p$-adic analytic group whose Lie algebra is split semisimple. Let $p$ be a very good prime for $G$ and let $G_{\mathbb{C}}$ be a complex semisimple algebraic group with the same root system as $G$. Let $r$ be half the smallest possible dimension of a non-zero coadjoint $G_{\mathbb{C}}$-orbit. Then any coadmissible $KG$-module $M$ which is infinite dimensional over $K$ satisfies $d(M) \ge r$.
\end{MainThm}

The invariant $r$ depends only on the root system of $G$ and is well-known in representation theory; we recall the exact values that it takes in $\S \ref{MinNilp}$. Thus, coadmissible $KG$-modules are either finite dimensional over $K$ or rather `large'. It is easy to see that in type $A$ the lower bound is attained by a module induced from a closed subgroup. We do not know if the bound is best possible in general.

Theorem \ref{main} was inspired by an analogous result \cite{Smith} of S.P. Smith for the universal enveloping algebra of a complex semisimple Lie algebra. His proof does not adapt to our context because it depends on the fact that the canonical dimension function for enveloping algebras is just the Gelfand-Kirillov dimension and is therefore particularly well-behaved. More precisely: if $M$ is a finitely generated module over the enveloping algebra of such a Lie algebra $\fr{g}$ and $N\subseteq M$ is a finitely generated module over the enveloping algebra of a subalgebra $\fr{h}$ of $\fr{g}$, then $d(N)\leq d(M)$. Gelfand-Kirillov dimension is not available for modules over Iwasawa algebras, and in fact the analogous property for the canonical dimension function fails for both Iwasawa algebras and the completed enveloping algebras of the next section.

\subsection{Completed enveloping algebras}
One way to understand the finer structure of modules over $KG$ is to study the connection between Iwasawa algebras and certain completed enveloping algebras, which are much closer in spirit to objects found in traditional representation theory. This connection was originally discovered by Lazard in his seminal 1965 paper \cite{Laz1965}; we will explain it by means of an example.

Let $R$ be the ring of integers of $K$ and suppose that $K/\Qp$ is unramified for simplicity. Take $G = \ker \left(\SL_2(\Zp) \to \SL_2(\Fp)\right)$ to be the first congruence kernel of $\SL_2(\Zp)$; then the Iwasawa algebra $RG$ can be identified with a non-commutative formal power series ring $R[[F,H,E]]$ in three variables and $KG$ is just $RG \otimes_RK$. Now let $\fr{g}$ be the $R$-Lie algebra $\fr{sl}_2(R) = R f \oplus R h \oplus R e$, and consider the $p$-adic completion
\[ \widehat{\mathcal{U}_K} := \invlim \left(\frac{ U(\fr{g})}{ p^a U(\fr{g})}\right) \otimes_R K \]
of the usual enveloping algebra $U(\fr{g}_K)$ of $\fr{g}_K = \fr{sl}_2(K)$. Lazard observed that it is possible to obtain $\widehat{\mathcal{U}_K}$ as a completion of $KG$ with respect to an intrinsically defined norm; see $\S \ref{MicIwa}$ for more details. The immediate advantage of replacing $KG$ by this completion is that $\widehat{\mathcal{U}_K}$ is a much more accessible object: its topological generators satisfy standard relations such as $[e,f] = h$, whereas the commutation relations between $E$ and $F$ are more intricate.

\subsection{Distribution algebras}
Let $G$ be a uniform pro-$p$ group. Lazard defined a $\Zp$-Lie algebra $L_G$ associated to $G$; this turns out to be free of rank $d = \dim G$ over $\Zp$ and satisfies $[L_G, L_G] \leq pL_G$. Letting $\fr{g} = \frac{1}{p}L_G \otimes_{\Zp} R$, the completed enveloping algebra $\widehat{\mathcal{U}_K}$ of $\fr{g}$ is defined in the same way, and we show in Theorem \ref{MicIwaGrp} that it can also be obtained as a particular algebraic microlocalisation of the Iwasawa algebra $RG$. General theorems now imply that the natural map $KG \to \widehat{\mathcal{U}_K}$ is flat. The main problem with passing to this microlocalisation of $KG$ however is that the map is not \emph{faithfully} flat: there are non-zero $KG$-modules $M$ with $\widehat{\mathcal{U}_K} \otimes_{KG} M = 0$.

In a series of papers including \cite{ST1,ST2,ST,ST3} Schneider and Teitelbaum study a class of rings that they call the \emph{distribution algebras} $D(G,K)$ of a $p$-adic analytic group $G$. From an algebraic viewpoint, these rings can be defined as the projective limit of a sequence of Noetherian algebras $D_w = D_{\sqrt[w]{1/p}}(G,K)$; here $w$ can be any real number $\geq 1$ and $D_w$ is the completion of $KG$ with respect to the degree function $\deg_w$ which is characterised by the property that its value on $p$ is $w$ and its value on each of the standard topological generators of $RG$ is $1$. Schneider and Teitelbaum prove in \cite[Theorem 4.11]{ST} that the natural map $KG \to D(G,K)$ \emph{is} faithfully flat, so $D(G,K)$ is in some sense better than $\widehat{\mathcal{U}_K} = D_1$. Unfortunately, $D(G,K)$ is almost never Noetherian; we finesse this difficulty by never passing to the projective limit $D(G,K)$. As we explain in Theorem \ref{MainTech}, the essense of the proof of \cite[Theorem 4.11]{ST} is that for any given non-zero finitely generated $KG$-module $M$, the base-changed $D_w$-module $D_w\otimes_{KG}M$ is non-zero for sufficiently large $w$.

\subsection{Fr\'echet-Stein algebras}
We focus on the algebras $D_w$ where the parameter $w$ is an integral power of $p$, $w = p^n$ say. The advantage of doing so is that $D_{p^n}$ is more understandable algebraically: it is closely related to a certain crossed product
\[ \widehat{\mathcal{U}_{n,K}}\ast G/G^{p^n},\]
where $\widehat{\mathcal{U}_{n,K}}$ is the subalgebra of $\widehat{\mathcal{U}_K}$ obtained by completing the Iwasawa algebra $KG^{p^n}$ of the open subgroup $G^{p^n}$ of $G$ with respect to its intrinsic norm; see Proposition \ref{CrossProd} and also \cite{Frommer} for more details. We suspect that in fact $D_{p^n}$ is isomorphic to this crossed product, but we will not need to prove this: our Corollary \ref{MainTech} essentially implies that to prove Theorem \ref{main} above it is enough to prove the corresponding result (Theorem \ref{LowerBound}) for each of the algebras $\widehat{\mathcal{U}_{n,K}}$.

In a recent preprint \cite{Schmidt2} Schmidt has studied the projective limit of our $\widehat{\mathcal{U}_{n,K}}$, the so-called Arens--Michael envelope of the enveloping algebra $U(\fr{g}_K)$. Both the distribution algebras $D(G,K)$ and the Arens-Michael envelopes studied by Schmidt are examples of what Schneider and Teitelbaum call Fr\'echet-Stein algebras. Our work on the structure of completed enveloping algebras in this paper may be viewed as an in-depth study of local data for these Fr\'echet-Stein structures.

\subsection{Non-commutative affinoid algebras}\label{NCAA}
It turns out that as a $K$-vector space, $\widehat{\mathcal{U}_K}$ is the set of restricted power series in a free generating set for $\fr{g}$: suppose that $u_1,\ldots,u_d$ is free generating set for $\fr{g}$ as an $R$-module; then
\[ \widehat{\mathcal{U}_K} = \left\{ \sum_{\mathbf{\alpha}\in\mathbb{N}^d} \lambda_{\alpha}\mathbf{u}^{\alpha}  : \quad \lambda_{\alpha} \in K \quad\mbox{and}\quad \lambda_{\alpha} \longrightarrow 0 \quad\mbox{as}\quad |\alpha| \to \infty\right\}\]
where $\mathbf{u}^{\alpha}$ denotes the product $u_1^{\alpha_1}\cdots u_d^{\alpha_d}$.  In this way $\widehat{\mathcal{U}_K}$ can be identified as a $K$-Banach space with the \emph{Tate algebra} $K \langle u_1,\ldots u_d\rangle$ in $d$ commuting variables. Tate algebras form the basis of rigid, or non-archimedean, analysis; one views this particular Tate algebra as the algebra of rigid analytic functions on the unit ball $\fr{g}^\ast = \Hom_R(\fr{g}, R)$ of the $K$-vector space $\fr{g}_K^\ast=\Hom_K(\fr{g}_K, K)$. Thus we view $\widehat{\mathcal{U}_K}$ as a rigid analytic quantization of $\fr{g}^\ast$, analogous to the usual way of viewing $U(\fr{g}_K)$ as an algebraic quantization of $\fr{g}^\ast_K$.

The subalgebras $\widehat{\mathcal{U}_{n,K}}$ of $\widehat{\mathcal{U}_K}$ mentioned above have the following form:
 \[ \widehat{\mathcal{U}_{n,K}} = \left\{ \sum_{\mathbf{\alpha}\in\mathbb{N}^d} \lambda_{\alpha}\mathbf{u}^\alpha  : \quad \lambda_{\alpha} \in K \quad\mbox{and}\quad p^{-n|\alpha|}\lambda_{\alpha} \longrightarrow 0 \quad\mbox{as}\quad |\alpha| \to \infty\right\}\] and can therefore be identified with the Tate algebras $K \langle p^nu_1, \ldots, p^nu_d \rangle$ as $K$-Banach spaces. We may therefore view them as rigid analytic quantizations of larger and larger closed balls $\fr{g}^\ast \subset p^{-1} \fr{g}^\ast \subset p^{-2} \fr{g}^\ast \subset \cdots$ in $\fr{g}_K^\ast$ with respect to the $p$-adic topology. Thus the collection $\{\h{\mathcal{U}_{n,K}} : n \in \mathbb{N}\}$ together forms a quantization of the rigid analytification $(\fr{g}^\ast_K)^{an}$ of the $K$-variety $\fr{g}^\ast_K$.

The gluing of these closed balls should correspond algebraically to taking the projective limit of the affinoid algebras and thus we recover the Arens-Michael envelopes of $U(\fr{g}_K)$ studied by Schmidt.

\subsection{Tate-Weyl algebras}\label{RigAnQ}
It is well-known in non-commutative algebra that the Weyl algebras $\mathbb{C}[x_1,\ldots, x_m, \partial_1, \ldots, \partial_m]$ are more tractable objects than universal enveloping algebras $U(\fr{g}_{\mathbb{C}})$, and frequently influence their structure. In our setting we can form the standard $p$-adic completion
\[\h{\mathcal{A}_K} = \left\{ \sum_{\alpha,\beta \in \mathbb{N}^m} \lambda_{\alpha\beta} x^\alpha \partial^\beta  : \quad \lambda_{\alpha\beta} \in K \quad\mbox{and}\quad \lambda_{\alpha\beta} \longrightarrow 0 \quad\mbox{as}\quad |\alpha| + |\beta| \to \infty\right\}\]
of the Weyl algebra $K[x_1,\ldots, x_m, \partial_1,\ldots, \partial_m]$ with coefficients in $K$, and consider analogous subalgebras
\[\h{\mathcal{A}_{n,K}} = \left\{ \sum_{\alpha,\beta \in \mathbb{N}^m} \lambda_{\alpha\beta} x^\alpha \partial^\beta  \in \h{\mathcal{A}_K}: \quad p^{-n|\beta|}\lambda_{\alpha\beta} \longrightarrow 0 \quad\mbox{as}\quad |\alpha| + |\beta| \to \infty\right\}\]
enjoying certain stronger convergence properties. We tentatively call these the \emph{Tate-Weyl algebras} and view $\h{\mathcal{A}_{n,K}}$ as a quantization of the rational domain
\[Y_n := \left\{ (\xi_1,\ldots,\xi_m, \zeta_1,\ldots,\zeta_m) \in K^{2m} : |\xi_i|\leq 1\quad\mbox{and}\quad |\zeta_i| \leq p^n \quad\mbox{for all} \quad i\right\}.\]
More generally, let $X$ be a smooth $R$-scheme locally of finite type with generic fibre $X_K$, and let $X_0 \subseteq X_K^{an}$ be the unit ball inside the rigid analytification $X_K^{an}$ of $X_K$. Let $(T^\ast X)_0 \subseteq (T^\ast X_K)^{an}$ be the corresponding cotangent bundles and consider the rigid analytic variety
\[ Y_n := X_0 \times_{X_K^{an}} p^{-n} (T^\ast X)_0.\]
By patching together appropriate microlocalisations of Tate-Weyl algebras, we obtain a completed deformation $\widehat{\D_{n,K}}$ of the sheaf $\D$ of crystalline differential operators on $X$; this is a sheaf on $X$ which is now only supported on the special fibre $X_k$. Just as in $\S \ref{NCAA}$ above we view $\widehat{\D_{n,K}}$ as a quantization of the rigid analytic variety $Y_n$ and then the collection $\{\widehat{\D_{n,K}} : n \in \mathbb{N}\}$ together forms a quantization of $Y := X_0 \times_{X_K^{an}} (T^\ast X_K)^{an} = \cup Y_n$; see \S\ref{Defs} for more details. This is broadly analogous to the usual way of viewing the sheaf of differential operators $\D_K$ on $X_K$ as an algebraic quantization of $T^\ast X_K$.

\subsection{Characteristic varieties}
It is possible to associate to each coherent sheaf $\M$ of $\widehat{\D_{n,K}}$-modules a characteristic variety $\Ch(\M)$. Because of the nature of these rigid analytic quantizations we have to be a little careful in the definition of $\Ch(\M)$; since any reasonable filtration of $\h{\D_{n,K}}$ has an associated graded ring in characterstic $p$, $\Ch(\M)$ is effectively forced to be an algebraic subset of the special fibre $T^\ast X_k$ of the cotangent bundle although ideally it should be a subset of $Y$. This causes us a number of problems. For example Gabber's Theorem on the integrability of the characteristic variety  \cite{Gabber} can be used to give another proof of Smith's original theorem, but we cannot use this result because it is heavily dependent on characteristic zero methods and does not directly apply to our completed enveloping algebras.

\subsection{Arithmetic $\D$-modules}\label{BernArithD} The sheaf $\widehat{\D_{0,K}}$, and in particular the Tate-Weyl algebra $\h{\A_K}$, was studied before from a slightly different viewpoint by Berthelot in \cite{Berth}. In fact, $\widehat{\D_{0,K}}$ is essentially Berthelot's sheaf $\widehat{\mathscr{D}}^{(0)}_{\mathcal{X}, \mathbb{Q}}$ of arithmetic differential operators of level zero on the formal neighbourhood $\mathcal{X}$ of the special fibre $X_k$ in $X$; the only difference between our approaches is that we view $\widehat{\D_{n,K}}$ as a sheaf on $X$ supported on $X_k$ for simplicity, and do not mention formal neighbourhoods. Our second main result is a $p$-adic analogue of the classical Bernstein Inequality for algebraic $\D$-modules.

\begin{MainThm}\label{BernIneqIntro} Let $\h{\A_{n,K}}$ be as in $\S\ref{RigAnQ}$ and suppose that $M$ is a finitely generated non-zero $\h{\A_{n,K}}$-module. Then
\[\dim \Ch(M) \geq m.\]
\end{MainThm}

Berthelot has proved a version of the Bernstein Inequality for $F - \mathscr{D}^\dag_{\mathcal{X}, \mathbb{Q}}$-modules in \cite[Th\'eor\`eme 5.3.4]{Berth2}; our Theorem \ref{BernIneqIntro} can be viewed as a generalization of his result in the case when $X = \mathbb{A}^m_R$. We do not consider the arithmetic differential operators of higher level $\widehat{\mathscr{D}}^{(\ell)}_{\mathcal{X}, \mathbb{Q}}$ in this paper; our sheaves $\widehat{\D_{n,K}}$ should be viewed as a deformation of the level zero arithmetic differential operators $\widehat{\mathscr{D}}^{(0)}_{\mathcal{X}, \mathbb{Q}}$.

Close to the end of the preparation of this paper we discovered than Caro has removed the Frobenius condition from Berthelot's result in \cite{Caro}. His proof is rather different to ours. Like Berthelot, Caro does not consider our algebras $\h{\A_{n,K}}$ for $n>0$.

\subsection{Beilinson-Bernstein localisation}
Let $\mb{G}$ be a split semisimple algebraic group over $R$ with $R$-Lie algebra $\fr{g}$ and let $X$ be its flag-scheme $\mb{G} / \mb{B}$. We show in $\S \ref{BB}$ that the analogue of the Beilinson-Bernstein Localisation Theorem \cite{BB} holds in our setting.

 \begin{MainThm}\label{BBIntro} Suppose that $\lambda$ is a dominant regular weight. There is an equivalence of abelian categories between finitely generated $\h{\mathcal{U}_{n,K}}$-modules with central character corresponding to $\lambda$, and coherent sheaves of modules over the sheaf $\h{\D^\lambda_{n,K}}$ of completed deformed twisted crystalline differential operators on $X$. \end{MainThm}

Part of this result is essentially due to Noot-Huyghe in \cite{Noot2}; she established the equivalence of categories between coherent sheaves of $\h{\D^\lambda_{0,K}}$-modules and finitely generated modules for the ring of global sections of $\h{\D^\lambda_{0,K}}$. In many places our proof of this part of Theorem \ref{BBIntro} follows hers although the presentation is sometimes a little different. However she does not explicitly compute the ring of global secitons in her paper. We do so following the ideas in \cite{BMR1}.

We also prove that the rigid analytic quantization construction sketched above in $\S \ref{RigAnQ}$ is compatible with Beilinson-Bernstein localisation; this means that just as in the classical case of complex enveloping algebras one can pull back the characteristic variety of a $\h{\mathcal{U}_{n,K}}$-module from $\fr{g}_k^\ast$ to $T^\ast X_k$ along the Grothendieck-Springer map and study the corresponding $\h{\D^\lambda_{n,K}}$-module instead.

After Theorem \ref{BernIneqIntro}, this effectively reduces the proof of Theorem \ref{main} to the following analogue of Quillen's Lemma \cite{Quillen} for classical enveloping algebras:

\begin{MainThm}\label{QuillenIntro}
Let $M$ be a simple $\h{\mathcal{U}_{n,K}}$-module, let $Z$ be the centre of $\h{\mathcal{U}_{n,K}}$, and suppose that $n > 0$. Then $M$ is $Z$-locally finite.
\end{MainThm}
Theorems \ref{BernIneqIntro} and \ref{QuillenIntro} and the computation of global sections in Theorem \ref{BBIntro} may be viewed as the main technical contributions of this paper. We do not believe that the restriction on $n$ in Theorem \ref{QuillenIntro} is really necessary.
\subsection{Future directions of research}

This work raises a number of possibilities for future avenues of study.  First it suggests that further study of the representation theory of the completed enveloping algebras $\h{\mathcal{U}_{n,K}}$ might be fruitful for better understanding the representation theory of Iwasawa algebras. Although current knowledge suggests that there are very few prime ideals in the Iwasawa algebras $KG$, there are plenty in $\h{\mathcal{U}_{n,K}}$ and a classification of primitive ideals in these latter algebras looks to be possible and may well influence the structure of coadmissible $KG$-modules. Similarly, attempting to define a version of the BGG category $\mathcal{O}$ for completed enveloping algebras is likely to have important consequences if successful.

Although we have only used our techniques to study the canonical dimension function it seems plausible that they might also be useful in attempting to better understand other invariants such as the Euler characteristic of \cite{CoaSuj1999}.

It also seems worth further pursuing the study of $\D$-modules on rigid analytic spaces. In this paper we only really deal with spaces that are locally polydiscs; it would be interesting to attempt to develop a more general theory.

\subsection{Structure of this paper}

In sections 2 and 3 we recall some standard results from non-commutative algebra and define almost commutative affinoid $K$-algebras. In section 4 we develop the theory of crystalline differential operators on a homogeneous space defined over an arbitrary commutative ring; we suspect this is well-known to experts but we could not find a good single reference in this generality so we include it for the sake of those from another field. In sections 5 and 6 we set up the language for and then prove Theorem \ref{BBIntro}. We also explain here how the characteristic variety of a module over a completed enveloping algebra behaves under localisation. In section 7 we prove Theorem \ref{BernIneqIntro}, and in section 8 we prove Theorem \ref{QuillenIntro}. In section 9 we apply all that has gone before along with a study of the fibres of the Grothendieck--Springer resolution to give a lower bound on the canonical dimension of $\h{\mathcal{U}_{n,K}}$-modules which are infinite dimensional over $K$. In section 10 we explain the relationship between the Iwasawa algebras $KG$ and the completed enveloping algebras culminating in the proof of Theorem \ref{main}. Finally in section 11 we study $KG$-modules that are finite dimensional over $K$; we essentially give a complete classification of them. See Prasad's appendix in \cite{ST1} for parallel results for distribution algebras.

\subsection{Acknowledgements}
We would like to thank Ian Grojnowski for giving us the idea of using the Beilinson-Bernstein localisation theorem to study Iwasawa algebras. The first author and Grojnowski have been working, in parallel with this paper, on another geometric approach to Iwasawa algebras. This work, which has been ongoing for several years, has been very influential in the creation of several parts of this paper and should appear soon.

The idea to obtain Theorem \ref{main} as a corollary of the Beilinson-Bernstein localisation theorem and the work of Schneider and Teitelbaum was conceived during the final workshop of the `Non-Abelian Fundamental Groups in Arithmetic Geometry` programme at the Isaac Newton Institute in Cambridge. We are very grateful to the Institute for providing excellent working conditions.

We are also very grateful to Peter Schneider for his detailed comments on an earlier version of this paper. Finally, we heartily thank the referee for reading this paper so thoroughly and providing many useful comments and suggestions.

The first author was partially supported by an Early Career Fellowship from the Leverhulme Trust. The second author was funded by EPSRC grant EP/C527348/1 for a part of this research.

\section{Background}
Our convention regarding left and right modules is as follows. The term module means \emph{left} module, unless explicitly specified otherwise. Noetherian rings are left and right Noetherian, and other ring-theoretic adjectives such as Artinian are used in a similar way.

\subsection{Filtered rings and modules}
Let $\Lambda$ be either $\mathbb{Z}$ or $\mathbb{R}$. A \emph{$\Lambda$-filtration} $F_\bullet A$ on a ring $A$ is a set $\{F_\lambda A|\lambda\in \Lambda\}$ of additive subgroups of $A$ such that
\begin{itemize}
\item $1\in F_0A$;
\item $F_\lambda A\subset F_\mu A$ whenever $\lambda < \mu$,
\item $F_\lambda A\cdot F_\mu A\subset F_{\lambda + \mu}A$ for all $\lambda,\mu \in\Lambda$.
\end{itemize}
The filtration on $A$ is said to be \emph{separated} if $\bigcap_{\lambda \in \Lambda}F_\lambda A=\{0\}$, and it is said to be  \emph{exhaustive} if $\bigcup_{\lambda \in\Lambda}F_\lambda A=A$. Our filtrations will always be exhaustive. Note also that the second condition says that our filtrations are always \emph{increasing}.

 Given a filtration $F_\bullet A$ of $A$ we may make $A$ into a topological ring by letting the $F_\lambda A$ be a fundamental system of neighbourhoods of $0$. When $\Lambda=\mathbb{Z}$ , we say the filtration is \emph{complete} if any Cauchy sequence in $A$ converges to a unique limit.

In a similar way, given a $\Lambda$-filtered ring $F_\bullet A$ and an $A$-module $M$, a \emph{filtration} of $M$ is a set $\{F_\lambda M|\lambda \in\Lambda\}$ of additive subgroups of $M$ such that
\begin{itemize} \item $F_\lambda M\subset F_\mu M$ whenever $\lambda < \mu$, and
\item $F_\lambda A \cdot F_\mu M \subset F_{\lambda + \mu}M$ for all $\lambda,\mu\in \Lambda$.
\end{itemize}
Again, the filtration of $M$ is said to be \emph{separated} if $\bigcap_{\lambda\in\Lambda}F_\lambda M=\{0\}$ and the filtration of $M$ is said to be \emph{exhaustive} if $\bigcup_{\lambda \in \Lambda}F_\lambda M=M$.

\subsection{Degree functions}\label{DegFun}
An exhaustive $\Lambda$-filtration can equivalently be given by a \emph{degree function}. This is a function $\deg : A \to \Lambda \cup \{\infty\}$ such that
\begin{itemize}
\item $\deg(1) = 0$,
\item $\deg(0) = \infty$,
\item $\deg(x + y) \geq \min(\deg(x), \deg(y))$, and
\item $\deg(xy) \geq \deg(x) + \deg(y)$
\end{itemize}
for all $x,y \in A$. If $F_\bullet A$ is a $\Lambda$-filtration on $A$ then
\[\deg(x) :=  \sup \{\lambda \in \Lambda : x \in F_{-\lambda} A\}\]
is a degree function, and conversely, if $\deg : A \to \Lambda \cup \{\infty\}$ is a degree function then
\[F_\lambda A := \{x \in A \st \deg(x) \geq -\lambda\}\]
defines an exhaustive $\Lambda$-filtration on $A$.

Typically when we're dealing with a \emph{positive} $\Lambda$-filtration (one where $F_\lambda A = 0$ for all $\lambda <0$), we will use the language of filtrations, and when we're dealing with a \emph{negative} $\Lambda$-filtration (one where $F_\lambda A = A$ for all $\lambda \geq 0$) we will use the language of degree functions. This explains the minus signs in the above definitions.

\subsection{Associated graded rings and modules} Let $A$ be a $\Lambda$-filtered ring. Define
\[F_{\lambda_-}A := \bigcup_{\mu < \lambda} F_\mu A\]
and note that if $\Lambda = \mathbb{Z}$ then $F_{n_-}A = F_{n-1}A$ for all $n \in \mathbb{Z}$. We can now form two related $\Lambda$-graded rings: the \emph{associated graded ring}
\[ \gr A = \bigoplus_{\lambda \in \Lambda} F_\lambda A/F_{\lambda_-}A\]
and the \emph{Rees ring}
\[ \widetilde{A}= \bigoplus_{\lambda \in \Lambda} F_\lambda A\hspace{1mm} t^\lambda \subseteq A[\Lambda]\]
 which we view as a subring of the group ring $A[\Lambda]$ of the abelian group $\Lambda$; here $A[\Lambda]$ is the free $A$-module on the set of symbols $\{t^\lambda : \lambda \in\Lambda\}$ which is a set-theoretic copy of $\Lambda$. We denote the $\lambda$-th homogeneous piece of $\gr A$ by $\gr_\lambda A$. For $a\in A$ but not in $\bigcap_{\lambda} F_\lambda A$, we'll also write $\gr a$ for the principal symbol of $a$ in $\gr A$.

Given a filtered $F_\bullet A$ module $F_\bullet M$, we similarly define
\[F_{\lambda_-}M := \bigcup_{\mu < \lambda} F_\mu M\]
and then the \emph{associated graded module} $\gr M$ of $M$ is
\[ \gr M = \bigoplus_{\lambda\in \Lambda} F_\lambda M/F_{\lambda_-}M.\]
Clearly $\gr M$ is naturally a graded $\gr A$-module.

We say that a $\mathbb{Z}$-filtration on a ring $A$ is \emph{Zariskian} if the Rees ring $\widetilde{A}$ is Noetherian and $F_{-1}A$ is contained in the Jacobson radical of $F_0A$. In particular, a $\mathbb{Z}$-filtration is Zariskian whenever the filtration on $A$ is complete and the associated graded ring $\gr A$ is Noetherian --- this follows from \cite[Proposition II.2.2.1]{LVO}.

\subsection{Microlocalisation}\label{Micro}

We recall some basic results in the theory of algebraic microlocalisation.

Suppose that $A$ is a Zariskian filtered ring, $T$ is an Ore set in $\gr A$ consisting of homogeneous elements and $M$ is a finitely generated $A$-module with a good filtration.

\begin{lem} Let  $S:=\{ s\in A\mid \gr s\in T\}$. Then
\be
\item $S$ is an Ore set in $A$.
\item There is a natural Zariskian filtration on $A_S$ such that $\gr A_S\cong (\gr A)_T$.
\item There is a good filtration on $M_S$ as an $A_S$-module such that $\gr M_S\cong (\gr M)_T$.
\ee
\end{lem}

\begin{proof}
(a) follows from \cite[Corollary 2.2]{Li}. (b) follows from \cite[Proposition 2.8 and Proposition 2.3]{Li}. (c) follows from \cite[Corollary 2.5(1,2) and Proposition 2.6(1)]{Li}
\end{proof}

\begin{defn} Given the notation above we define the \emph{microlocalisation} $Q_T(A)$ of $A$ at $T$ to be the completion of the induced Zariskian filtration on $A_S$. Similarly, we define the \emph{microlocalisation} $Q_T(M)$ of $M$ at $T$ to be the completion of the induced good filtration on $M_S$.
\end{defn}

\begin{cor} With the notation above $Q_T(A)$ is a flat $A$-module and $Q_T(A)\otimes_A M\cong Q_T(M)$. Moreover $Q_T(M)=0$ if and only if if $M_S=0$.
\end{cor}

\begin{proof}
The flatness follows from \cite[Corollary 2.4 and Corollary 2.7(1)]{Li}. The isomorphism $Q_T(A)\otimes_A M\cong Q_T(M)$ is \cite[Corollary 2.7(2)]{Li}. The last part follows from \cite[Corollary 2.5(3)]{Li}
\end{proof}\subsection{Auslander--Gorenstein rings and dimension functions} \label{AG}

\begin{defn} Let $A$ be a Noetherian ring.
\be
\item We say that a finitely generated (left or right) $A$-module $M$ satisfies \emph{Auslander's condition} if for every $i\geq 0$ and every submodule $N$ of $\Ext^i_A(M,A)$ we have $\Ext^j_A(N,A)=0$ for all $j<i$.
\item We say that $A$ is \emph{Auslander--Gorenstein} if the left and right self-injective dimension of $A$ is finite and every finitely generated (left or right) $A$-module satisfies Auslander's condition.
\item We say that $A$ is \emph{Auslander regular} if it is Auslander-Gorenstein and has finite (left and right) global dimension.
\ee
\end{defn}

\begin{defn} If $A$ is an Auslander--Gorenstein ring, and $M$ is a finitely generated $A$-module then the \emph{grade} of $M$ is given by \[ j_A(M):=\inf\{ j\mid \Ext^j_A(M,A)\neq 0\} \] and the \emph{canonical dimension} of $M$ is given by
\[ d_A(M):=\injdim_A A-j_A(M).\] We say an $A$-module $M$ is \emph{pure} if every finitely generated non-zero submodule of $M$ has the same canonical dimension. We say a finitely generated $A$-module $M$ is \emph{critical} if every proper quotient of $M$ has strictly smaller canonical dimension.
\end{defn}

\begin{defn} Let $A$ be a Noetherian ring. An \emph{exact dimension function} is an assignment, to each finitely generated $A$-module $M$, a value $\delta(M) \in \mathbb{Z} \cup \{-\infty\}$ satisfying the following conditions:
\begin{enumerate}[{(}i{)}]
\item $\delta(0) = -\infty$,
\item $\delta(M) = \max\{\delta(M'), \delta(M'')\}$ whenever $0 \to M' \to M \to M'' \to 0$ is a short exact sequence of finitely generated $A$-modules, and
\item $\delta(M) < \delta(A/P)$ whenever $P$ is a prime ideal of $A$ and $M$ is a torsion $A/P$-module.
\end{enumerate}
If in addition $\delta$ satisfies the following condition:
\begin{enumerate}[{(}iv{)}]
\item for any finitely generated $A$-module $M$, there is an integer $n$ such that whenever $M = M_0 \supseteq M_1 \supseteq M_2 \supseteq \cdots$ is a descending chain of $A$-submodules of $M$, we have $\delta(M_i / M_{i+1}) < \delta(M)$ for all $i \geq n$,
\end{enumerate}
then we say that $\delta$ is \emph{finitely partitive}.
\end{defn} 

\begin{prop} Let $A$ be a ring.
\be
\item If $A$ is Auslander--Gorenstein then $d_A$ is a finitely partitive exact dimension function.
\item If $A$ has a Zariskian filtration such that the associated graded ring $\gr A$ is Auslander--Gorenstein, then $A$ is Auslander--Gorenstein. Moreover if $\gr A$ is Auslander regular then $A$ is Auslander regular.
\ee
\end{prop} 
\begin{proof}

(a)   is \cite[Proposition 4.5]{Lev91} and (b) is \cite[Theorem 3.9]{Bj89}.
\end{proof}

\subsection{Base change}\label{basechange1}

Suppose that $K$ is a field and $K'$ is a finite algebraic field extension of $K$.

\begin{lem} If $A$ is an Auslander--Gorenstein $K$-algebra then $A':=K'\otimes_K A$ is Auslander--Gorenstein. Moreover if $M$ is a finitely generated $A$-module then \[ d_A(M)=d_{A'}(K'\otimes_K M).\] Similarly, if $N$ is a finitely generated $A'$-module then $N$ is a finitely generated $A$-module by restriction and $d_A(N)=d_{A'}(N)$.
\end{lem}
\begin{proof}

First, \cite[$\S$5.4]{ArdBro2007} gives that faithfully flat Frobenius extensions of Auslander--Gorenstein rings are Auslander--Gorenstein with the same self-injective dimension. Next, \cite[Example C]{BellFarn} gives that if $K'$ is a simple algebraic extension of $K$ then $A'$ is a $\id$-Frobenius extension of $A$. Also, \cite[Proposition 1.3]{BellFarn} gives that an $\id$-Frobenius extension of an $\id$-Frobenius extension is an $\id$-Frobenius extension and so the first part follows.

Now, if $M$ is a finitely generated $A$-module, we have isomorphisms \[ \Ext^j_{A'}(K'\otimes_K M,A')\cong K'\otimes_K \Ext^j_A(M,A)\] for each $j\ge 0$. This implies that $j_A(M)=j_{A'}(K'\otimes_K M)$, because $K'$ is faithfully flat over $K$.  Thus we obtain the second part.

Finally, if $N$ is a finitely generated $A'$-module, we have isomorphisms \[\Ext^j_A(N,A)\cong \Ext^j_{A'}(N,A')\] for each $j\ge 0$ and the result follows.
\end{proof}

\begin{prop} Suppose that $A$ and $B$ are Auslander--Gorenstein rings such that $\injdim_A A=\injdim_B B$ and $B$ is a flat $A$-module. If $M$ is a finitely generated $A$-module and  $\Ext^{j_A(M)}_A(M,A)\otimes_A B\neq 0$ then $d_A(M)=d_B(B\otimes_A M)$.

In particular if $B$ is a faithfully flat $A$-module then $d_A(M)=d_B(B\otimes_A M)$ for every finitely generated $A$-module $M$.
\end{prop}

\begin{proof}
It suffices to prove that $j_A(M)=j_B(B\otimes_A M)$ under these conditions. But $\Ext^i_B(B\otimes_A M,B)\cong \Ext^i_A(M,A)\otimes_AB$ for all $i\geq 0$ since $B$ is flat over $A$. Thus $j_A(M)=j_B(B\otimes_A M)$ is equivalent to $\Ext^{j_A(M)}_A(M,A)\otimes_A B\neq 0$ and the first part follows.

The second part is a trivial consequence of the first.
\end{proof}

\subsection{Lattices}\label{Latt}
We will several times require the following very useful result.

\begin{lem} Let $A$ be a Noetherian ring and let $\pi \in A$ be a central element. Suppose $A$ is $\pi$-adically complete. Let $M$ be a finitely generated $A$-module such that $M/\pi M$ has finite length. Let $M_0 \supseteq M_1 \supseteq M_2 \supseteq \cdots $ be a descending chain of $A$-submodules of $M$ such that $M_n \nsubseteq \pi M$ for all $n$. Then $\bigcap M_n \nsubseteq \pi M$.
\end{lem}
\begin{proof} First, note that the $\pi$-adic filtration on $M$ and on each $M / M_n$ is complete and separated by \cite[\S 3.2.3(v)]{Berth}, so each $M_n$ is closed in the topology defined by the $\pi$-adic filtration. Let $a \geq 1$ be an integer; then $M / \pi^a M$ is a finite extension of quotients of the finite length module $M/ \pi M$ and so has finite length. It follows that $M$ is \emph{pseudo-compact} in the sense of \cite[\S IV.3]{Gabriel}. Now we may apply \cite[Proposition IV.3.11]{Gabriel} to deduce that $\pi M + \bigcap M_n = \bigcap_n(\pi M + M_n)$. But $M / \pi M$ has finite length, so there exists $r$ such that $\pi M + M_n = \pi M + M_r$ for all $n \geq r$. Hence $\pi M + \bigcap M_n = \pi M + M_r > \pi M$ and the result follows.
\end{proof}

Let $R$ be a discrete valuation ring with uniformizer $\pi$, residue field $k$ and field of fractions $K$.

\begin{defn}
Let $V$ be a $K$-vector space. We say that an $R$-submodule $L$ of $V$ is a \emph{$R$-lattice} if $V = K \cdot L$ and $\bigcap_{a = 0}^\infty \pi^a L = 0$.
\end{defn}
Equivalently, the \emph{$\pi$-adic filtration} on $V$ given by $F_iV = \pi^{-i} L$ is exhaustive and separated. We call the $k$-vector space $\gr_0 V = L / \pi L$ the \emph{slice} of $V$. Since our vector spaces will frequently be infinite dimensional over $K$, in general the lattice $L$ will not be finitely generated as an $R$-module.

\begin{prop} Suppose that the discrete valuation ring $R$ is complete. Then every $R$-lattice $N$ in a finite dimensional $K$-vector space $V$ is finitely generated over $R$.
\end{prop}
\begin{proof} Let $M$ be the $R$-submodule of $V$ generated by a basis of $V$. Then $M$ is finitely generated over $R$ and $M / \pi M$ has finite length since $R/\pi R = k$ is a field. Consider the descending chain of $R$-submodules $M_i = M \cap \pi^i N$ of $M$ for $i\ge 0$. Since $N$ is an $R$-lattice in $V$, $\bigcap M_i = 0$ so $M_t \subseteq \pi M$ for some $t\geq 0$ by the Lemma. Thus $N \cap \pi^{-i} M \subseteq N\cap \pi^{-i + 1}M$ for all $i \geq t$, whence inductively $N \cap \pi^{-i}M \subseteq N\cap \pi^{-t + 1}M$ for all $i\geq t$. Since $M$ is an $R$-lattice in $V$,
\[N = \bigcup_{i\geq t} N\cap \pi^{-i} M \subseteq \pi^{-t+1}M\]
so $N$ is finitely generated $R$-module because $\pi^{-t+1}M$ is a Noetherian $R$-module.
\end{proof}

\section{Almost commutative affinoid algebras}

We develop the basic theory of almost commutative affinoid algebras in this section. Unless explicitly stated otherwise, $R$ will denote a complete discrete valuation ring with uniformizer $\pi$, residue field $k$ and field of fractions $K$. We do not make any assumptions on the characteristics of $k$ or $K$.

\subsection{Doubly filtered rings}\label{DoubleFilt}

\begin{defn} Let $A$ be a $K$-algebra. We say that $A$ is \emph{doubly filtered} if it has an $R$-subalgebra $F_0A$ which is an $R$-lattice in $A$ and if the slice $\gr_0 A$ of $A$ is a $\mathbb{Z}$-filtered ring. We say that $A$ is a \emph{complete doubly filtered $K$-algebra} if $F_0A$ is complete with respect to its $\pi$-adic filtration, and the filtration on $\gr_0 A$ is also complete. A \emph{morphism} of doubly filtered $K$-algebras is a $K$-linear ring homomorphism $\varphi : A \to B$ which preserves the lattices in $A$ and $B$ and which induces a filtered $k$-linear homomorphism $\gr_0 \varphi : \gr_0 A \to \gr_0 B$ between the slices.
\end{defn}

\begin{lem} Let $A$ be a doubly filtered $K$-algebra. Then the associated graded ring of $A$ with respect to its $\pi$-adic filtration is isomorphic to the Laurent polynomial ring in one variable over the slice $\gr_0 A$ of $A$:
\[\gr A \cong (\gr_0 A)[s, s^{-1}].\]
\end{lem}

\noindent We will always denote the associated graded ring of the slice of $A$ by
\[\Gr A := \gr (\gr_0 A).\]
Note that $A \mapsto \Gr(A)$ is a functor from the category of doubly filtered $K$-algebras to the category of graded $k$-algebras.

\subsection{Good double filtrations}\label{GoodDouble}
Let $A$ be a doubly filtered $K$-algebra, and let $M$ be an $A$-module. A \emph{double filtration} on $M$ consists of an $R$-lattice $F_0M$ in $M$ which is an $F_0A$-submodule, and a $\mathbb{Z}$-filtration $F_\bullet \gr_0M$ on $\gr_0 M$ compatible with the filtration on $\gr_0A$. We call
\[\Gr(M) := \gr(\gr_0 M)\]
the \emph{associated graded module} of $M$ with respect to this double filtration. The double filtration on $M$ is said to be \emph{good} if
\begin{itemize}
\item the filtration on $\gr_0 M$ is separated, and
\item $\Gr(M)$ is a finitely generated $\Gr(A)$-module.
\end{itemize}
When $A$ is a complete doubly filtered $K$-algebra such that $\Gr(A)$ is Noetherian, it follows from \cite[Theorem I.5.7]{LVO} that this is equivalent to the filtration on $\gr_0 M$ being good in the sense of \cite[\S I.5.1]{LVO}: the Rees module of $\gr_0 M$ is finitely generated over the Rees ring of $\gr_0 A$. The following elementary result will be very useful in the future.

\begin{lem} Let $A$ be complete doubly filtered $K$-algebra and let $M$ be a doubly filtered $A$-module.
\be
\item If the double filtration on $M$ is good, then $M$ is a finitely generated $A$-module.
\item If $\Gr(A)$ is Noetherian then so are $A$ and $F_0A$.
\ee \end{lem}
\begin{proof} (a) The $\pi$-adic filtration $F_iM := \pi^{-i}F_0M$ on $M$ is separated because $F_0M$ is a lattice in $M$, and the given filtration $F_\bullet \gr_0M$ on $\gr_0M$ is separated by assumption. In view of Lemma \ref{DoubleFilt}, the result now follows by applying \cite[Theorem I.5.7]{LVO} twice.

(b) It is enough to show that $A$ is left Noetherian. Let $I$ be a left ideal of $A$. The double filtration on $A$ induces a double filtration on $I$; in this way, $\gr_0 I$ is a left ideal in $\gr_0 A$, the filtration on $\gr_0 I$ is separated and $\Gr(I)$ is a left ideal of $\Gr(A)$. Hence the double filtration on $I$ is good so $I$ is finitely generated. A similar argument shows that $F_0A$ is also Noetherian.
\end{proof}

Whenever $L$ is an $R$-module, $L_K$ will denote the $K$-vector space $K \otimes_R L$.

\begin{prop} Let $A$ be a complete doubly filtered $K$-algebra such that $\Gr(A)$ is Noetherian.
\be \item Every finitely generated $\pi$-torsion-free $F_0A$-module $L$ is an $R$-lattice in $L_K$.
\item Every finitely generated $A$-module $M$ has at least one good double filtration.
\ee \end{prop}
\begin{proof} Note that $\Gr(A)$ Noetherian implies that $F_0A$ is Noetherian, by part (b) of the Lemma.

(a) $N := \bigcap_{j=0}^\infty \pi^j L$ is a finitely generated $F_0A$-submodule of $L$ since $F_0A$ is Noetherian. Moreover $N = \pi N$. Because $F_0A$ is $\pi$-adically complete, $\pi$ is in the Jacobson radical of $F_0A$ and hence $N = 0$ by Nakayama's Lemma.  Since $L$ is $\pi$-torsion-free, we can identify it with its image inside $L_K$ and hence $L$ is an $R$-lattice in $L_K$.

(b) Let $m_1,\ldots,m_\ell$ be an $A$-generating set for $M$ and let $F_0M = \sum_{i=1}^\ell F_0A.m_i$. Then $F_0M$ is an $R$-lattice in $M$ by (a) and $\gr_0M = F_0M / \pi F_0M = \sum_{i=1}^\ell \gr_0 A.\bar{m_i}$ is a finitely generated $\gr_0 A$-module. Now setting $F_j \gr_0 M := \sum_{i=1}^\ell F_j \gr_0 A .\bar{m_i}$ defines a filtration $F_\bullet \gr_0 M$ on $\gr_0 M$ such that $\gr(\gr_0 M)$ is finitely generated over $\Gr(A)$. The filtration on $\gr_0A$ is complete and $\Gr(A)$ is Noetherian by assumption, so $F_\bullet \gr_0 M$ is separated by \cite[Proposition II.2.2.1 and Theorem I.4.14]{LVO}.
\end{proof}

\subsection{Characteristic varieties}\label{CharVarAff}
\begin{defn} Let $A$ be a complete doubly filtered $K$-algebra such that $\Gr(A)$ is commutative and Noetherian, and let $M$ be a finitely generated $A$-module. Choose a good double filtration $(F_0M, F_\bullet \gr_0 M)$ on $M$. The \emph{characteristic variety} of $M$ is the Zariski closed subset
\[\Ch(M) := \Supp(\Gr(M)) \subseteq \Spec(\Gr(A))\]
of the prime spectrum of the commutative Noetherian $k$-algebra $\Gr(A)$.
\end{defn}

Of course the dimension of $\Ch(M)$ will always be equal to the Krull dimension of the $\Gr(A)$-module $\Gr(M)$. 

\begin{prop} The characteristic variety $\Ch(M)$ does not depend on the choice of good double filtration on $M$. Moreover if $0\to L\to M\to N\to 0$ is a short exact sequence of $A$-modules then $\Ch(M) =\Ch(L)\cup \Ch(N).$
\end{prop}
\begin{proof} Let us first fix the lattice $F_0M$ in $M$. Then it is well-known \cite[Chapter III, Lemma 4.1.9]{LVO} that the characteristic support  $\Supp(\gr \gr_0 M)$ does not depend on the choice of good filtration on $\gr_0 M$, and also that $\Supp(\gr \gr_0 M)$ only depends on the class of $\gr_0 M$ in the Grothendieck semigroup of $\gr_0 A$-modules \cite[Lemma D.3.3]{HTT}. On the other hand this class does not depend on the choice of lattice $F_0 M$ inside $M$ by \cite[Proposition 1.1.2]{Ginz} for example.

Now given a short exact sequence as in the statement, it is straightforward to find $F_0A$ lattices $F_0L$ and $F_0N$ in $L$ and $N$ respectively so that there is a short exact sequence $0\to F_0L\to F_0 M\to F_0N\to 0$. Reducing this mod $\pi$ gives a short exact sequence $0\to \gr_0 L\to \gr_0 M\to \gr_0 N\to 0$ since $F_0 N$ is flat over $R$. 

Now giving $\gr_0 L$ and $\gr_0 N$ the subspace and quotient filtrations from some good filtration on $\gr_0 M$ defines good double filtrations on $L$ and $N$ such that $0\to \Gr(L)\to \Gr(M)\to \Gr(N)\to 0$ is exact and then the result is clear.
\end{proof}

\begin{thm} Suppose that $A$ is a complete doubly filtered $K$-algebra such that $\Gr(A)$ is commutative and regular; then $A$ is Auslander regular. If in addition every simple $\Gr(A)$-module $N$ has $d_{\Gr(A)}(N)=0$ then
\[\dim \Ch(M)+j_A(M)=\dim \Gr(A)\]
for all finitely generated $A$-modules $M$. 

Moreover if $M$ is a pure $A$-module then every irreducible component of $\Ch(M)$ has the same dimension. 
\end{thm}
\begin{proof}
The first part follows by applying Proposition \ref{AG}(b) twice. Now
\[\dim \Ch(M) + j_{\Gr(A)}(\Gr(M)) = \dim \Gr(A)\]
by \cite[Theorem 1.3]{SZ} and $j_{\Gr(A)}(\Gr(M))=j_A(M)$ by \cite[Theorem III.2.5.2]{LVO}. 

Finally if $M$ is pure then by applying \cite[Theorem 3.8]{Bj89} twice we may find a good double filtration of $M$ such that $\Gr(M)$ is pure and the result follows.
\end{proof}

Notice in particular that the second part of the theorem applies whenever $\Gr(A)$ is a polynomial ring over $k$. 

\subsection{Almost commutative algebras}\label{ExAC}
We now generalise the more-or-less standard theory of almost commutative algebras over a field, see \cite[\S 8.4]{MCR}. Let $R$ be a commutative Noetherian base ring.
\begin{defn}
Let $A$ be a positively $\mathbb{Z}$-filtered $R$-algebra with $F_0A$ an $R$-subalgebra of $A$. We say that $A$ is \emph{almost commutative} if $\gr A$ is a finitely generated commutative $R$-algebra.
A \emph{morphism} of almost commutative $R$-algebras is an $R$-linear filtered ring homomorphism.
\end{defn}

It follows from \cite[Proposition II.2.2.1]{LVO} that almost commutative $R$-algebras are always Noetherian. Moreover every factor ring of an almost commutative ring is again almost commutative.

\begin{examps}
\be
\item Let $X$ be an affine $R$-scheme of finite type. Then there exists a presentation $R[X_1,\ldots, X_m] \twoheadrightarrow \O(X)$ which endows $\O(X)$ with a positive filtration induced from the natural degree filtration on $R[X_1,\ldots, X_m]$. This gives $\O(X)$ the structure of an almost commutative $R$-algebra.

\item Let $X$ be a smooth affine $R$-scheme of finite type. Then the ring of crystalline differential operators $\D(X)$ is generated by $\O(X)$ and $\T(X)$ --- see $\S\ref{cryst}$ for a precise definition. The associated graded ring with respect to the filtration by order of differential operators
\[\gr \D(X) \cong \Sym_{\O(X)} \T(X) \cong \O(\T^\ast (X))\]
is commutative and $\T(X)$ is a finitely generated $\O(X)$-module. Therefore $\D(X)$ is an almost commutative $R$-algebra.

\item If $\fr{g}$ is an $R$-Lie algebra which is free of finite rank as an $R$-module, then the universal enveloping algebra $U(\fr{g})$ is an almost commutative $R$-algebra with respect to the usual Poincar\'e-Birkhoff-Witt-filtration.

\item \label{defTateWeyl} Let $V$ be a free $R$-module of finite rank equipped with an alternating $R$-bilinear form $\omega\colon V\times V\to R$. The \emph{enveloping algebra} $R_\omega[V]$ of $(V,\omega)$ is the quotient of the tensor $R$-algebra $\bigoplus_{i=0}^\infty V^{\otimes i}$ on $V$ by the relations
\[ vw - wv = \omega(v,w) \quad\mbox{for all}\quad v,w \in V.\]
Consider the Lie algebra $\fr{h}_\omega := V \oplus Rz$ with Lie bracket determined by the rules $[v,w] = \omega(v,w)z$ and $[v,z] = 0$ for all $v,w \in \fr{h}_\omega$. Then $R_\omega[V]$ is isomorphic to the factor ring of $U(\fr{h}_\omega)$ by the ideal generated by $z - 1$. Since $\gr U(\fr{h}_\omega)$ is isomorphic to the polynomial algebra $\Sym_R(V)[z]$ by the Poincar\'e-Birkhoff-Witt Theorem, we see that $\gr R_\omega[V] \cong \Sym_R(V)$ when we equip $R_\omega[V]$ with the natural positive filtration with $R$ in degree $0$ and $V$ in degree $1$. Thus $R_\omega[V]$ is an almost commutative $R$-algebra.
\ee
\end{examps}

\subsection{Deformations}\label{Defs}
We now return to assuming that $R$ is a complete discrete valuation ring with field of fractions $K$ and uniformizer $\pi$.
\begin{defn} Let $A$ be a positively $\mathbb{Z}$-filtered $R$-algebra with $F_0A$ an $R$-subalgebra of $A$. We call $A$ a \emph{deformable $R$-algebra} if $\gr A$ is a flat $R$-module.
A \emph{morphism} of deformable $R$-algebras is an $R$-linear filtered ring homomorphism.
\end{defn}

All the almost commutative $R$-algebras appearing in the above examples, with the possible exception of (a), are deformable.

\begin{defn} Let $A$ be a deformable $R$-algebra and let $n$ be a non-negative integer. The \emph{$n$-th deformation} of $A$ is the following $R$-submodule of $A$:
\[A_n := \sum_{i\geq 0} \pi^{in} F_iA.\]
\end{defn}
There is a unique ring homomorphism $A[t] \to A_K$ which extends the inclusion $A \hookrightarrow A_K$ and sends $t$ to $\pi^n$. Since $A_n$ is the image of the Rees ring $\Ex{A} = \bigoplus_{m=0}^\infty  t^m F_mA$ under this homomorphism, $A_n$ is actually an $R$-subalgebra of $A$. It can in fact be shown that $A_n$ is isomorphic to the factor ring $\Ex{A}/\langle t - \pi^n\rangle.$

This definition is clearly functorial and in this way we obtain an endofunctor $A \mapsto A_n$ of the category of deformable $R$-algebras. Of course the deformation $A_n$ does not depend on the choice of uniformizer $\pi$.

\begin{lem} Let $A$ be a deformable $R$-algebra. Then $A_n$ is also a deformable $R$-algebra for all $n \geq 0$ and there is a natural isomorphism $\gr A \to \gr A_n$.
\end{lem}
\begin{proof} It is clear that $A_n$ is an $R$-lattice in $A_K$. Give $A_n$ the subspace filtration $F_mA_n := F_mA \cap A_n$. Because $\gr A$ is flat over $R$, we have
\[F_mA_n = \sum_{i=0}^m \pi^{in} F_iA.\]
Define an $R$-linear map $F_m A / F_{m-1} A \to F_m A_n / F_{m-1}A_n$ by the formula
\[x + F_{m-1}A \mapsto \pi^{mn} x + F_{m-1}A_n.\]
Since $\gr A$ is flat over $R$, this map is an injection. It is straightforward to verify that it is actually a bijection and that it extends to a ring isomorphism between $\gr A$ and $\gr A_n$.

Finally if $f\colon A\to B$ is a morphism of deformable $R$-algebras the diagram  \[ \xymatrix{ \gr A  \ar[d]_{\gr f} \ar[r]^\cong & \gr A_n \ar[d]^{\gr f_n} \\ \gr B\ar[r]^\cong & \gr B_n} \] commutes. 
\end{proof}

\subsection{Deformations and tensor products}\label{DefTens}
Equipping the polynomial algebra $R[x_1,\ldots, x_\ell]$ with the natural degree filtration gives a first example of a deformable $R$-algebra. It is easy to see that its $n$-th deformation is simply $R[\pi^n x_1, \ldots, \pi^n x_\ell]$.

\begin{prop} Let $A$ be a deformable $R$-algebra. Then $A\otimes_R R[x_1,\ldots,x_\ell]$ is a deformable $R$-algebra when equipped with the tensor filtration. Moreover \[ (A\otimes_R R[x_1,\ldots,x_\ell])_n=A_n\otimes (R[x_1,\ldots,x_\ell])_n.\] 
\end{prop}
\begin{proof} By induction on the number of variables it suffices to consider the case $\ell=1$ and $x_1=x$. We may then identify $A\otimes_R R[x]$ with $A[x]$.

Now $\gr (A[x])\cong (\gr A)[x]$ where $x$ has degree one in the right hand side. Thus $\gr (A[x])$ is flat over $R$ and so deformable. 

By functoriality of $(-)_n$, the natural inclusions $A\to A[x]$ and $R[x]\to A[x]$ of deformable $R$-algebras induce inclusions $A_n\to A[x]_n$ and $R[x]_n\to A[x]_n$. Thus the universal property of the tensor product yields a map from $A_n\otimes_R R[x]_n$ to $A[x]_n$. By considering the identification $A\otimes_R R[x]$ with $A[x]$ we see that this map is an inclusion and so it suffices to see that that it is surjective. But if $a_i\in F_iA$ then $\pi^{in}a_i\otimes \pi^{jn}x^j$ is an element of $A_n\otimes_R R[x]_n$ and the images of these elements span $A[x]_n$.
\end{proof}

Peter Schneider has sent us a proof that the functors $A\mapsto A_n$ preserve finite coproducts --- that is tensor products over $R$ --- in the category of deformable $R$-algebras. However, we will not need the full strength of this result in this work. 

\subsection{$\pi$-adic completions}\label{GrComp}
We will now exhibit a functorial way of producing complete doubly filtered $K$-algebras.

\begin{defn} Let $A$ be a deformable $R$-algebra. The \emph{$\pi$-adic completion} of $A$ is
\[ \h{A} = \invlim A/\pi^a A.\]
This is an $R$-lattice in the $K$-algebra
\[\hK{A} := \widehat{A}\otimes_RK.\]
\end{defn}

\begin{lem} Let $A$ be a deformable $R$-algebra. Then $\hK{A}$ is a complete doubly filtered $K$-algebra, and is a natural isomorphism
\[\Gr(\hK{A}) = \gr(\h{A} / \pi \h{A}) \cong \gr A / \pi \gr A.\]
\end{lem}
\begin{proof} We see that $\h{A}$ is an $R$-lattice in $\hK{A}$ and that $\gr_0 \hK{A} \cong A / \pi A$. The filtration on $A$ induces a filtration on $A/\pi A$ and
\[ \Gr(\hK{A}) = \gr \gr_0 \hK{A} \cong \gr (A/\pi A) \cong \gr A / \pi \gr A\]
since $\pi$ is a central regular element in $A$ of degree zero. The filtration on $A/\pi A$ is complete because the filtration on $A$ is positive by assumption.\end{proof}
Thus we have a countable family of functors $A \mapsto \hnK{A}$ from deformable $R$-algebras to complete doubly filtered $K$-algebras.

\begin{cor} For each $n\geq 0$  there is a natural isomorphism
\[\Gr(\hnK{A}) \to \gr A / \pi \gr A.\]
\end{cor}
\begin{proof} Apply Lemmas \ref{Defs} and \ref{GrComp}.
\end{proof}

\subsection{Almost commutative affinoid algebras}\label{acsa}
\begin{defn} Let $A$ be a complete doubly filtered $K$-algebra. We say that $A$ is an \emph{almost commutative affinoid $K$-algebra} if its slice is an almost commutative $k$-algebra.
\end{defn}

Almost commutative affinoid $K$-algebras are always Noetherian by \cite[Proposition II.2.2.1]{LVO}.

\begin{prop} Let $A$ be an almost commutative deformable $R$-algebra. Then $\hnK{A}$ is an almost commutative affinoid $K$-algebra for all $n\geq 0$.
\end{prop}
\begin{proof} $\hnK{A}$ is a complete doubly filtered $K$-algebra by Lemma \ref{GrComp} and the filtration on its slice $\gr_0 \hnK{A}$ is positive by construction. Since $\Gr(\hnK{A}) \cong \gr A / \pi \gr A$ by Corollary \ref{GrComp}, $\Gr(\hnK{A})$ is a finitely generated commutative $k$-algebra and hence $\gr_0 A$ is an almost commutative $k$-algebra.
\end{proof}
We thus obtain a family of almost commutative affinoid $K$-algebras $\hnK{A}$, whenever we have an almost commutative deformable $R$-algebra $A$; see $\S \ref{ExAC}$ for a list of examples. Note that in Example \ref{ExAC}(a), the completion $\widehat{\O(X)_{n,K}}$ of the commutative $R$-algebra $\O(X)$ is an affinoid $K$-algebra in the sense of \cite{BGR}, and can be viewed as the ring of rigid analytic functions on an affinoid variety $X_{n,K}$. This justifies our terminology.

We note in passing that Soibelman's quantum affinoid algebras $K\{T\}_{q,r}$ appearing in \cite{Soi1} and \cite{Soi2} are examples of almost commutative affinoid $K$-algebras not of the form $\widehat{A_{n,K}}$ for some almost commutative $R$-algebra $A$, provided that $|q - 1| < 1$.

\subsection{Base change}\label{basechange2}
Let $v$ be the normalised discrete valuation on $K$, and let $K'$ be a field extension of $K$ which is complete with respect to a normalised discrete valuation $v'$. Recall that $K'/K$ is said to be \emph{finitely ramified} if $v'|_K=ev$ for some integer $e>0$; in this case $e := v'(\pi)$ is called the \emph{ramification index} of $K'/K$. Let $R'$ be the valuation ring of $v'$ and let $\pi'$ be a uniformizer of $R'$; then the ideal $\pi R'$ is generated by $\pi'^e$.

\begin{lem} Let $A$ be a deformable $R$-algebra, and $K'$ be a complete finitely ramified field extension of $K$.
\be \item There is a natural filtration on $A':=R'\otimes_R A$ such that $A'$ is a deformable $R'$-algebra.
\item $R'\otimes_R A_n = (A')_{en}$.
\item If $[K':K] < \infty$ then $K' \otimes_K \widehat{A_{n,K}}$ is isomorphic to $\widehat{A'_{en,K'}}$ as a complete doubly filtered $K'$-algebra.
\ee\end{lem}
\begin{proof}\hfill

(a) We define $F_iA'=R'\otimes_R F_iA$ for all $i$. Then $\gr A'$ is isomorphic to $R'\otimes_R \gr A$ and is therefore a flat $R'$-module.

(b) Since $\pi R' = \pi'^e R'$, we have
\[R' \otimes_R A_n = \sum_{i\geq 0}\pi^{in}R'\otimes_R F_iA = \sum_{i\geq 0}\pi'^{ien}F_iA' = (A')_{en}.\]

(c) Because $R'$ is a finitely generated $R$-module, $R' \otimes_R \widehat{A_n}$ is isomorphic to $\widehat{R' \otimes_R A_n}$. Therefore
\[K' \otimes_K \widehat{A_{n,K}} = K' \otimes_{R'} (R' \otimes_R \widehat{A_n}) \cong K' \otimes_{R'} \widehat{(A')_{en}} = \widehat{(A')_{en,K'}}\]
by part (b).
\end{proof}
Note that if $K'$ is an infinite extension of $K$ then $K' \otimes_K \widehat{A_{n,K}}$ will not be $\pi'$-adically complete, in general.
\begin{prop} Suppose that $A$ is an almost commutative affinoid $K$-algebra, $M$ is a finitely generated $A$-module and $K'$ is a complete, finitely ramified field extension of $K$. If $A'$ is the almost commutative affinoid $K'$-algebra obtained by completing $K'\otimes_K A$ then \[ \dim \Ch(A'\otimes_A M)=\dim \Ch(M).\]
\end{prop}
\begin{proof} Let $M'=A'\otimes_A M$ and let $k'$ be the residue field of $k$. Then $R'\otimes_R F_0M$ is an $R'$-lattice in $M'$ such that
\[\gr_0 M' =k' \otimes_{R'} (R' \otimes_R F_0M) \cong k'\otimes_R F_0M \cong k' \otimes_k \gr_0M.\]
This isomorphism induces a good double filtration on $M'$ from the good double filtration on $M$, and then $\Gr (M')\cong k'\otimes_k \Gr (M)$. The result follows.\end{proof}

\section{Crystalline differential operators on homogeneous spaces}\label{Diffops}
\subsection{Notation}\label{NotnDiffops}
In this section, $R$ will denote a fixed commutative Noetherian ground ring. Unadorned tensor products and scheme products will be assumed to be taken over $R$ and over $\Spec(R)$, respectively. If $V$ is a free $A$-module over a commutative ring $A$, then $\Sym_A V$ denotes the symmetric algebra of $V$ over $A$. All the results in this section are well-known when $R$ is a field; we spell them out in this more general context for the sake of the reader because we cannot find a suitable single reference.

Throughout $\S$\ref{Diffops}, $X$ will denote a scheme over $\Spec(R)$ which is smooth, separated and locally of finite type. We write $\T$ for the sheaf of sections of the tangent bundle $TX$. 

\subsection{Crystalline differential operators} \label{cryst}
\begin{defn} The \emph{sheaf of crystalline differential operators on $X$} is defined to be the enveloping algebra $\D$ of the tangent Lie algebroid $\T$.\end{defn}

\noindent Thus $\D$ is a sheaf of rings, generated by the structure sheaf $\O$ and the $\O$-module $\T$ subject to only the relations
\begin{itemize}
\item $f\partial=f\cdot \partial$ and $\partial f-f \partial=\partial(f)$ for each $f\in\O$ and $\partial\in\T$;
\item $\partial\partial'-\partial'\partial=[\partial,\partial']$ for $\partial,\partial'\in\T$.
\end{itemize}

\noindent Being a quotient of a universal enveloping algebra, the sheaf $\D$ comes equipped with a natural Poincar\'e-Birkhoff-Witt filtration
\[0 \subset F_0 \D \subset F_1 \D \subset F_2 \D \subset \cdots\]
consisting of coherent $\O$-submodules, such that
\[F_0 \D=\O, \quad F_1\D=\O \oplus \T, \quad\mbox{and}\quad F_m\D = F_1\D\cdot F_{m-1}\D \quad\mbox{for}\quad m>1.\]
Since $X$ is smooth, the tangent sheaf $\T$ is locally free and the associated graded algebra of $\D$ is isomorphic to the symmetric algebra of $\T$:
\[\gr \D := \bigoplus_{m=0}^\infty \frac{F_m \D}{F_{m-1} \D}\cong \Sym_{\O} \T.\]
If $q : T^\ast X \to X$ is the cotangent bundle of $X$ defined by the locally free sheaf $\T$, then we can also identify $\gr \D$ with $q_\ast\O_{T^\ast X}$.

\subsection{$\mb{H}$-torsors}\label{Torsor}

Let $\mb{H}$ be a flat affine algebraic group over $R$ of finite type. Let $\Ex{X}$ be a scheme equipped with an action $\mb{H} \times \Ex{X}  \to \Ex{X} $ of $\mb{H}$ on $\Ex{X}$. We say that a morphism $\xi : \Ex{X} \to X$ is an \emph{$\mb{H}$-torsor} if $\xi$ is faithfully flat and locally of finite type, the action of $\mb{H}$ respects $\xi$, and the map 
\[\Ex{X} \times \mb{H} \to \Ex{X} \times_X \tilde{X}\]
which sends $(x,h) \mapsto (x,hx)$ is an isomorphism. An open subscheme $U$ of $X$ is said to \emph{trivialise the torsor} $\xi$ if there exists an $\mb{H}$-invariant isomorphism
\[ U\times \mb{H} \tocong \xi^{-1}(U) \]
where $\mb{H}$ acts on $U\times \mb{H}$ by left translation on the second factor. Let $\mathcal{S}_X$ denote the set of open subschemes $U$ of $X$ such that 
\begin{itemize}
\item $U$ is affine,
\item $U$ trivialises $\xi$,
\item $\O(U)$ is a finitely generated $R$-algebra.
\end{itemize}	
Since $X$ is separated, it is easy to see that $\mathcal{S}_X$ is stable under intersections. Moreover, if $U \in \mathcal{S}_X$ and $W$ is an open affine subscheme of $U$, then $W \in \mathcal{S}_X$. We say that $\xi$ is \emph{locally trivial} for the Zariski topology if $X$ can be covered by opens in $\mathcal{S}_X$. Thus $\mathcal{S}_X$ is a base for $X$ whenever $\xi$ is locally trivial.

\begin{lem} If $\xi : \Ex{X} \to X$ is a locally trivial $\mb{H}$-torsor, then $\xi^\sharp : \O_X \to (\xi_\ast \O_{\Ex{X}})^{\mb{H}}$ is an isomorphism.
\end{lem}
\begin{proof} This is a local problem on $X$, so we may assume that $X$ is affine and $\xi : \Ex{X} = X \times \mb{H} \to X$ is the projection onto the first factor. Since $\O(X)$ is a flat $R$-module, it is a direct limit of free $R$-modules. Now
\[ (\xi_\ast \O_{\Ex{X}})^{\mb{H}}(X) = \O(X \times \mb{H})^{\mb{H}} = (\O(X) \otimes \O(\mb{H}))^{\mb{H}} = \O(X)\]
since rational cohomology commutes with direct limits by \cite[Lemma I.4.17]{Jantzen}.
\end{proof}

\subsection{The enhanced cotangent bundle}
\label{EnhCB}
 Let $\xi : \Ex{X} \to X$ be an $\mb{H}$-torsor. The action of $\mb{H}$ on $\Ex{X}$ induces a rational action of $\mb{H}$ on $\O(V)$ for any $\mb{H}$-stable open subscheme $V \subseteq \Ex{X}$ and therefore induces an action of $\mb{H}$ on $\T_{\Ex{X}}$ as follows:
\[ (h \cdot \partial)(f) = h \cdot \partial(h^{-1} \cdot f)\]
whenever $\partial \in \T_{\Ex{X}}, f \in \O$ and $h \in \mb{H}$. In this way we obtain the \emph{sheaf of enhanced vector fields} on $X$:
\[\Ex{\T} := (\xi_\ast \T_{\Ex{X}})^\mb{H}.\]
We can differentiate the $\mb{H}$-action on $\Ex{X}$ to obtain an $R$-linear Lie homomorphism
 \[j : \fr{h} \to \T_{\Ex{X}}\]
where $\fr{h}$ is the Lie algebra of $\mb{H}$. Now suppose that $\xi$ is locally trivial and let $\tau \in \Ex{\T}(U)$ for some open subscheme $U \subseteq X$. Then $\tau$ is an $\mb{H}$-invariant vector field on $\xi^{-1}(U)$, so in particular it is an $\mb{H}$-linear endomorphism of $\O(\xi^{-1}(U))$. Hence it preserves $\O(\xi^{-1}(U))^{\mb{H}}$ and by Lemma \ref{Torsor} it induces a vector field $\sigma(\tau) \in \T(U)$. This defines a map of $\O$-modules \[ \sigma : \Ex{\T} \longrightarrow \T\]
which is also known as the \emph{anchor map} of the Lie algebroid $\Ex{\T}$. It is easy to see that the anchor map fits into a complex of $\O$-modules
\begin{equation}\label{Cx} 0 \to \fr{h} \otimes \O \stackrel{j \otimes 1}{\longrightarrow} \Ex{\T} \stackrel{\sigma}{\longrightarrow} \T \to 0\end{equation}
which is functorial in $\Ex{X}$.
\begin{lem} The restriction of $(\ref{Cx})$ to any $U \in \mathcal{S}_X$ is split exact.  If $\xi$ is locally trivial, then $(\ref{Cx})$ is exact and $\Ex{\T}$ is locally free. \end{lem}
\begin{proof} We have $\T(U \times \mb{H})= \left(\T(U) \otimes \O(\mb{H}) \right) \oplus \left(\O(U) \otimes \T(\mb{H})\right).$ Since $\T(U)$ is a locally free $\O(U)$-module and $\O(U)$ is a flat $R$-module, $\T(U)$ is a flat $R$-module and hence a direct limit of free $R$-modules. Therefore
\[\left(\T(U) \otimes \O(\mb{H}) \right)^{\mb{H}} = \T(U) \otimes \O(\mb{H})^{\mb{H}} = \T(U)\]
again by \cite[Lemma I.4.17]{Jantzen}. Now $\T(\mb{H})^{\mb{H}} = j(\fr{h})$ so
\[ \T(U \times \mb{H})^{\mb{H}} = \T(U) \oplus \left(\O(U) \otimes j(\fr{h})\right).\]
Let $U \times \mb{H} \tocong \xi^{-1}(U)$ be an $\mb{H}$-invariant isomorphism. Then 
\[\Ex{\T}(U) = \T(\xi^{-1}(U))^{\mb{H}} \cong \T(U \times \mb{H})^{\mb{H}} = \T(U) \oplus \left(\O(U) \otimes j(\fr{h})\right)\]
and the first part follows because $U$ is affine and every term in $(\ref{Cx})$ is a quasi-coherent $\O$-module. The second part follows from the first and the functoriality of $(\ref{Cx})$.\end{proof}

We call the vector bundle $\tau : \Ex{T^\ast X} \to X$ associated to the locally free sheaf $\Ex{\T}$ the \emph{enhanced cotangent bundle} of $X$. It can in fact be shown that the enhanced cotangent bundle $\Ex{T^\ast X}$ is isomorphic to the quotient scheme $(T^\ast \Ex{X}) / \mb{H}$.

\subsection{Lemma}\label{DofTinvT} The natural map $U(\fr{h}) \longrightarrow \Gamma(\mb{H},\D)^{\mb{H}}$ is an isomorphism.
\begin{proof} Since $\mb{H}$ is affine and Noetherian, we can find isomorphisms 
\[\gr \Gamma(\mb{H}, \D) \cong \Gamma(\mb{H}, \gr \D) \cong \Gamma(\mb{H}, \Sym_{\O}\T) \cong \Sym_{\O(\mb{H})} \T(\mb{H}) \cong \O(\mb{H}) \otimes S(\fr{h})\] 
of commutative graded $R$-algebras. Since $S(\fr{h})$ is a flat $R$-module, taking $\mb{H}$-invariants and applying \cite[Lemma I.4.17]{Jantzen} shows that the natural map $S(\fr{h}) \to (\gr \Gamma(\mb{H},\D))^{\mb{H}}$ is an isomorphism. This map factors as follows:
\[S(\fr{h}) = \gr U(\fr{h}) \to \gr ( \Gamma(\mb{H},\D)^{\mb{H}}) \hookrightarrow (\gr \Gamma(\mb{H},\D))^{\mb{H}}\]
where the second arrow is injective since taking $\mb{H}$-invariants is always left exact. It follows that both arrows are isomorphisms, and $U(\fr{h})\to \Gamma(\mb{H},\D)^{\mb{H}}$ is an isomorphism as claimed. 
\end{proof}

\subsection{Relative enveloping algebras}\label{RelEnv}
Given an $\mb{H}$-torsor $\xi: \Ex{X}\to X$, $\xi_\ast\D_{\Ex{X}}$ is a sheaf of algebras on $X$ with an $\mb{H}$-action. Following \cite[p. 180]{BoBrII}, we define the \emph{relative enveloping algebra} of the torsor to be the sheaf of $\mb{H}$-invariants of $\xi_\ast\D_{\Ex{X}}$:
\[\Ex{\D}:=(\xi_\ast\D_{\Ex{X}})^\mb{H}.\]
This sheaf carries a natural filtration
\[F_m \Ex{\D} := (\xi_\ast F_m\D_{\Ex{X}})^{\mb{H}}\]
induced by the filtration on $\D_{\Ex{X}}$ by order of differential operators.

\begin{prop} There is an isomorphism of sheaves of filtered $R$-algebras
\[\D_{|U}\otimes U(\fr{h}) \tocong \Ex{\D}_{|U} \]
for any $U \in \mathcal{S}_X$. If $\xi$ is locally trivial, then there is an isomorphism
\[\tau_\ast \O_{\Ex{T^\ast X}} = \Sym_{\O} \Ex{\T}  \tocong \gr \Ex{\D}\]
of sheaves of graded $R$-algebras.
\end{prop}
\begin{proof} Let $U \in \mathcal{S}_X$, and let $U \times \mb{H}\tocong \xi^{-1}(U)$ be a trivialisation of $\xi$ over $U$. Using Lemma \ref{DofTinvT}, we obtain isomorphisms of filtered $R$-algebras
\[ \D(U) \otimes U(\fr{h}) \tocong \left(\D(U) \otimes \D(\mb{H})\right)^\mb{H}  \tocong \D(U\times \mb{H})^{\mb{H}} \tocong \Ex{\D}(U) \]
which are compatible with restrictions to Zariski open subschemes $V$ contained in $U$. Thus we obtain an isomorphism of sheaves of filtered $R$-algebras
\[ \eta : \D_{|U} \otimes U(\fr{h})  \tocong \Ex{\D}_{|U}.\]
Now the natural inclusion $\T_{\Ex{X}} \to F_1\D_{\Ex{X}}$ induces an $\O$-linear morphism $\Ex{\T} \to F_1 \Ex{\D}/ F_0 \Ex{\D}$ and therefore a morphism of graded $\O$-algebras
\[ \alpha : \Sym_{\O} \Ex{\T} \to \gr \Ex{\D}.\]
On the other hand, Lemma \ref{EnhCB} gives us an isomorphism of $\O_{|U}$-modules 
\[\theta : \T_{|U} \hspace{2mm} \oplus \hspace{2mm} \O_{|U} \otimes \fr{h} \tocong \Ex{\T}_{|U},\]
and these maps fit together into the commutative diagram
\[ \xymatrix{ \Gamma(U, \Sym_{\O} \Ex{\T}) \ar[r]^{\alpha(U)} & \Gamma(U, \gr \Ex{\D}) \\ \Sym_{\O(U)} \Ex{\T}(U) \ar[u] & \gr \Ex{\D}(U) \ar[u] \\ \Sym_{\O(U)}(\T(U) \oplus \O(U) \otimes \fr{h}) \ar[u]^{\Sym \theta(U)} & \gr ( \D(U) \otimes U(\fr{h}) ) \ar[u]_{\gr \eta(U)} \\ \Sym_{\O(U)} \T(U)\hspace{1mm} \otimes \hspace{1mm} S(\fr{h}) \ar[r]\ar[u] & \gr \D(U) \hspace{1mm}\otimes \hspace{1mm} \gr U(\fr{h}). \ar[u] }\]
The top two vertical maps are isomorphisms because $U$ is affine and Noetherian, and the bottom horizontal map is an isomorphism since $\gr \D \cong \Sym_{\O} \T$. The remaining vertical maps are isomorphisms by functoriality, and therefore $\alpha(U)$ is an isomorphism for any $U \in \mathcal{S}_X$. Since $\xi$ is locally trivial, it follows that $\alpha$ is an isomorphism. 

The equality $\tau_\ast \O_{\Ex{T^\ast X}} = \Sym_{\O} \Ex{\T}$ follows from the definition of the enhanced cotangent bundle $\tau : \Ex{T^\ast X} \to X$.\end{proof}

Note that even if $\xi$ is locally trivial, the sheaf $\Ex{\D}$ will in general \emph{not} be isomorphic to $\D \otimes U(\fr{h})$ since the torsor $\xi$ will not in general be globally trivial.

\begin{cor} Let $U \in \mathcal{S}_X$. Then
\be \item $\gr \Ex{\D}(U) \cong \Sym_{\O(U)} \Ex{\T}(U)$, and 
\item $\Ex{\D}(U)$ is an almost commutative $R$-algebra. \ee
\end{cor}
\begin{proof} (a) This follows from the proof of the Proposition.

(b) Since $\mb{H}$ is of finite type, its Lie algebra $\fr{h}$ has finite rank over $R$, so $\Ex{\T}(U)$ is a finitely generated projective $\O(U)$-module by Lemma \ref{EnhCB}. Hence $\gr \Ex{\D}(U)$ is a finitely generated commutative $\O(U)$-algebra. By definition of $\mathcal{S}_X$, $\O(U)$ is a finitely generated $R$-algebra, and it now follows from part (a) that $\gr \Ex{\D}(U)$ is a finitely generated commutative $R$-algebra.
\end{proof}

\subsection{Algebraic groups and homogeneous spaces} \label{AlgGps}
Let $\mb{G}$ be a connected, split reductive, affine algebraic group scheme over $R$. It is known that $\mb{G}$ is flat over $R$ \cite[\S II.1.1]{Jantzen}. Let $\mb{B}$ be a closed and flat Borel $R$-subgroup scheme, let $\mb{N}$ be its unipotent radical and let $\mb{H} := \mb{B}/\mb{N}$ be the \emph{abstract Cartan group}. 

Let $\Ex{\Fl}$ denote the homogeneous space $\mb{G} / \mb{N}$. Because $[\mb{B}, \mb{B}]$ is contained in $\mb{N}$, 
\[ b\mb{N}  \cdot g \mb{N} := gb \mb{N}, \quad\quad\quad b \in \mb{B}, \quad g \in \mb{G}\]
defines an action of $\mb{H}$ on $\Ex{\Fl}$, which commutes with the natural action of $\mb{G}$ on $\Ex{\Fl}$.

\begin{lem}\be \hfill
\item $\Ex{\Fl}$ and $\Fl := \mb{G} / \mb{B}$ are smooth separated schemes over $R$.
\item The action of $\mb{H}$ on $\Ex{\Fl}$ induces an isomorphism $\Ex{\Fl}/\mb{H}\cong\Fl$.
\item The natural projection $\xi \colon \Ex{\Fl}\rightarrow \Fl$ is a locally trivial $\mb{H}$-torsor.
\ee\end{lem}
\begin{proof} Because $\mb{G}$, $\mb{B}$ and $\mb{N}$ can be defined over $\mathbb{Z}$, we may assume that $R = \mathbb{Z}$.

(a) It follows from \cite[\S I.5.6(9)]{Jantzen} that $\Fl$ and $\Ex{\Fl}$ are schemes; since they are homogeneous spaces under the action of $\mb{G}$ it follows that they must be smooth.

(b) This is clear, when viewed on the level of the functor of points.

(c) By \cite[\S II.1.10(2)]{Jantzen}, we may cover $\Fl$ by open subschemes $U_i$ each isomorphic to $\mathbb{A}^{\dim \Fl}$ (the Weyl translates of the big Bruhat cell) and find morphisms $\sigma_i : U_i\rightarrow \mb{G}$ splitting the projection map $\mb{G}\rightarrow\Fl$. Composing these with the projection map $\mb{G}\rightarrow\Ex{\Fl}$ gives maps $\overline{\sigma_i}\colon U_i\rightarrow\Ex{\Fl}$ such that $\xi \circ \overline{\sigma_i}=\id_{U_i}$. Now $(u,b\mb{N}) \mapsto \sigma_i(u)b\mb{N}$ is the required $\mb{H}$-invariant isomorphism $U_i \times \mb{H} \tocong \xi^{-1}(U_i)$ and we may apply \cite[Proposition III.4.1]{Milne}.
\end{proof}
We call the homogeneous spaces $\Fl=\mb{G}/\mb{B}$ and $\Ex{\Fl} = \mb{G}/\mb{N}$ the \emph{flag variety} of $\mb{G}$ and the \emph{basic affine space} of $\mb{G}$, respectively. We will write $\Ex{\D}$ for the relative enveloping algebra of the $\mb{H}$-torsor $\xi : \Ex{\Fl}\to \Fl$, and also write $\tau  : \Ex{T^\ast \Fl} \to \Fl$ for the structure map of the enhanced cotangent bundle of the flag variety.

\subsection{The enhanced moment map}\label{EnhMom} Let $\fr{g},\fr{b},\fr{n}$ and $\fr{h}$ be the Lie algebras of $\mb{G}$, $\mb{B}$, $\mb{N}$ and $\mb{H}$, respectively. We can differentiate the natural $\mb{G}$-action on $\Ex{\Fl}$ to obtain an $R$-linear Lie homomorphism
 \[\varphi : \fr{g} \to \T_{\Ex{\Fl}}.\]
Since the $\mb{G}$-action commutes with the $\mb{H}$-action on $\Ex{\Fl}$, this map descends to an $R$-linear Lie homomorphism $\varphi : \fr{g} \to \Ex{\T}_\Fl$ and an $\O_{\Fl}$-linear morphism
\[ \varphi : \O_{\Fl} \otimes \fr{g} \longrightarrow \Ex{\T}_{\Fl}\]
of locally free sheaves on $\Fl$, which dualizes to give a morphism of vector bundles over $\Fl$ from the enhanced cotangent bundle to the trivial vector bundle of rank $\dim \fr{g}$:
\[ \Ex{T^\ast \Fl} \longrightarrow \Fl \times \fr{g}^\ast.\]
Here $\fr{g}^\ast := \Spec\left(\Sym_R\fr{g}\right)$ is being thought of as an $R$-scheme. Composing this morphism with the projection map onto the second factor gives the \emph{enhanced moment map}
\[ \beta : \Ex{T^\ast \Fl} \to \fr{g}^\ast\]
of $R$-schemes.  This morphism is also sometimes known as the \emph{Grothendieck-Springer resolution} of $\fr{g}^\ast$.

\begin{prop}\hfill
\be
\item The morphism of sheaves $\varphi : \O_{\Fl} \otimes \fr{g} \longrightarrow \Ex{\T}_{\Fl}$ is surjective.
\item The enhanced moment map is a projective morphism.
\ee
\end{prop}
\begin{proof} Since the space $\Ex{\Fl}$ is homogeneous, the geometric fibres of $\varphi$ are surjective by \cite[Proposition II.6.7]{Borel}. Part (a) follows because $\O_\Fl \otimes \fr{g}$ and $\Ex{\T}_\Fl$ are locally free.

The dual map $\Ex{T^\ast \Fl} \longrightarrow \Fl \times \fr{g}^\ast$ is a closed immersion by part (a). Part (b) now follows because $\Fl$ is a projective scheme by \cite[\S II.1.8]{Jantzen}.
\end{proof}

\subsection{Quantizing the moment map} \label{QMM}
The Lie homomorphism $\varphi : \fr{g} \to \Ex{\T}$ extends to a graded $R$-algebra homomorphism
\[\Sym(\varphi) : \Sym_R \fr{g} \to \Sym_{\O} \Ex{\T}\]
which can be viewed as the pull-back map on functions
\[ \beta^\sharp : \O(\fr{g}^\ast) \to  \tau_\ast \O_{\Ex{T^\ast \Fl}}\]
associated to the enhanced moment map $\beta : \Ex{T^\ast \Fl} \to \fr{g}^\ast$. It also extends to a filtered $R$-algebra homomorphism
\[U(\varphi): U(\fr{g}) \to \Ex{\D}\]
encoding the action of $\fr{g}$ on $\Ex{\Fl}$ by $\mb{H}$-invariant vector fields.

\begin{lem} The representation $U(\varphi) : U(\fr{g}) \to \Ex{\D}$ quantizes the enhanced moment map in the sense that $\gr U(\varphi)=\Sym(\varphi).$
\end{lem}
\begin{proof} Using Proposition \ref{RelEnv} and Lemma \ref{AlgGps}(c), we can naturally identify $\gr \Ex{\D}$ with $\Sym_{\O} \Ex{\T}$. The restriction of $\gr U(\varphi)$ to $\fr{g}$ is the representation $\varphi$ by definition, so $\gr U(\varphi)$ and $\Sym(\varphi)$ agree on the generators $\fr{g}$ of $\Sym_R \fr{g}$. The result follows.
\end{proof}

In fact, $U(\varphi)$ quantizes $\beta$ in a stronger sense. There are natural Poisson structures on both $\fr{g}^\ast$ and on $\Ex{T^\ast \Fl}$ which are induced by the non-commutative algebras $U(\fr{g})$ and $\Ex{\D}$, and the morphism $\Sym(\varphi)$ is compatible with these structures.

\subsection{The Harish-Chandra homomorphism}
\label{HCmap}
Since our group $\mb{G}$ is split by assumption, we can find a Cartan subgroup $\mb{T}$ of $\mb{G}$ complementary to $\mb{N}$ in $\mb{B}$. Let $i : \mb{T} \tocong \mb{H}$ denote the natural isomorphism, and let $i : \fr{t} \tocong \fr{h}$ be the induced isomorphism between the corresponding Lie algebras. The adjoint action of $\mb{T}$ on $\fr{g}$ induces a root space decomposition
\[ \fr{g} = \fr{n} \oplus \fr{t} \oplus \fr{n}^+\]
and we will regard $\fr{n}$, the Lie algebra of $\mb{N}$, as being spanned by \emph{negative} roots. This decomposition induces an isomorphism of $R$-modules
\[\U{g} \cong U(\fr{n}) \otimes U(\fr{t}) \otimes U(\fr{n}^+)\]
and a direct sum decomposition
\[\U{g} = \U{t} \oplus \left( \fr{n} \U{g} + \U{g}\fr{n}^+\right).\]
Now the adjoint action of the group $\mb{G}$ induces a rational action of $\mb{G}$ on $\U{g}$ by algebra automorphisms, so we may consider the subring $\U{g}^{\mb{G}}$ of $\mb{G}$-invariants. We call the composite of the natural inclusion of $\U{g}^{\mb{G}} \hookrightarrow \U{g}$ with the projection $\U{g} \twoheadrightarrow \U{t}$ onto the first factor defined by this decomposition the \emph{Harish-Chandra homomorphism}:
\[\phi : \U{g}^{\mb{G}} \longrightarrow \U{t}.\]
Recall from $\S \ref{EnhCB}$ that the infinitesimal action of $\mb{H}$ on $\Ex{\Fl}$ is denoted by 
\[j : \fr{h} \to \Ex{\T}_\Fl\]
and that it extends to an $R$-algebra homomorphism $j : \U{h} \to \Ex{\D}$. Since $\mb{H}$ is commutative, Proposition \ref{RelEnv} shows that $j$ is a central embedding.

\begin{lem} Suppose that $R$ is an integral domain. Then there exists a commutative diagram
\[\xymatrix{ \U{g}^{\mb{G}} \ar[d]\ar[r]^{\phi} & \U{t} \ar[d]^{j \circ i} \\
\U{g} \ar[r]_{U(\varphi)} & \Ex{\D}
}\]
of filtered rings: the restriction of $U(\varphi)$ to $\U{g}^{\mb{G}}$ is equal to $j \circ i \circ \phi$.
\end{lem}
\begin{proof} Since the objects in this diagram have no $R$-torsion by the Poincar\'e-Birkhoff-Witt Theorem, we may replace $R$ by an algebraic closure its field of fractions. But in this case this result is well-known: when the characteristic of $R$ is zero this was first observed in \cite[\S 9]{GK}. See \cite[Lemma 3.1.5(b)]{BMR1} for the general case.\end{proof}

\section{Completions, deformations and characteristic varieties}

We refer the reader to \cite[\S 0.5.3.1]{EGAI} for the definition of \emph{coherent $\mathcal{A}$-modules} over a sheaf $\mathcal{A}$ of not necessarily commutative rings over a topological space, and denote the abelian category of coherent $\mathcal{A}$-modules by $\coh(\mathcal{A})$. Recall that the sheaf $\mathcal{A}$ is said to be \emph{coherent} if $\mathcal{A}$ is itself a coherent $\mathcal{A}$-module. We begin by establishing some generalities on coherent modules and $I$-adic completions, following \cite[\S 3]{Berth}.

\subsection{Coherently $\msD$-affine spaces}\label{CohDaffine}
Let $X$ be a topological space and let $\msD$ be a coherent sheaf of rings on $X$. 

\begin{defn}  We say that $X$ is \emph{coherently $\msD$-acyclic} if every coherent $\msD$-module $\M$ is $\Gamma(X,-)$-acyclic and has the property that $\M(X)$ is a coherent $\msD(X)$-module. 

We say that $X$ is \emph{coherently $\msD$-affine} if $X$ is coherently $\msD$-acyclic and every coherent $\msD$-module is generated by its global sections as a $\msD$-module.

If $\mathcal{S}$ is a base for $X$, we say that $\mathcal{S}$ is \emph{coherently $\msD$-acyclic}, respectively \emph{coherently $\msD$-affine}, if for all $U\in\mathcal{S}$, $U$ is coherently $\msD_{|U}$-acyclic, respectively coherently $\msD_{|U}$-affine.
\end{defn}

A classic example of a space that is coherently $\msD$-affine is obtained by taking $X$ to be a smooth affine complex algebraic variety, and $\msD$ the sheaf of differential operators on $X$. 

The main reason for these definitions comes from the following

\begin{prop} Let $X$ be coherently $\msD$-acyclic. Then $\ker \Gamma(X,-)$ is a Serre subcategory of $\coh(\msD)$, and $\Gamma(X,-)$ and $\msD \otimes_{\msD(X)} -$ induce mutually inverse equivalences of abelian categories between $\coh(\msD)/\ker \Gamma(X,-)$ and the category $\coh(\msD(X))$ of coherent $\msD(X)$-modules. Moreover if $X$ is coherently $\msD$-affine then $\ker \Gamma(X,-)=0$. 
\end{prop}
\begin{proof} Let $\Gamma := \Gamma(X,-)$, $D := \msD(X)$ and $\Loc := \msD \otimes_{D} -$. Since $X$ is coherently $\msD$-acyclic, $\Gamma$ is exact on $\coh(\msD)$ which implies that $\ker \Gamma$ is closed under subquotients and extensions.

Now $(\Loc, \Gamma)$ is an adjuction between the category of all $D$-modules, and the category of all sheaves of $\msD$-modules. Since $\Loc$ is right exact and $\msD$ is coherent, $\Loc$ sends $\coh(D)$ to $\coh(\msD)$. Since $X$ is coherently $\msD$-acyclic, $(\Loc, \Gamma)$ restricts to an adjunction between $\coh(D)$ and $\coh(\msD)$ and $\Gamma$ is exact. By \cite[Lemma 2.4]{BeiGinz} it thus suffices for the second part to prove that  the counit $M \stackrel{\eta_M}{\longrightarrow} \Gamma(\Loc(M))$ is an isomorphism for $M \in \coh(D)$.  

Since $\Loc$ is right exact and $\Gamma$ is exact on $\coh(\msD)$, the composite $\Gamma\circ \Loc$ is also right exact. Since $M$ is coherent, it has a finite presentation; because $\eta_D$ is an isomorphism by definition, the Five Lemma now implies that $\eta_M$ is also an isomorphism as required. 

The final part is immediate; if $\M$ is generated by global sections then $\Gamma(\M)=0$ if and only if $\M=0$.
\end{proof}

\subsection{Completions}\label{Compl}
Recall from \cite[\S II.9]{Hart} that inverse limits exist in the category of sheaves of abelian groups over any topological space.

\begin{defn} Suppose that $\msD$ contains a constant central subsheaf $Z$, which contains an ideal $I$. We define $\hc{D} := \invlim \msD / I^n \msD$ be the \emph{$I$-adic completion} of $\msD$.
\end{defn}

This is a sheaf of $Z$-algebras on $X$ and $\Gamma(U, \hc{D}) = \invlim \Gamma(U, \msD/ I^n \msD)$ for any open subset $U$ of $X$. We will use the following elementary result repeatedly.

\begin{lem} Suppose $\Gamma(X,-)$ is exact on $\coh(\msD)$. Then for any coherent $\msD$-module $\M$ and any ideal $J$ in $Z$ such that $J\cdot \msD(X)$ is finitely generated, there is a natural isomorphism of $\msD(X)$-modules
\[ (\M / J \M)(X) \cong \M(X) / J \M(X).\]
\end{lem}
\begin{proof} Choose $z_1,\ldots,z_m \in J$ that generate $J \cdot \msD(X)$. Since $\Gamma(X, -)$ is exact on $\coh(\msD)$ by assumption and $Z$ is central in $\msD$, the exact sequence
\[ \M^m \stackrel{(z_1,\ldots,z_m)}{\xrightarrow{\hspace{1cm}}} \M \to \M / J \M \to 0\]
in $\coh(\msD)$ gives rise to the exact sequence
\[ \M(X)^m \stackrel{(z_1,\ldots,z_m)}{\xrightarrow{\hspace{1cm}}} \M(X) \to (\M / J \M)(X) \to 0\]
and the result follows.
\end{proof}
If $\msD(X)$ is left Noetherian, then $I^n \cdot \msD(X)$ is finitely generated for all $n \geq 1$. Lemma \ref{Compl} now implies that $\hc{D}(X)$ is just the $I$-adic completion of $\msD(X)$:
\[\hc{D}(X) = \invlim (\msD / I^n \msD)(X) \cong \invlim \msD(X) / I^n\msD(X).\] 
\begin{defn} If $\mathcal{S}$ is a base for $X$ we say that $\msD$ is \emph{Noetherian on $\mathcal{S}$} if $\msD(U)$ is left Noetherian for all $U \in \mathcal{S}$.
\end{defn}
\subsection{The functor $M \mapsto M^\Delta$}
\label{LocAffine}
Throughout $\S\ref{LocAffine}-\ref{CohMcomplete}$ we fix a base $\mathcal{S}$ for $X$ such that $X \in \mathcal{S}$, and assume that:
\begin{itemize}
\item $\msD$ is Noetherian on $\mathcal{S}$,
\item $\mathcal{S}$ is coherently $\msD$-acyclic, and
\item $\hc{D}$ is coherent.
\end{itemize}
Let $D := \msD(X)$. We define a functor $M \mapsto M^{\Delta}$ from $\hD := \widehat{\msD(X)}$-modules to sheaves of $\hc{D}$-modules by the formula
\[ M^{\Delta} := \invlim \msD \otimes_{D} M / I^n M.\]

Since we are assuming that $X \in \mathcal{S}$, the algebra $D$ is left Noetherian so $\coh(D)$ is simply the category of all finitely generated $D$-modules.
\begin{prop} \hfill
\be
\item The functor $M \mapsto M^{\Delta}$ is exact on $\coh(D)$.
\item $\hD^{\Delta} = \hc{D}$.
\item $M^\Delta$ is a coherent $\hc{D}$-module whenever $M$ is a finitely presented $\hD$-module.
\item If $u : \hc{D}^r \to \hc{D}^s$ is a map of $\hc{D}$-modules, then $\coker(u) \cong M^\Delta$ for some finitely presented $\hD$-module $M$.
\ee
\end{prop}
\begin{proof}
(a) Let $0 \to A \to B \to C \to 0$ be an exact sequence of coherent $D$-modules. Since $X$ is coherently $\msD$-affine, $\msD \otimes_{D} -$ is exact on $\coh( D )$ by Proposition \ref{CohDaffine}. Therefore the sequence of towers of $\msD$-modules
\[0 \to \left[ \msD \otimes_{D} \frac{A + I^n B}{I^nB} \right]_n \to \left[ \msD \otimes_{D} \frac{B}{I^nB} \right]_n \to \left[ \msD \otimes_{D} \frac{C}{I^nC}\right]_n \to 0\]
is exact. The maps in the left-most non-zero tower are surjective, so it trivially satisfies the Mittag-Leffler condition. Taking inverse limits gives a short exact sequence
\[0 \to \invlim  \msD \otimes_{D} \frac{A + I^n B}{I^nB} \to B^\Delta \to C^\Delta \to 0.\]
Since $D$ is left Noetherian, by the Artin-Rees Lemma \cite[\S 3.2.3(i)]{Berth} we can find an integer $n_0$ such that $I^n A \subseteq A \cap I^n B \subseteq I^{n - n_0}A$ for all $n \geq n_0$, so the natural map 
\[A^\Delta = \invlim \msD \otimes_D \frac{A}{I^nA}\longrightarrow \invlim \msD \otimes_D \frac{A + I^n B}{I^nB}\]
is an isomorphism and the result follows.

(b) $\hD^\Delta = \invlim \msD \otimes_D \frac{\hD}{I^n \hD} = \invlim \msD \otimes_D \frac{D}{I^nD} \cong \invlim \msD / I^n \msD = \hc{D}$.

(c) Since $M^\Delta$ is a $\hc{D}$-module, there is a natural map $\hc{D} \otimes_{\hD} M \longrightarrow M^\Delta$ of $\hc{D}$-modules, which is an isomorphism when $M = \hD^r$ by part (b). Since $M$ is finitely presented, using parts (a) and (b) together with the Five Lemma shows that $\hc{D} \otimes_{\hD} M$ is in fact naturally isomorphic to $M^\Delta$. Since $\hc{D}$ is coherent by assumption, it also follows that $M^\Delta$ is coherent.

(d) Let $M$ be the cokernel of the map $\Gamma(X,u) : \Gamma(X, \hc{D}^r) \to \Gamma(X, \hc{D}^s)$. Then $\hD^r \stackrel{\Gamma(X,u)}{\longrightarrow}\hD^s \to M \to 0$ is exact. Applying the exact functor $(-)^\Delta$ to this presentation of $M$ produces the exact sequence $\hc{D}^r \stackrel{u}{\longrightarrow} \hc{D}^s \to M^\Delta \to 0$, so $\coker(u) \cong M^\Delta$ as required.
\end{proof}

\subsection{Lemma}\label{CohMcomplete}
Let $\M$ be a coherent $\hc{D}$-module. Then the natural map $\M \to \invlim \M / I^n \M$ is an isomorphism.
\begin{proof} Since $\M$ is coherent, by shrinking $X$ if necessary we may assume that $\M$ is finitely presented: $\M = \coker(u)$ for some morphism $u : \hc{D}^r \to \hc{D}^s$ of $\hc{D}$-modules. We may then assume that $\M  = M^\Delta$ for some finitely presented $\hD:= \hc{D}(X)$-module $M$ by Proposition \ref{LocAffine}(d). Let $z_1,\ldots, z_m$ generate $I^n$, so that we have the exact sequence of finitely generated $\hD$-modules
\[ M^m \stackrel{(z_1,\ldots,z_m)}{\xrightarrow{\hspace{1cm}}} M \to M / I^n M \to 0.\]
By Proposition \ref{LocAffine}(a), the sequence
\[ (M^\Delta)^m \stackrel{(z_1,\ldots,z_m)}{\xrightarrow{\hspace{1cm}}} M^\Delta \to (M / I^n M)^\Delta \to 0\]
is exact, so \[ \M /I^m \M  = M^\Delta / I^n M^\Delta \cong (M / I^n M)^\Delta \cong \msD \otimes_D \frac{M}{I^nM}\]
for any $n \geq 1$. Hence $\M = M^\Delta = \invlim \msD \otimes_D M/I^nM \cong \invlim \M / I^n \M.$ \end{proof}

\subsection{Theorem}\label{Cartan} Suppose that $\msD$ is Noetherian on $\mathcal{S}$, $\mathcal{S}$ is coherently $\msD$-acyclic and $\hc{D}$ is coherent.  Then $\mathcal{S}$ is also coherently $\hc{D}$-acyclic. If moreover $\mathcal{S}$ is coherently $\msD$-affine then it is also coherently $\hc{D}$-affine.

\begin{proof} For the first part it suffices to show that if $X\in\mathcal{S}$ then $X$ is coherently $\hc{D}$-acyclic.

Let $\M$ be a coherent $\hc{D}$-module. For each $n \geq 1$, let $M_n := \Gamma(X, \M / I^n \M)$ and $D_n := \Gamma(X, \msD / I^n \msD)$. Define $M := \Gamma(X, \M)$, $D := \Gamma(X, \msD)$ and note that $D_n \cong D / I^n D$ by Lemma \ref{Compl}. 

By Lemma \ref{CohMcomplete}, $\M$ is isomorphic to $\invlim \M / I^n \M$. Each $\M / I^n \M$ is a coherent $\hc{D}$-module killed by $I^n$, and is therefore a coherent $\msD$-module. So $H^i(X, \M / I^n \M) = 0$ for all $n \geq 1$ and all $i > 0$ and $M_n$ is a finitely generated $D$-module for all $n\geq 1$, because $X$ is coherently $\msD$-acyclic. Next, the short exact sequence
\[0 \to I^n \M / I^{n+1} \M \to \M / I^{n+1} \M \to \M / I^n \M \to 0\]
consists of coherent $\msD$-modules; since $H^1(X, I^n \M / I^{n+1}\M) = 0$, each map in the tower $ \cdots \to M_{n+1} \to M_n \to \cdots \to M_1$ is surjective. Hence this tower satisfies the Mittag-Leffler condition, and it follows from \cite[Proposition 0.13.3.1]{EGAIII} that $H^i(X, \M) = 0$ for all $i > 0$. Thus $\M$ is $\Gamma$-acyclic.

Now, the algebra $\hD := \hc{D}(X)$ is isomorphic to the $I$-adic completion of $D$. Since $D$ is left Noetherian by assumption, $\hD$ is also left Noetherian by \cite[\S 3.2.3(vi)]{Berth}. Thus, to show that $M$ is a coherent $\hD$-module it is enough to prove that $M$ is finitely generated. Since $\Gamma(X,-)$ is exact on $\coh(\msD)$ by assumption and since
\[ \frac{\M / I^n\M} {I^{n-1} \cdot ( \M / I^n \M )} \cong \frac{\M}{I^{n-1} \M},\]
Lemma \ref{Compl} implies that
\[ M_n / I^{n-1}M_n \cong M_{n-1}\]
for all $n \geq 1$. Since $M_1$ is a finitely generated $D_1$-module, $M = \M(X) = \invlim M_n$ is a finitely generated $\hD := \hc{D}(X)$-module by \cite[Lemma 3.2.2]{Berth}. Thus $\mathcal{S}$ is coherently $\hc{D}$-acyclic.

For the last part it suffices to show that $\M$ is generated by global sections. By Proposition \ref{CohDaffine}, $\M / I^n \M \cong \msD \otimes_D M_n$, and $M_n \cong M / I^n M$ by \cite[Lemma 3.2.2]{Berth}. So Lemma \ref{CohMcomplete} implies that 
\[ \M \cong \invlim \M / I^n \M \cong \invlim \msD \otimes_D M/I^n M = M^\Delta.\]
Choose a presentation $F_1 \to F_0 \to M \to 0$ of $M$ consisting of finitely generated free $\hD$-modules. It inducs an exact sequence $F_1^\Delta \to F_0^\Delta \to M^\Delta \to 0$ of coherent $\hc{D}$-modules by Proposition \ref{LocAffine}(a), and a commutative diagram
\[ \xymatrix{ \hc{D} \otimes_{\hD} F_1 \ar[r]\ar[d] & \hc{D} \otimes_{\hD} F_0 \ar[r]\ar[d] & \hc{D} \otimes_{\hD} M \ar[r]\ar[d] & 0 \\
F_1^\Delta \ar[r] & F_0^\Delta \ar[r] & M^\Delta \ar[r] & 0 }\]
with exact rows. Proposition \ref{LocAffine}(b) implies that the two vertical arrows on the left are isomorphisms, so the map 
\[ \hc{D} \otimes_{\hD}M \tocong M^\Delta\]
is an isomorphism by the Five Lemma and $\M$ is generated by global sections. \end{proof}

\subsection{Notation}\label{NotnDefCompCh}
Now we return to the setting and notation of $\S \ref{NotnDiffops}$, and make the additional hypothesis that $R$ is a complete discrete valuation ring with uniformizer $\pi$, residue field $k$ and field of fractions $K$. We will denote the special and generic fibres of $X$ by $X_k := X \times \Spec(k)$ and $X_K := X \times \Spec(K)$ respectively. We also suppose that we are given a locally trivial $\mb{H}$-torsor $\xi : \Ex{X} \to X$ over $X$ for some flat affine algebraic group $\mb{H}$ of finite type over $R$, as in $\S \ref{Torsor}$. Let $\mathcal{S} := \mathcal{S}_X$ be the base for $X$ consisting of open affine subschemes $U$ of $X$ of finite type that trivialise $\xi$.
\vspace{0.1cm}
\noindent We fix a deformation parameter $n \geq 0$.

\subsection{The sheaf $\dnt$}\label{EnhDn} Recall the category of deformable $R$-algebras from $\S$\ref{Defs}, and let $\Ex{\D}$ denote the relative enveloping algebra of the $\mb{H}$-torsor $\xi$.

\begin{lem}\be \hfill
\item The category of deformable $R$-algebras has all limits, and the forgetful functor to $R$-algebras preserves limits.
\item $\Ex{\D}$ is a sheaf of deformable $R$-algebras.\ee\end{lem}
\begin{proof} (a) Let $(A_i)_{i \in I}$ be an inverse system of deformable $R$-algebras with connecting homomorphisms $f_{ji} : A_j \to A_i$ for $j \geq i$, and let 
\[A := \invlim A_i = \{ (a_i) \in \prod\limits_{i \in I} A_i \st f_{ji}(a_j) = a_i \quad\mbox{for all}\quad j \geq i\}\]
be the inverse limit. Equip $\prod\limits_{i \in I} A_i$ with the product filtration, and $A \subseteq \prod\limits_{i \in I} A_i$ with the subspace filtration. Then $A$ is a positively filtered $R$-algebra, and $\gr A$ is an $R$-submodule of the direct product of the $\gr A_i$. Since $R$ is a discrete valuation ring, it follows that $\gr A$ is flat over $R$.

(b) Certainly $\Ex{\D}$ is a sheaf of filtered $R$-algebras. Let $U \in \mathcal{S}$; then $\gr (\Ex{\D}(U))$ is a locally free $\O(U)$-module by Corollary \ref{RelEnv} and hence is a flat $R$-algebra. Thus $\Ex{\D}(U)$ is a deformable $R$-algebra for any $U \in \mathcal{S}$. Now $\mathcal{S}$ is a base for $X$, so $\Ex{\D}(U)$ is a deformable $R$-algebra for any open $U \subseteq X$ by part (a).
\end{proof}

\begin{defn} Let $\dnt$ be the sheafification of the presheaf obtained by postcomposing $\Ex{\D}$ with the deformation functor $A \to A_n$ from $\S$\ref{Defs}.
\end{defn}

Note that the Lemma implies that $\dnt$ is in fact a sheaf of deformable $R$-algebras. It can be shown that the deformation functors do not commute with arbitrary finite inverse limits; this explains the need to sheafify. However it is still possible to compute local sections of $\dnt$ over $U \in \mathcal{S}$ as follows.

\begin{prop} 
\hfill\be
\item $\dnt(U) \cong \Ex{\D}(U)_n$ for all $U \in \mathcal{S}$. 
\item $\dnt(U)$ is almost commutative for all $U \in \mathcal{S}$.
\item There is an isomorphism of sheaves $\tau_\ast \O_{\Ex{T^\ast X}} = \Sym_{\O} \Ex{\T} \tocong \gr \dnt $.
\item The sheaf $\dnt$ is coherent.
\ee \end{prop}
\begin{proof}
(a) By  \cite[\S 0.3.2.2]{EGAI}, it is enough to show that the sequence
\[ 0 \to \Ex{\D}(U)_n \to \bigoplus_{i=1}^m \Ex{\D}(U_i)_n \longrightarrow \bigoplus_{i<j} \Ex{\D}(U_i \cap U_j)_n\] of deformable $R$-algebras is exact whenever $U = U_1 \cup \cdots \cup U_m$ is a cover of $U$ by some other $U_i \in \mathcal{S}$.  By Lemma \ref{Defs} and Proposition \ref{RelEnv}, the associated graded of this sequence is isomorphic to
\[ 0 \to \Sym_{\O(U)}\Ex{\T}(U) \to \bigoplus_{i=1}^m \Sym_{\O(U_i)}\Ex{\T}(U_i) \longrightarrow \bigoplus_{i<j} \Sym_{\O(U_i \cap U_j)}\Ex{\T}(U_i \cap U_j) \]
 and this is exact since $\Sym_\O \Ex{\T}$ is a sheaf on $X$. Hence the first sequence is exact.

(b) This follows from part (a), Lemma \ref{Defs} and Corollary \ref{RelEnv}.

(c) By Lemma \ref{Defs}, for every open subscheme $U$ of $X$ there is an isomorphism of graded $R$-algebras $\gr (\Ex{\D}(U)) \to \gr (\Ex{\D}(U)_n)$ which is natural in $U$. After applying sheafification, this induces a morphism of sheaves of graded $R$-algebras 
\[\gamma : \gr \Ex{\D} \to \gr \Ex{\D}_n.\] 
The sections of this morphism over $U \in \mathcal{S}$ can be identified with $\gr (\Ex{\D}(U)) \to \gr (\Ex{\D}(U)_n)$ because $U$ is affine and because $\Ex{\D}_n(U) = \Ex{\D}(U)_n$ by part (a). Thus $\gamma(U)$ is an isomorphism for all $U \in \mathcal{S}$, so $\gamma$ is an isomorphism since $\mathcal{S}$ is a base for $X$. Now precompose $\gamma$ with the isomorphism $\alpha : \Sym_{\O}\Ex{\T} \tocong \gr \Ex{\D}$ given by Proposition \ref{RelEnv}.



(d) Let $V \subseteq U$ in $\mathcal{S}$; by \cite[Proposition 3.1.1]{Berth} it will be enough to show that $\dnt(U)$ is Noetherian and that the restriction morphism $\dnt(U) \to \dnt(V)$ is flat. The first part follows from (a) and (b), and for the second, it is enough to show that the morphism $\gr \dnt(U) \to \gr \dnt(V)$ is flat by \cite[Proposition 1.2]{ST}. But by part (c), this morphism is just the restriction map of regular functions $\O(\Ex{T^\ast U}) \to \O(\Ex{T^\ast V})$ corresponding to the Zariski open immersion $\Ex{T^\ast V} \to \Ex{T^\ast U}$ and is therefore flat.
\end{proof}

\subsection{Good filtrations}\label{GoodFilt}
Let $\M$ be a coherent $\dnt$-module. A \emph{filtration} $F_\bullet \M$ of $\M$ is a sequence $0 \subseteq F_0 \M \subseteq F_1 \M \subseteq \cdots$ of subsheaves of abelian groups of $\M$ such that $F_i \dnt \cdot F_j \M \subseteq F_{i+j} \M$ for all $i,j \geq 0$. We say that this filtration is \emph{good} if the associated graded sheaf $\gr \M := \bigoplus_{i\geq 0} F_i\M / F_{i-1} \M$ is a coherent $\gr \dnt$-module.

\begin{lem} Let $\M$ be a coherent $\dnt$-module.
\be \item $\M$ is quasi-coherent as an $\O$-module.
\item There exists an $\O$-coherent submodule $\F$ of $\M$ such that $\M = \dnt \cdot \F$.
\item There exists at least one good filtration $F_\bullet \M$ on $\M$.
\item If $F_\bullet \M$ is a good filtration on $\M$, then each $F_i\M$ is a coherent $\O$-module.
\ee \end{lem}
\begin{proof}
(a) By Proposition \ref{EnhDn}(c), each graded piece $\gr_i \dnt$ of $\dnt$ is a coherent $\O$-module. Since coherent $\O$-modules are closed under extensions by \cite[Proposition II.5.7]{Hart}, each filtered part $F_i \dnt$ is a coherent $\O$-module. Now $\dnt$ is the  direct limit of the $F_i \dnt$, and therefore a quasi-coherent $\O$-module. Since $\M$ is a coherent $\dnt$-module, $\M$ is locally finitely generated over $\dnt$ and therefore quasi-coherent as an $\O$-module.

(b) By \cite[Corollaire I.9.4.9]{EGAI}, $\M$ is the direct limit of its $\O$-coherent submodules. Since $\M$ is coherent as a $\dnt$-module, we can find a coherent $\O$-submodule $\F$ of $\M$ such that $\M = \dnt \cdot \F$.

(c) Setting $F_i\M := F_i \dnt \cdot \F$ defines a good filtration on $\M$.

(d) This is a local statement, so we may assume that $\gr \M$ is finitely generated over $\gr \dnt$. But then each $\gr_i \M$ is a quotient of a direct sum of finitely many copies of $\gr_i \dnt$, which is a coherent $\O$-module by Proposition \ref{EnhDn}(c), so each $\gr_i \M$ is coherent over $\O$. The result follows since coherent $\O$-modules are stable under extensions.
\end{proof}

\subsection{The $\pi$-adic completion $\hdnt$}
\label{Hdnt}
We now apply the theory from $\S\ref{CohDaffine}-\S\ref{Cartan}$ by taking $\msD$ to be the sheaf $\dnt$ constructed in $\S \ref{EnhDn}$, and taking $I$ to be the ideal generated by $\pi$.

\begin{defn} Let $\hdnt := \invlim \dnt / \pi^a \dnt$ be the $\pi$-adic completion of $\dnt$.
\end{defn}

This is a sheaf of $R$-algebras on $X$ and $\Gamma(U, \hdnt) = \invlim \Gamma(U, \dnt/ \pi^a \dnt)$ for any open subscheme $U$ of $X$. The restriction of $\hdnt$ to the generic fibre $X_K$ is clearly zero, so $\hdnt$ is only supported on the special fibre $X_k$ of $X$. Because $X_k$ is closed in $X$, we identify sheaves on $X$ supported only on $X_k$ with their sheaf-theoretic pullbacks to the special fibre in what follows without further mention.

\begin{prop} Let $U \in \mathcal{S}$. \be
\item $(\dnt/\pi^a\dnt)(U) \cong \dnt(U) / \pi^a\dnt(U)$ for all $a \geq 1$.
\item $\hdnt(U) \cong \widehat{\dnt(U)}$ is Noetherian.
\item The sheaf $\hdnt$ is coherent.
\ee\end{prop}
\begin{proof} (a) $\dnt$ is a quasi-coherent $\O$-module by Proposition \ref{EnhDn}(d) and Lemma \ref{GoodFilt}(a). Since $U$ is affine and Noetherian, $H^1(U, \dnt) = 0$ and the result follows.

(b) Part (a) gives that $\hdnt(U) \cong \widehat{\dnt(U)}$. But $\dnt(U) = \Ex{\D}(U)_n$ is Noetherian by Proposition \ref{EnhDn}(a) and (b), so $\widehat{\dnt(U)}$ is Noetherian by \cite[\S 3.2.3(vi)]{Berth}. 

(c) Let $V \subseteq U$ be in $\mathcal{S}$; by part (b) and \cite[Proposition 3.1.1]{Berth} it will be enough to show that the restriction morphism $\hdnt(U) \to \hdnt(V)$ is flat. By applying \cite[Proposition 1.2]{ST} twice, it is enough to show that 
\[\gr\left(\Ex{\D}(U)_n / \pi \Ex{\D}(U)_n\right) \to \gr\left(\Ex{\D}(V)_n / \pi \Ex{\D}(V)_n\right)\]
is flat. By Lemma \ref{GrComp} and Proposition \ref{RelEnv}, this morphism can be identified with $\O(\Ex{T^\ast U}) \otimes_R k \longrightarrow \O(\Ex{T^\ast V})\otimes_Rk$, which is flat because it is the pull-back of functions along the Zariski open immersion $\Ex{T^\ast V} \times_X X_k \to \Ex{T^\ast U} \times_X X_k$.
\end{proof}

\begin{cor} $\mathcal{S}$ is coherently $\hdnt$-affine.
\end{cor}
\begin{proof} It is enough to show that $X$ is coherently $\hdnt$-affine whenever $X\in \mathcal{S}$, so let's assume that this is the case. 

Since $\dnt$ is Noetherian on $\mathcal{S}$ and coherent by Proposition \ref{EnhDn}, and $\hdnt$ is coherent by part (c) of the Proposition, by Theorem \ref{Cartan} it suffices to show that $X$ is coherently $\dnt$-affine, so let $\M$ be a coherent $\dnt$-module.

Since $X$ is affine and Noetherian, and $\M$ is quasi-coherent as an $\O$-module by Lemma \ref{GoodFilt}(a), $H^i(X,\M) = 0$ for all $i > 0$.

Next, choose a good filtration on $\M$ using Lemma \ref{GoodFilt}(c). Then $\Gamma(X, \gr \M)$ is a finitely generated $\Gamma(X, \gr \dnt)$-module by \cite[Theorem II.5.4]{Hart}. Since $\gr \Gamma(X,\M)$ is a submodule of $\Gamma(X, \gr \M)$ and $\Gamma(X, \gr \dnt)$ is Noetherian by Proposition \ref{EnhDn}(c), it follows that $\M(X)$ is a finitely generated $\dnt(X)$-module.

Finally, since $\M$ is a quasi-coherent $\O$-module, and $X$ is affine, $\M$ is generated by global sections as an $\O$-module and therefore as a $\dnt$-module.
\end{proof}

\subsection{The sheaf $\hdntK$}\label{PiAdicC}  
Recall the category of complete doubly filtered $K$-algebras from $\S \ref{DoubleFilt}$.

\begin{lem} The category of complete doubly filtered $K$-algebras with positively filtered slices has finite limits, and the forgetful functor to $K$-algebras preserves finite limits.
\end{lem}
\begin{proof}
By \cite[Theorem V.2.1]{MacLane} it suffices to show that this category has finite products and equalisers.

Let $A_1,\ldots,A_n$ be complete doubly filtered $K$-algebras and let $A:=\prod A_i$ be their product as $K$-algebras. It is easy to check that $F_0A:=\prod F_0A_i$ is a $\pi$-adically complete $R$-lattice in $A$ and that $\gr_0 A=\prod \gr_0 A_i$ in the category of $k$-algebras. By giving $\gr_0A$ the product filtration induced from the $\gr_0A_i$ we can make $A$ into a complete doubly filtered $K$-algebra that satisfies the universal property for products.

Now let $A$ and $B$ be complete doubly filtered $K$-algebras with positively filtered slices, and let $f,g\colon A\to B$ be two morphisms. Let $C=\{a\in A\mid f(a)=g(a)\}$ be their equaliser in the category of $K$-algebras; then $F_0C := C\cap F_0A$ is a lattice in $C$ and it is $\pi$-adically complete being a closed $R$-submodule of $F_0A$. Since $B$ is $\pi$-torsion-free, we may identify $\gr_0C$ with a $k$-subalgebra of $\gr_0A$. When we equip $\gr_0 C$ with the subspace filtration from $\gr_0 A$, $C$ becomes a complete doubly filtered $K$-algebra with positive slice, since the filtration on $\gr_0A$ is positive by assumption. It is straightforward to verify that $C$ satisfies the universal property for equalisers. 
 \end{proof}
 

\begin{defn} The \emph{sheaf of completed, deformed, crystalline differential operators of the torsor $\xi : \Ex{X} \to X$ is} 
\[\hdntK := \hdnt \otimes_R K.\]
\end{defn}
Since $K$ is a flat $R$-module, for any open subset $U$ of $X$ we have
\[\Gamma(U, \hdntK) = \Gamma(U, \hdnt)\otimes_R K.\]

\begin{prop} \hfill
\be\item If $U \in \mathcal{S}$ then $\hdntK(U)$ is an almost commutative affinoid $K$-algebra and $\Gr(\hdntK(U)) \cong \O(\Ex{T^\ast U_k})$.
\item $\hdntK$ is a sheaf of complete doubly filtered $K$-algebras whenever $X$ is quasi-compact.
\item The sheaf $\hdntK$ is coherent.
\item $\hdnt / \pi \hdnt$ is isomorphic to $\dnt / \pi \dnt$.
\ee
\end{prop}
\begin{proof} (a) Since $U \in \mathcal{S}$, $\dnt(U)$ is an almost commutative $R$-algebra by Proposition \ref{EnhDn}(b), and we saw in $\S \ref{EnhDn}$ that $\dnt(U)$ is a deformable $R$-algebra. Hence 
\[\hdntK(U) = \hdnt(U) \otimes_R K \cong \widehat{\dnt(U)} \otimes_RK = \widehat{\Ex{\D}(U)_n} \otimes_RK \]
by Propositions \ref{Hdnt}(b) and \ref{EnhDn}(a), and this is an almost commutative affinoid $K$-algebra by Proposition \ref{acsa}. The second statement follows from Corollaries \ref{GrComp} and \ref{RelEnv}(a).

(b) Since $X$ is quasi-compact and locally Noetherian by assumption, every open subset $U$ of $X$ is the finite union $U = U_1\cup \cdots \cup U_m$ of some $U_i \in \mathcal{S}$. Since $\hdntK$ is a sheaf, $\hdntK(U)$ is the inverse limit of the $\hdntK(U_i)$ which are almost commutative affinoid $K$-algebras by part (a). Now apply the Lemma.

(c) This follows from Proposition \ref{Hdnt}(c).

(d) The natural map $\dnt \longrightarrow \hdnt$ induces a diagram of sheaves of $R$-algebras
\[ \xymatrix{ 0 \ar[r] & \dnt \ar[r]^{\pi}\ar[d] & \dnt \ar[r]\ar[d] & \dnt / \pi \dnt \ar[r]\ar[d] & 0 \\
0 \ar[r] & \hdnt \ar[r]_{\pi} & \hdnt \ar[r]  & \hdnt / \pi \hdnt \ar[r] & 0 }\]
with exact rows. By Proposition \ref{Hdnt}(a) and Corollary \ref{Hdnt}, the rows of the diagram
\[ \xymatrix{ 0 \ar[r] & \dnt(U) \ar[r]^{\pi}\ar[d] & \dnt(U) \ar[r]\ar[d] & (\dnt / \pi \dnt)(U) \ar[r]\ar[d] & 0 \\
0 \ar[r] & \hdnt(U) \ar[r]_{\pi} & \hdnt(U) \ar[r]  & (\hdnt / \pi \hdnt)(U) \ar[r] & 0 }\]
are still exact for any $U \in \mathcal{S}$. Since $\hdnt(U)$ is the $\pi$-adic completion of $\dnt(U)$ by Proposition \ref{Hdnt}(b), the third vertical map in this diagram is an isomorphism for any $U \in \mathcal{S}$ and the result follows.
\end{proof}

\subsection{Modules over $\hdntK$ and double filtrations}
\label{DoubleFiltSheaf}
Whenever $\N$ is a sheaf of $R$-modules on $X$, let $\N_K := \N \otimes_R K$ denote the corresponding sheaf of $K$-vector spaces. Let $\M$ be a sheaf of $K$-vector spaces over $X$ and let $F_0\M$ be a sheaf of $R$-submodules of $\M$. We say that $F_0\M$ is an \emph{$R$-lattice} in $\M$ if
\begin{itemize}
\item the natural map $(F_0 \M)_K \to \M$ is an isomorphism, and
\item $\bigcap_{a = 0}^\infty  \pi^a F_0 \M = 0$.
\end{itemize}
If $F_0 \M$ is an $R$-lattice in $\M$, we call $\gr_0 \M := F_0 \M / \pi F_0\M$ the \emph{slice} of $\M$; this is a sheaf of $k$-vector spaces. For example, $\hdnt$ is an $R$-lattice in $\hdntK$ with slice isomorphic to $\dnt / \pi \dnt$ by Proposition \ref{PiAdicC}(d).

Now suppose that $\M$ is a sheaf of $\hdntK$-modules. A \emph{double filtration} on $\M$ consists of an $R$-lattice $F_0\M$ in $\M$ which is a $\hdnt$-submodule, and a positive $\mathbb{Z}$-filtration $F_\bullet \gr_0 \M$ on $\gr_0 \M$ compatible with the filtration on $\dnt$. We call
\[\Gr \M := \gr(\gr_0 \M)\]
the \emph{associated graded sheaf} of $\M$ with respect to this double filtration. 

Note that $\hdntK(U)$ is a doubly filtered $K$-algebra by Proposition \ref{PiAdicC}(b) for any open $U \subseteq X$, and recall the notion of double filtrations on modules over such algebras from $\S \ref{GoodDouble}$.

\begin{lem} Let $(F_0 \M, F_\bullet \gr_0 \M)$ be a double filtration on a $\hdntK$-module $\M$ and let $U \subseteq X$ be an open subscheme. 
\be
\item $((F_0\M)(U), (F_\bullet \gr_0 \M)(U))$ is a double filtration on $M := \M(U)$.
\item There is a natural embedding $\Gr(M) \hookrightarrow (\Gr \M)(U)$ of $\Gr(\hdntK(U))$-modules. 
\ee
\end{lem}
\begin{proof} Certainly $F_0 M := (F_0\M)(U)$ is an $R$-lattice and a $\hdnt(U)$-submodule in $M$. The short exact sequence $0 \to F_0 \M \stackrel{\pi}{\to} F_0 \M \to \gr_0 \M \to 0$ induces an embedding $\gr_0 M \hookrightarrow (\gr_0 \M)(U)$ of $\hdnt(U) / \pi \hdnt(U)$-modules. The separated filtration on the sheaf $\gr_0\M$ induces a separated filtration on $(\gr_0\M)(U)$ and hence a separated filtration on $\gr_0 M$. Taking associated graded modules, we obtain an inclusion
\[\Gr(M) = \gr \gr_0 M \hookrightarrow \gr \left((\gr_0\M)(U)\right)\]
of graded $\gr \left(\hdnt(U)/\pi\hdnt(U)\right)$-modules. 

Each short exact sequence $0 \to F_{i-1} (\gr_0 \M) \to F_i (\gr_0 \M) \to \gr_i (\gr_0 \M) \to 0$ of sheaves induces an embedding
\[\gr_i ((\gr_0 \M)(U)) \hookrightarrow (\gr_i (\gr_0 \M))(U).\]
Putting these together gives the required inclusion $\Gr(M) \hookrightarrow (\Gr(\M))(U)$ of graded $\Gr(\hdntK(U))$-modules.
\end{proof}

\subsection{Good double filtrations}
\label{GoodDoubleSheaf}
\begin{defn}
Let $\M$ be a sheaf of $\hdntK$-modules. We say that a double filtration $(F_0\M, F_\bullet \gr_0 \M)$ on $\M$ is \emph{good} if $F_0\M$ is a coherent $\hdnt$-module and $F_\bullet \gr_0 \M$ is a good filtration on $\gr_0 \M$ as a $\dnt$-module. 
\end{defn}

\begin{prop} Let $\M$ be a coherent $\hdntK$-module. 
\be \item $\M$ has at least one good double filtration $(F_0\M, F_\bullet \gr_0 \M)$.
\item If $(F_0\M, F_\bullet \gr_0 \M)$ is a good double filtration on $\M$ then for any $U \in \mathcal{S}$, $((F_0\M)(U), (F_\bullet \gr_0 \M)(U))$ is a good double filtration on $\M(U)$ and the map
\[\Gr(\M(U)) \hookrightarrow (\Gr\M)(U)\]
appearing in Lemma \ref{DoubleFiltSheaf}(b) is an isomorphism.
\ee
\end{prop}
\begin{proof} (a) By the proof of \cite[Lemma 3.4.3]{Berth}, we can find a coherent $\hdnt$-submodule $\N$ of $\M$ such that the natural map $\N_K \to \M$ is an isomorphism. Since $\N \cong \invlim \N / \pi^n \N$ by Lemma \ref{CohMcomplete}, $\N$ is an $R$-lattice in $\M$. Now $\N / \pi \N$ is a coherent $\hdnt / \pi \hdnt$-module and $\dnt$ surjects onto $\hdnt / \pi \hdnt$ by Proposition \ref{PiAdicC}(d), so $\N / \pi \N$ is also a coherent $\dnt$-module. We can therefore find a good filtration $F_\bullet (\N / \pi \N)$ on $\N / \pi \N$ by Lemma \ref{GoodFilt}(c). Thus $(\N, F_\bullet(\N / \pi \N))$ is a good double filtration on $\M$.

(b) Since $\F_0\M$ is a coherent $\hdnt$-module, $H^1(U, \F_0\M) = 0$ by Corollary \ref{Hdnt}, so the inclusion $\gr_0 (\M(U))  \hookrightarrow (\gr_0 \M)(U)$ is an isomorphism. The cokernel of the inclusion
\[\gr_i ((\gr_0 \M)(U)) \hookrightarrow (\gr_i (\gr_0 \M))(U)\]
is precisely $H^1(U, F_{i-1}\gr_0 \M)$ which is zero because $U$ is affine and Noetherian and because $F_{i-1}\gr_0 \M$ is a coherent $\O$-module by Lemma \ref{GoodFilt}(d). Therefore 
\[\Gr(\M(U)) \longrightarrow (\Gr\M)(U)\]
is an isomorphism. Now $\Gr\M$ is a coherent $\gr \dnt$-module by assumption and $U$ is affine and Noetherian, so $(\Gr \M)(U)$ is a finitely generated $(\gr \dnt)(U)$-module. But $\gr \dnt \cong \tau_\ast \O_{\Ex{T^\ast X}}$ by Proposition \ref{EnhDn}(c) and $(\Gr \M)(U)$ is killed by $\pi$, so $\Gr (\M(U)) \cong (\Gr \M)(U)$ is a finitely generated $\Gr(\hdntK(U))$-module by Proposition \ref{PiAdicC}(a).
\end{proof}

\subsection{Theorem} \label{hdntKaff} $\mathcal{S}$ is coherently $\hdntK$-affine.

\begin{proof} Since $\hdntK$ is coherent by Proposition \ref{PiAdicC}(d), once again it suffices to consider the case where $X\in\mathcal{S}$ and show that $X$ is coherently $\hdntK$-affine.

By Proposition \ref{GoodDoubleSheaf}(a), $\M$ has an $R$-lattice $F_0\M$ that is coherent as a $\hdnt$-module. By Corollary \ref{Hdnt}, $F_0\M$ is $\Gamma$-acyclic, $F_0\M(X)$ is finitely generated as a $\hdnt(X)$-module and $F_0\M$ is generated by global sections. The result follows easily from this together with the fact that $H^i(X,\M)\cong H^i(X,F_0\M)\otimes_R K$ for all $i\ge 0$. This latter is true because $X$ is a Noetherian space, so cohomology commutes with direct limits by \cite[Proposition III.2.9]{Hart}. \end{proof}

\subsection{Characteristic varieties}\label{CharVar} 
Let $\M$ be a coherent $\hdntK$-module. Pick a good double filtration on $\M$ using Proposition \ref{GoodDoubleSheaf}; then $\Gr(\M) = \gr (\gr_0\M)$ is a coherent $\gr \dnt$-module and $\gr \dnt \cong \tau_\ast \O_{\Ex{T^\ast X}}$ by Proposition \ref{EnhDn}(c).

\begin{defn} The \emph{characteristic variety} of $\M$ is the support of $\Gr(\M)$ regarded as a sheaf on the enhanced cotangent bundle $\Ex{T^\ast X}$. More precisely, writing $\widetilde{\Gr(\M)}$ for the $\O_{\Ex{T^\ast X}}$-module $\O_{\Ex{T^\ast X}}\otimes_{\tau^{-1}\tau_\ast \mathcal{O}_{\Ex{T^\ast X}}}\tau^{-1}(\Gr(\M))$, we define

\[\Ch(\M) := \Supp(\widetilde{\Gr(\M)}) \subseteq \Ex{T^\ast X}.\]
\end{defn}
Since $\widetilde{ \Gr(\M)}$ is a coherent sheaf on $T^\ast X$, $\Ch(\M)$ is closed. $\Ch(\M)$ is actually contained in the special fibre $\Ex{T^\ast X_k}$ of the enhanced cotangent bundle because $\Gr(\M)$ is annihilated by $\pi$.

\begin{lem} Let $\M$ be a coherent $\hdntK$-module and let $U \in \mathcal{S}$. Then $\Ch(\M) \cap \Ex{T^\ast U} = \Ch(\M(U))$.
\end{lem}
\begin{proof} By Proposition \ref{GoodDoubleSheaf}(b), the good double filtration on $\M$ induces a good double filtration on $\M(U)$ such that $\Gr \M(U) \cong (\Gr\M)(U)$. Then
\[\begin{array}{lllll} \Ch(\M(U)) &=& \Supp(\Gr(\M(U))) &=& \Supp( \widetilde{ (\Gr\M)(U) }) \\ &=& \Supp( \widetilde{\Gr \M}_{|\Ex{T^\ast U}} ) &=& \Ch(\M) \cap \widetilde{T^\ast U}\end{array}\]
as required.
\end{proof}

\begin{cor} \hfill \be 
\item $\Ch(\M)$ does not depend on the choice of good double filtration on $\M$.

\item If $0\to \mathcal{L}\to \M\to \mathcal{N}\to 0$ is a short exact sequence of coherent $\hdntK$-modules then $\Ch(\M)=\Ch(\mathcal{L})\cup \Ch(\mathcal{N})$.  
\ee
\end{cor}
\begin{proof} By the Lemma, $\Ch(\M) = \bigcup_{U \in \mathcal{S}} \Ch(\M(U))$ only depends on the local sections $\M(U)$ of $\M$. It follows from Theorem \ref{hdntKaff} that $\Gamma(U,-)$ is exact on coherent $\hdntK$-modules, and now both parts follow from Proposition \ref{CharVarAff}. \end{proof}

\subsection{Coherent cohomology}\label{CohCom}
We now specialise to the setting of $\S \ref{AlgGps}$, so that $X$ is the flag variety and $\xi : \Ex{\Fl} \to \Fl$ is the locally trivial $\mb{H}$-torsor from the basic affine space to the flag variety. Recall the enhanced moment map $\beta : \Ex{T^\ast \Fl} \to \fr{g}^\ast$ from $\S$\ref{EnhMom} and the enhanced cotangent bundle $\tau : \Ex{T^\ast \Fl} \to \Fl$ from $\S$\ref{AlgGps}:
\[\xymatrix{ \Ex{\Fl}\ar[dr]_{\xi} & & \Ex{T^\ast \Fl} \ar[dl]^{\tau} \ar[dr]_{\beta} & \\
                           & \Fl & & \fr{g}^\ast
}\]

Let $\M$ be a $\dnt$-module. Then $H^\bullet(\Fl, \M) := \bigoplus_{i\geq 0} H^i(\Fl, \M)$ is naturally a $\Gamma(\Fl, \dnt)$-module, and we can view it as a $\Un$-module via the natural map $\varphi_n : \Un \to \Gamma(\Fl,\Ex{\D})_n \to \Gamma(\Fl,\dnt)$ obtained by applying the deformation functor to the morphism $U(\varphi) : U(\fr{g}) \to \Ex{\D}$ of deformable $R$-algebras defined in $\S$\ref{QMM} and then sheafifying.

\begin{prop} Let $\M$ be a coherent $\dnt$-module and let $F_\bullet \M$ be a good filtration on $\M$. Then
\be \item $H^\bullet(\Fl, \gr \M)$ is a finitely generated $\O(\fr{g}^\ast)$-module, and
\item $H^\bullet(\Fl,\M)$ is a finitely generated $\Un$-module.
\ee \end{prop}
\begin{proof} (a) The morphism $\tau : \Ex{T^\ast \Fl} \to \Fl$ is affine. By definition, $\gr \M$ is a coherent $\gr \dnt \cong \tau_\ast \O_{\Ex{T^\ast \Fl}}$-module. Thus if we define $\Ex{\gr \M} = \mathcal{O}_{\Ex{T^\ast \Fl}}\otimes_{\tau^{-1}\tau_\ast\mathcal{O}_{\Ex{T^\ast\Fl}}}\tau^{-1}(\gr \M)$, then $\Ex{\gr \M}$ is a coherent $\O_{\Ex{T^\ast \Fl}}$-module such that $\tau_\ast \Ex{\gr \M}=\gr \M$. Next,
\[H^\bullet(\Fl, \gr \M) \cong H^\bullet(\Ex{T^\ast \Fl}, \Ex{\gr \M})\] by  \cite[Corollaire III.1.3.3]{EGAIII}. Since the enhanced moment map $\beta : \Ex{T^\ast \Fl} \to \fr{g}^\ast$ is projective by Proposition \ref{EnhMom}(b),  $R^\bullet \beta_\ast \Ex{\gr \M}$ is a coherent $\O_{\fr{g}^\ast}$-module by \cite[Theorem III.8.8(b)]{Hart}. Since $\fr{g}^\ast$ is affine, $\Gamma(\fr{g}^\ast, R^\bullet \beta_\ast \Ex{\gr \M}) \cong H^\bullet(\Ex{T^\ast \Fl}, \Ex{\gr \M})$ by \cite[Proposition III.8.5]{Hart} and therefore $H^\bullet(\Fl, \gr \M)$ is a finitely generated $\O(\fr{g}^\ast)$-module.

(b) The filtration $F_\bullet \M$ on $\M$ induces a filtration on any \v{C}ech complex computing $H^\bullet (\Fl, \M)$ and hence a convergent third octant cohomological spectral sequence
\[E_1^{ij}= H^{i+j}(\Fl,\gr_{-i}\M)\Rightarrow H^{i+j}(\Fl,\M) .\]
This spectral sequence implies that $\gr H^\bullet (\Fl, \M)$ is a subquotient of $H^\bullet(\Fl, \gr \M)$ as a $\gr \Un$-module. But $\gr \Un$ is Noetherian because it is isomorphic to $S(\fr{g})$ by Lemma \ref{Defs}, and the result follows.
\end{proof}

\subsection{Characteristic varieties of global sections}\label{Bry1}
The set of global sections $\Gamma(\Fl, \M)$ of a $\hdntK$-module $\M$ is a module over $A := \hugnK$ via the completed deformed ring homomorphism $\widehat{\varphi_{n,K}} : A \to \Gamma(\Fl, \hdntK)$, and $A$ is an almost commutative affinoid $K$-algebra by Proposition \ref{acsa}. The next result is an analogue of \cite[Lemma 1.6(c)]{BoBrIII}.

\begin{prop} Let $\M$ be a coherent $\hdntK$-module and let $M = \Gamma(\Fl,\M)$. Then $M$ is a finitely generated $A$-module and
\[ \Ch(M) \subseteq \beta(\Ch(\M)).\]
\end{prop}
\begin{proof} Choose a good double filtration on $\M$. Then by Lemma \ref{DoubleFiltSheaf} it induces a double filtration on $M$ and there is a natural embedding
\[\Gr(M) \hookrightarrow \Gamma(\Fl, \Gr(\M))\]
of $\Gr(A) \cong S(\fr{g}_k)$-modules. By Proposition \ref{CohCom}(a), $\Gamma(\Fl,\Gr(\M))$ is finitely generated over $S(\fr{g}_k)$. Since $S(\fr{g}_k)$ is Noetherian, $\Gr(M)$ is a finitely generated $\Gr(A)$-module and hence the double filtration on $M$ is good. Hence $M$ is finitely generated over $A$ by Lemma \ref{GoodDouble} and
\[\Ch(M) = \Supp(\Gr(M)) \subseteq \Supp(\Gamma(\Fl, \Gr(\M))).\]
Let $\F :=  \Ex{\Gr \M} := \O_{\Ex{T^\ast \Fl}}\otimes_{\tau^{-1}\tau_\ast\O_{\Ex{T^\ast \Fl}}} \Gr(\M)$ so that $\Ch(\M) = \Supp(\F)$. Since $\Gr(\M)$ is already a $\tau_\ast \O_{\Ex{T^\ast \Fl}}-$module, the natural map $\Gr(\M) \to \tau_\ast \F$ is an isomorphism. But $\tau : \Ex{T^\ast \Fl} \to \Fl$ is an affine morphism,  so
\[\Gamma(\Fl, \Gr(\M)) = \Gamma(\Fl, \tau_\ast \F) = \Gamma(\Ex{T^\ast \Fl}, \F) = \Gamma(\fr{g}^\ast, \beta_\ast \F).\]
Whenever $f : X \to Y$ is a continuous map between topological spaces and $\mathcal{G}$ is a sheaf of abelian groups on $X$, we have $\Supp(f_\ast \mathcal{G}) \subseteq \overline{f(\Supp(\mathcal{G}))}$. Hence
\[\Ch(M) \subseteq \Supp(\Gamma(\Fl, \Gr(\M))) = \Supp(\Gamma(\fr{g}^\ast,\beta_\ast \F)) = \Supp(\beta_\ast \F) \subseteq \overline{\beta(\Ch(\M))}\]
because $\fr{g}^\ast$ is an affine scheme. Finally $\Ch(\M)$ is a closed subscheme of $\Ex{\T^\ast \Fl}$ since $\F$ is coherent, and $\beta$ is a proper morphism by Proposition \ref{EnhMom} and \cite[Theorem II.4.9]{Hart}, so $\beta(\Ch(\M))$ is closed and the result follows. \end{proof}

\subsection{The localisation functor}\label{EnhLoc}
Let $M$ be an $A := \hugnK$-module. For any open subscheme $U \subseteq \Fl$, the algebra $\hdntK(U)$ is an $A$-module via ring homomorphism $\widehat{\varphi_{n,K}} : \hugnK \to \hdntK(U)$. We can therefore define the \emph{localisation functor}
\[ \Loc :  \md{\hugnK} \to \md{\hdntK}\]
by letting $\Loc(M)$ be the sheafification of the presheaf $U \mapsto \hdntK(U) \otimes_A M$ on $\Fl$. This functor is right exact because it is the left adjoint to the global sections functor $\Gamma$.

\begin{prop} Let $M$ be a finitely generated $A$-module and let $\M = \Loc(M)$. Then
\be \item $\M$ is a coherent $\hdntK$-module.
\item If $U \in \mathcal{S}$ then $\M(U) = \hdntK(U) \otimes_A M$.
\ee\end{prop}
\begin{proof}
(a) Since $A$ is an almost commutative affinoid $K$-algebra, it is Noetherian by Lemma \ref{GoodDouble}(b). Hence we can find a presentation $F_1 \to F_0 \to M \to 0$ of $M$, where $F_0$ and $F_1$ are finitely generated free $A$-modules. Since $\Loc$ is right exact, we obtain a presentation $\Loc(F_1) \to \Loc(F_0) \to \M \to 0$. Since $\Loc(A) \cong \hdntK$ is coherent by Proposition \ref{PiAdicC}(c), we deduce that $\M$ is also coherent.

(b) Let $N := \hdntK(U) \otimes_A M$, a finitely generated $\hdntK(U)$-module. The restriction of $\M$ to $U$ is isomorphic to $\hdntK_{|U} \otimes_{\hdntK(U)} N$, and it follows from Theorem \ref{hdntKaff} that the sections of this sheaf over $U$  are simply $N$.\end{proof}

\subsection{The characteristic variety of $\Loc(M)$}
\label{CharLocM}
\begin{lem} Let $M$ be a finitely generated $A := \hugnK$-module, let $U \in \mathcal{S}$ and let $B := \hdntK(U)$. Then every good double filtration on $M$ induces a good double filtration on $N := B \otimes_A M$ and a natural surjection
\[\Gr(B) \otimes_{\Gr(A)} \Gr(M) \twoheadrightarrow \Gr(N).\]
\end{lem}
\begin{proof}
Let $F_0 N$ be the image of $F_0B \otimes_{F_0A} F_0M$ in $N$. Then $F_0 N \cdot K = N$ and $F_0N$ is a finitely generated $F_0B$-submodule of $N$, so $F_0N$ is an $R$-lattice in $N$ by Proposition \ref{GoodDouble}(a). Moreover the $\gr_0 B$-module $\gr_0 N = F_0N / \pi F_0N$ is a quotient of $\gr_0 B \otimes_{\gr_0 A} \gr_0 M$.

Equip $\gr_0 N$ with the quotient filtration induced from the tensor filtration on $\gr_0 B \otimes_{\gr_0 A} \gr_0 M$. Then $\Gr(N)$ is a quotient of $\Gr(B) \otimes_{\Gr(A)} \Gr(M)$, which implies that the filtration on $\gr_0 N$ as a $\gr_0 B$-module is good. \end{proof}

The following result is an analogue of \cite[Proposition 1.8]{BoBrIII}.

\begin{prop} Let $M$ be a finitely generated $A$-module and let $\M = \Loc(M)$. Then
\[\beta(\Ch(\M)) \subseteq \Ch(M).\]
\end{prop}
\begin{proof} Choose a good double filtration on $M$ using Proposition \ref{GoodDouble}(b). Since $\M$ is a coherent $\hdntK$-module by Proposition \ref{EnhLoc}(a),
\[\Ch(\M) = \bigcup_{U \in \mathcal{S}} \Ch(\M(U))\]
by Lemma \ref{CharVar}. Let $U \in \mathcal{S}$ and let $B = \hdntK(U)$. Then $N := \M(U)$ is equal to $B \otimes_A M$ by Proposition \ref{EnhLoc}(b), and $N$ carries a good double filtration such that $\Gr(N)$ is a quotient of $\Gr(B)\otimes_{\Gr(A)} \Gr(M)$ by the Lemma. Therefore
\[\Ch(N) = \Supp(\Gr(N)) \subseteq \Supp(\Gr(B)\otimes_{\Gr(A)} \Gr(M)) \subseteq \beta^{-1} (\Supp(\Gr(M)))\]
by \cite[\S 0.4.3.1]{EGAI}, and the result follows.
\end{proof}

\section{The Beilinson-Bernstein theorem for $\hdnlK$}
\label{BB}
We continue with the notation from $\S\ref{AlgGps}$ in this section. $R$ will denote a complete discrete valuation ring of arbitrary characteristic, unless stated otherwise.
\subsection{Serre twists}
\label{grDVanishing}
The scheme $\Fl$ is projective over $R$ by \cite[\S II.1.8]{Jantzen}. Fix an embedding $i : \Fl \hookrightarrow \mathbb{P}^N_R$ into some projective space over $R$ and let $\mathcal{L} := i^\ast \O(1)$ be the corresponding very ample invertible sheaf on $\Fl$.

For any $\O_\Fl$-module $\M$ and any $s \in \mathbb{Z}$, we let $\M(s) :=\M \otimes_{\O_{\Fl}}  \mathcal{L}^{\otimes s}$ denote the Serre twist of $\M$.

\begin{lem} Let $\F$ be a coherent $\O_\Fl$-module. Then there exists an integer $u \geq 0$ such that $\F(s)$ is generated by its global sections and is $\Gamma$-acyclic whenever $s \geq u$.\end{lem}
\begin{proof} Because $\mathcal{L}$ is very ample, this follows from Serre's Theorems \cite[Theorem II.5.17 and Theorem III.5.2(b)]{Hart} \end{proof}

We will first study coherent modules over the sheaf of algebras
\[ \C := \Sym_{\O_\Fl}(\O_\Fl \otimes \fr{g}) \cong \O_\Fl \otimes \Sy{g}.\]
\begin{prop} Let $\M$ be a coherent $\C$-module. Then there exists an integer $t$ such that $\M(s)$ is $\Gamma$-acyclic whenever $s \geq t$.

\end{prop}
\begin{proof}
Consider the following diagram of schemes over $\Spec(R)$:\[ \xymatrix{ & \mathcal{B}\times\fr{g}^\ast\ar[dl]_{p}\ar[dr]^{q} & \\ \mathcal{B}\ar[dr] & & \fr{g}^\ast\ar[dl] \\ & \Spec(R) . & \\   }\] Since $\mathcal{C} = p_\ast\mathcal{O}_{\Fl\times \fr{g}^\ast}$, \cite[Proposition II.1.4.3 and Corollaire II.1.4.5]{EGAII} show that there is a coherent $\mathcal{O}_{\Fl\times \fr{g}^\ast}$-module $\mathcal{N}$ such that $\mathcal{M} \cong p_\ast \mathcal{N}$.

Now by \cite[Proposition II.4.6.13(i)]{EGAII}, the invertible $\mathcal{O}_{\fr{g}^\ast}$-module $\mathcal{O}_{\fr{g}^\ast}$ is ample relative to $\id_{\fr{g}^\ast}$, so \cite[Proposition II.4.6.13(iv)]{EGAII} gives that $p^\ast\mathcal{L}=\mathcal{L}\otimes_R\mathcal{O}_{\fr{g}^\ast}$ is ample relative to $q$. Thus it follows from the relative version of Serre's Theorem, \cite[Th\'eor\`eme III.2.2.1(ii)]{EGAIII}, that $\mathcal{N}(s):=\mathcal{N} \otimes_{\mathcal{O}_{\mathcal{B}\times\fr{g}^\ast}} \left(p^\ast\mathcal{L}\right)^{\otimes s}$ is $q_\ast$-acyclic for sufficiently large values of $s$.

Since $\fr{g}^\ast$ is affine, \cite[Proposition III.1.4.14]{EGAIII} implies that \[ H^i(\mathcal{B}\times \fr{g}^\ast,\mathcal{N}(s))\cong \Gamma(\fr{g}^\ast,R^iq_\ast\mathcal{N}(s))=0\] for $i>0$ and $s$ sufficiently large.  On the other hand, because $p$ is an affine morphism, \cite[Corollaire III.1.3.3]{EGAIII} tells us that \[ H^i(\mathcal{B}\times \fr{g}^\ast,\mathcal{N}(s))\cong H^i(\mathcal{B},p_\ast\mathcal{N}(s))\] for all $i$ and $s$. It remains to show that whenever $\mathcal{F}$ is a coherent $\mathcal{O}_{\Fl\times \fr{g}^\ast}$-module,  \[p_\ast(\mathcal{F} \otimes_{\mathcal{O}_{\mathcal{B}\times\fr{g}^\ast}}p^\ast\mathcal{L})\cong p_\ast\mathcal{F} \otimes_{\mathcal{O}_\mathcal{B}} \mathcal{L}\]
as then $p_\ast(\mathcal{N}(s))\cong \mathcal{M}(s)$ for all $s\geq 0$. But this follows from \cite[\S 0.5.4.10]{EGAI}. \end{proof}

\begin{cor} Let $\M$ be a coherent $\dnt$-module. Then there exists an integer $t$ such that $\M(s)$ is $\Gamma$-acyclic whenever $s \geq t$.
\end{cor}
\begin{proof} Choose a good filtration on $\M$ using Lemma \ref{GoodFilt}(c); then $\gr \M$ is a coherent $\gr \dnt$-module. By Proposition \ref{EnhMom}(a), the natural map $\varphi : \O_\Fl \otimes \fr{g} \to \Ex{\T}_\Fl$ is surjective, so $\C$ surjects onto $\gr(\dnt) \cong \Sym_{\O_\Fl} \Ex{\T}_{\Fl}$ by Proposition \ref{EnhDn}(c). So $\gr \M$ is a coherent $\C$-module and therefore there exists an integer $t$ such that
\[\gr (\M(s)) \cong (\gr \M)(s)\]
is $\Gamma$-acyclic for all $s \geq t$ by the Proposition. Because $\Fl$ is Noetherian, cohomology commutes with direct limits by \cite[Proposition III.2.9]{Hart}. Therefore each homogeneous component $\gr_i \M(s)$ is $\Gamma$-acyclic and each filtered piece $F_i \M(s)$ is also $\Gamma$-acyclic. Finally $\M(s)$ is the direct limit of the $F_i \M(s)$ and therefore is also $\Gamma$-acyclic, whenever $s \geq t$.\end{proof}

\subsection{The geometric translation functor}\label{GeomTr}
Let $\A$ be one of the sheaves $\dnt$, $\hdnt$ or $\hdntK$. For every integer $s$ we can consider the twisted sheaf
\[  \A^{(s)} := \mathcal{L}^{\otimes s} \otimes_{\O_\Fl} \A \otimes_{\O_\Fl} \mathcal{L}^{\otimes (-s)}.\]
As an $\O_\Fl$-module this sheaf is isomorphic to $\A$, but it is also naturally a sheaf of rings which \emph{a priori} is not isomorphic to $\A$. Of course, locally the sheaves of rings $\A^{(s)}$ and $\A$ \emph{are} isomorphic. We believe there to be a global isomorphism but we will not need it. 

Note that for every $\A$-module $\N$, the twisted sheaf $(s)\N: = \mathcal{L}^{\otimes s} \otimes_{\O_\Fl} \N$ is naturally an $\A^{(s)}$-module by contracting tensor products. If $\mathcal{L}$ is trivialisable on $U$ then $(s)\N|_U$ is isomorphic to $\N|_U$ as $\A|_U$-modules if we view $\A|_U$ as acting on $(s)\N|_U$ along a local isomorphism $\A|_U\to \A^{(s)}|_U$. 

We retain the notation $\mathcal{A}(s)$ to mean the left $\A$-module $\mathcal{A}\otimes_{\O_{\Fl}}\mathcal{L}^{\otimes s}$ with $\A$ acting on the left factor.

\begin{lem} Let $s \in \mathbb{Z}$. 
\be 
\item $\A^{(s)}$ is a coherent sheaf of rings. Moreover $\hdnt^{(s)}\cong \widehat{\dnt^{(s)}}$ as sheaves of rings.
\item If $\N$ is a coherent $\A$-module, then $(s)\N$ is a coherent $\A^{(s)}$-module. Also $\A(s)$ is a coherent $\A$-module. 
\item The functor $\N \mapsto (s)\N$ is exact.
\item If $\F$ is an $\O_{\Fl}$-module and $t,u \in \mathbb{Z}$, then
\[ (u)(\A^{(t)} \otimes_{\O_\Fl} \F) \cong \A^{(t+u)} \otimes_{\O_\Fl} (\F(u))\]
as $\A^{(t+u)}$-modules.
\ee
\end{lem}

\begin{proof}
(a)-(c) are all local properties that may be verified by working on a base that trivialises $\mathcal{L}$ together with the corresponding statements for the untwisted objects.

For (d) we observe that $\F(u)$ and $(u)\F$ are canonically isomorphic as $\mathcal{O}_\Fl$-modules and contract various tensor products. 
\end{proof}

\subsection{A family of generating objects}\label{Gen}
The following result is essentially \cite[Proposition 3.3(i)]{Noot1}, but we give the proof for the benefit of the reader.

\begin{thm}
The sheaves $\{\hdntK(s) : s \in \mathbb{Z}\}$ generate the category of coherent $\hdntK$-modules.
\end{thm}
\begin{proof} We set $\A = \dnt$ to aid legibility. Let $\M \in \coh(\hAK)$ and choose a good double filtration $(F_0 \M, F_\bullet \gr_0 \M)$ on $\M$ using Proposition \ref{GoodDoubleSheaf}. Then $\gr_0 \M$ is a coherent $\A$-module, so by Corollary \ref{grDVanishing} we can find an integer $t$ such that $(\gr_0 \M)(s)$ is $\Gamma$-acyclic for all $s \geq t$. Since $(\gr_0 \M)(s)\cong (s)\gr_0 \M$ as sheaves of $\O_{\Fl}$-modules we may deduce that $(s) \gr_0 \M$ is also $\Gamma$-acyclic for all  $s\geq t$.

Next, since $(t)(\gr_0 \M)$ is a coherent $\A^{(t)}$-module by Lemma \ref{GeomTr}(b), we can find a coherent $\O_{\Fl}$-submodule $\mathcal{F}$ of $(t)(\gr_0 \M)$ which generates it as an $\A^{(t)}$-module by the proof of Lemma \ref{GoodFilt}(b); this gives a surjection 
\[ \A^{(t)} \otimes_{\O_\Fl} \F \twoheadrightarrow (t)(\gr_0 \M)\]
of $\A^{(t)}$-modules. Twisting this by $\mathcal{L}^{ \otimes u}$ on the left and applying Lemma \ref{GeomTr}(d) gives surjections
\[ \A^{(t+u)} \otimes_{\O_\Fl} (\F(u)) \cong (u)(\A^{(t)} \otimes_{\O_\Fl} \F) \twoheadrightarrow (t+u)(\gr_0 \M)\]
of $\A^{(t+u)}$-modules for any integer $u$.

Now by Lemma \ref{grDVanishing}, there is a surjection $\O_\Fl^a \twoheadrightarrow \F(u)$ for some integers $a \geq 1$ and $u \geq 0$. We therefore obtain a surjection
\[ \A^{(t+u)} \otimes_{\O_\Fl} \O_\Fl^a \twoheadrightarrow \A^{(t+u)} \otimes_{\O_\Fl} (\F(u)) \twoheadrightarrow (t+u)(\gr_0 \M)\]
of $\A^{(t+u)}$-modules. Let $s := t + u$; then $(s)(\gr_0 \M)$ is $\Gamma$-acyclic and generated as a $\A^{(s)}$-module by its global sections.

Let $\N: = F_0 \M$, a coherent $\hA$-module and $\mathcal{K}:=(s)\N$ its twist. Since $\M$ has no $\pi$-torsion, for all $i \geq 0$ there is a short exact sequence of sheaves
\[0 \to \gr_0 \M \stackrel{\pi^i}{\longrightarrow} \N / \pi^{i+1} \N\to \N/\pi^i \N \to 0.\] Since also $H^1(\Fl, (s)(\gr_0 \M)) = 0$, twisting this sequence by $\mathcal{L}^{\otimes s}$ on the left and taking cohomology shows that the arrow
\[ \Gamma(\Fl, \mathcal{K} / \pi^{i+1} \mathcal{K}) \to \Gamma(\Fl, \mathcal{K} / \pi^i \mathcal{K})\]
is surjective for all $i \geq 0$. 

Thus a finite subset of $\Gamma(\Fl, (s)(\gr_0 \M))$ generating $(s)(\gr_0 \M)$ as a $\A^{(s)}$-module may be lifted inductively to elements $w_1,\ldots, w_a \in \invlim \Gamma(\Fl, \mathcal{K} / \pi^i \mathcal{K})$. 

Since $\mathcal{K}$ is a coherent $\hA^{(s)} \cong \widehat{\A^{(s)}}$-module, $\mathcal{K}(U)$ is $\pi$-adically complete for each $U \in \mathcal{S}$ by \cite[\S 3.2.3(v)]{Berth}, so the natural map $\mathcal{K} \to \invlim \mathcal{K}/ \pi^i \mathcal{K}$ is an isomorphism. Pulling back the $w_i$ along this isomorphism on global sections gives a finite collection of elements in $\Gamma(\Fl, \mathcal{K})$ that generate $\mathcal{K}$ as a $\hA^{(s)}$-module by Nakayama's Lemma. We therefore obtain a surjective map 
$\left(\hA^{(s)}\right)^a \twoheadrightarrow \mathcal{K}$ of $\hA^{(s)}$-modules, and by twisting back by $\mathcal{L}^{\otimes -s}$ on the left, a surjective map
\[ \left((-s)\hA^{(s)}\right)^a \twoheadrightarrow \N\]
of left $\hA$-modules. But $(-s)\hA^{(s)}\cong \hA(-s)$ as left $\hA$-modules by Lemma \ref{GeomTr}(d), so we obtain a surjective map $(\hdnt(-s))^a \twoheadrightarrow \N$ of $\hdnt$-modules, and after inverting $\pi$ a surjective map $(\hdntK(-s))^a  \twoheadrightarrow \M$ of $\hdntK$-modules, as required.
\end{proof}

\subsection{Twisted differential operators}\label{TDO} 
We can apply the deformation functor to the map $j : U(\fr{h}) \to \Ex{\D}$ defined in $\S\ref{EnhCB}$ to obtain a central embedding of the constant sheaf $U(\fr{h})_n$ into $\dnt$:
\[ U(\fr{h})_n \hookrightarrow \dnt.\]

Each linear functional $\lambda \in \Hom_R(\pi^n\fr{h}, R)$ extends to an $R$-algebra homomorphism $U(\fr{h})_n \tocong U(\pi^n \fr{h}) \to R$ and gives $R$ the structure of a $U(\fr{h})_n$-module which we denote by $R_\lambda$. \textbf{Until the end of $\S \ref{BB}$ we fix $\lambda \in \Hom_R(\pi^n\fr{h}, R)$.}

\begin{defn} The \emph{sheaf of deformed twisted differential operators} $\dnl$ on $\Fl$ with parameters $n,\lambda$ is the central reduction
 \[\dnl := \dnt \otimes_{U(\fr{h})_n} R_\lambda.\]	
We give $R_\lambda$ the trivial filtration $0 =: F_{-1}R_\lambda \subset R_\lambda =: F_0R_\lambda$ as a $U(\fr{h})_n$-module, and we view $\dnl$ as a sheaf of filtered $R$-algebras, equipped with the tensor filtration.

\end{defn}\begin{lem}\hfill 
\be 
\item Let $U \in \mathcal{S}$. Then $\left(\dnl\right)_{|U}$ is isomorphic to $(\D_n)_{|U}$ as a sheaf of filtered $R$-algebras.
\item $\dnl$ is a sheaf of deformable $R$-algebras.
\item There is an isomorphism of sheaves $\gr \dnl \cong \Sym_\O \T.$
\ee\end{lem}
\begin{proof} 
(a) Let $\mathcal{F}$ be a sheaf of deformable $R$-algebras on $\Fl$ and let $\F_n$ be the sheafification of the presheaf $V \mapsto \F(V)_n$. Since sheafification commutes with restriction, we see that $(\F_n)_{|U}$ is naturally isomorphic to $(\F_{|U})_n$.

Next, by Propositions \ref{DefTens} and \ref{RelEnv} there are \emph{sheaf} isomorphisms 
\[   (\D_{|U})_n \otimes U(\fr{h})_n \tocong \left(\D_{|U} \otimes U(\fr{h})\right)_n \tocong (\Ex{\D}_{|U})_n = \left(\Ex{\D}_n\right)_{|U} \]
which induce an isomorphism $\left(\dnl\right)_{|U}\tocong \left(\D_n\right)_{|U}$ of sheaves of filtered $R$-algebras. 

(b) Since $U$ is affine and Noetherian, it follows from $\S \ref{cryst}$ that $\gr (\D(U))$ is a locally free $\O(U)$-module. Now proceed as in the proof of Lemma \ref{EnhDn}(b).

(c) The universal property of sheafification together with \cite[\S I.6.13]{LVO} induce a morphism of sheaves 
\[\gr \Ex{\D}_n \otimes_{\gr U(\fr{h})_n} \gr R_\lambda \to \gr \dnl.\]
Now $\Sym_{\O}\Ex{\T} \cong \gr \Ex{\D}_n$ by Proposition \ref{EnhDn}(c) and this isomorphism sends the image of $S(\fr{h})$ in $\Sym_{\O}\Ex{\T}$ to $\gr U(\fr{h})_n \subseteq \gr \Ex{\D}_n$ by construction. Since $\gr R_\lambda$ is the trivial $\fr{h}$-module $R$ by definition and $\gr U(\fr{h})_n \cong S(\fr{h})$ by Lemma \ref{Defs}, we obtain a morphism of sheaves of graded $R$-algebras
\[ \Sym_{\O}\Ex{\T} \otimes_{S(\fr{h})} R \to \gr \dnl.\]
Now the short exact sequence $0 \to \fr{h} \otimes \O \stackrel{j \otimes 1}{\longrightarrow} \Ex{\T} \stackrel{\sigma}{\longrightarrow} \T \to 0$ of locally free sheaves on $\Fl$ from Lemma \ref{EnhCB} induces an isomorphism $\Sym_{\O}\Ex{\T} \otimes_{S(\fr{h})} R \stackrel{\Sym(\sigma)\otimes 1}{\longrightarrow} \Sym_{\O} \T$. We finally obtain a morphism of sheaves of graded $R$-algebras
\[ \Sym_{\O} \T \to \gr \dnl\]
which is seen to be an isomorphism over any $U \in \mathcal{S}$ by part (a).
\end{proof}

\subsection{Completions and central reduction}
\label{ComplCent}
The following elementary Lemma will be useful in what follows. Recall that if $M$ is an $R$-module then $\widehat{M}$ denotes its $\pi$-adic completion.

\begin{lem} Let $B \to A$ be a map of Noetherian $R$-algebras and let $M$ be a finitely generated $B$-module. Then $\widehat{A} \otimes_{\widehat{B}} \widehat{M}$ is $\pi$-adically complete, and there is a natural isomorphism of $R$-modules
\[\psi_M : \widehat{ A \otimes_B M } \longrightarrow \widehat{A} \otimes_{\widehat{B}} \widehat{M}.\]
If $B \to A$ is a central embedding and $M$ is a cyclic $B$-module, then $\psi_M$ is an $R$-algebra isomorphism.
\end{lem}
\begin{proof} The first part follows from \cite[\S 3.2.3(iii), (v)]{Berth}. Hence the natural map $A \otimes_B M \to \widehat{A} \otimes_{\widehat{B}} \widehat{M}$ extends to an $R$-linear map $\psi_M : \widehat{ A \otimes_B M } \longrightarrow \widehat{A} \otimes_{\widehat{B}} \widehat{M}$. Because $A$ and $B$ are Noetherian, it follows from \cite[\S 3.2.3(ii)]{Berth} that the functors $\widehat{A \otimes_B -}$ and $\widehat{A} \otimes_{\widehat{B}} -$ are right exact. Since $\psi$ is a natural transformation between these functors such that $\psi_B$ is an isomorphism, we can pick a presentation for $M$ by finitely generated free $B$-modules, apply both functors to this presentation and invoke the Five Lemma to deduce that $\psi_M$ is an isomorphism.

The last statement is now clear.\end{proof}

\begin{defn} Let $\hdnl := \invlim \dnl / \pi^a \dnl$ be the $\pi$-adic completion of $\dnl$ and let $\hdnlK := \hdnl \otimes_R K$.
\end{defn}

Since the discrete valuation ring $R$ is already $\pi$-adically complete by assumption, $R_\lambda$ is already a $\huhn$-module; we will denote the $\huhnK$-module $R_\lambda \otimes_R K$ by $K_\lambda$.
\begin{prop}  \hfill
\be\item If $U \in \mathcal{S}$ then $\hdnlK(U) \cong \h{\D(U)_{n,K}}$ is an almost commutative affinoid $K$-algebra and $\Gr(\hdnlK(U)) \cong \O(T^\ast U_k)$.
\item $\hdntK$ is a sheaf of complete doubly filtered $K$-algebras.
\item There is an isomorphism $\hdnlK\tocong\hdntK \otimes_{\huhnK} K_\lambda $ of sheaves of complete doubly filtered $K$-algebras.
\item The sheaf $\hdnlK$ is coherent.
\ee\end{prop}
\begin{proof} In view of Lemma \ref{TDO}, parts (a) and (b) follow from the proof of Proposition \ref{PiAdicC} after making appropriate changes.

(c) For each $U \in \mathcal{S}$, $\uhn \to \dnt(U)$ is a map of Noetherian $R$-algebras. Since $R_\lambda$ is $\pi$-adically complete, there is thus an isomorphism of complete $R$-algebras
\[ \psi_U : \widehat{\dnt(U) \otimes_{\uhn} R_\lambda} \tocong \widehat{ \dnt(U) }\otimes_{\huhn} R_\lambda  \] 
by the Lemma. The reduction of $\psi_U$ mod $\pi$ is a morphism of filtered algebras and the family of morphisms $(\psi_U)_{U\in \mathcal{S}}$ is compatible with restriction, so induces the required isomorphism $\psi : \hdnlK\tocong\hdntK \otimes_{\huhnK} K_\lambda$ of sheaves of complete doubly filtered $K$-algebras.

(d) Since $\hdnlK$ is a quotient of $\hdntK$, this follows from Proposition \ref{PiAdicC}(c).
\end{proof}

We will henceforth identify $\coh(\hdnlK)$ with the full subcategory of $\coh(\hdntK)$ consisting of sheaves such that $\pi^n \fr{h}$ acts via the character $\lambda$.

\subsection{The generic fibre}\label{UnComplete}
The sheaf $\dlK := \dnl \otimes_R K$ does not depend on the deformation parameter $n$. Its restriction to $\Fl_K$ is naturally isomorphic to a classical sheaf of twisted differential operators in the sense of Beilinson and Bernstein \cite{BB}.

The next result allows us to apply the classical theorem of Beilinson--Bernstein to those $\hdnlK$-modules that `can be uncompleted'.

\begin{thm} Let $\M$ be a coherent $\dnl$-module.
\be
\item If $\M_K$ is generated by global sections as a $\dlK$-module, then $\hM_K$ is generated by global sections as a $\hdnlK$-module.
\item If $\M_K$ is $\Gamma$-acyclic, then so is $\hM_K$.
\ee
\end{thm}
\begin{proof} Note that  $H^i(\Fl, \M_K) \cong H^i(\Fl, \M)\otimes_R K$ for all $i$ by \cite[Proposition III.2.9]{Hart}.
(a) Let $v_1,\ldots, v_a \in \Gamma(\Fl, \M_K)$ generate $\M_K$ as a $\dlK$-module; by clearing denominators we may assume that these generators all lie in $\Gamma(\Fl, \M)$. Let $\alpha : (\dnl)^a \to \M$ be the map of $\dnl$-modules defined by these global sections; then $\mathcal{C} := \coker(\alpha)$ is coherent and $\mathcal{C}_K = 0$ by assumption. Thus we have an exact sequence
\[ (\dnl)^a \to \M \to \mathcal{C} \to 0\]
in $\coh(\dnl)$. By Lemma \ref{GoodFilt}(a), the functor $\Gamma(U, -)$ is exact on $\coh(\dnl)$ for all $U \in \mathcal{S}$. Since $\dnl(U)$ is Noetherian by Lemma \ref{TDO} and since $\pi \in \dnl(U)$ is central, the functor of $\pi$-adic completion is exact on finitely generated $\dnl(U)$-modules by \cite[\S 3.2.3(ii)]{Berth}. Hence the sequence
\[ \widehat{(\dnl)^a(U)} \to \widehat{\M(U)} \to \widehat{\mathcal{C}(U)} \to 0\]
is exact for all $U \in \mathcal{S}$. Hence the sequence of sheaves
\[ (\hdnl)^a \stackrel{\widehat{\alpha}_K}{\longrightarrow} \hM \to \widehat{\mathcal{C}} \to 0\]
is exact. Now $\mathcal{C}$ is $\pi$-torsion and coherent as a $\dnl$-module, so $\pi^m \mathcal{C} = 0$ for some integer $m$ since $\Fl$ is quasi-compact. Hence $\pi^m\widehat{\mathcal{C}} = 0$ also and therefore the morphism $\widehat{\alpha}_K : (\hdnlK)^a \to \hM_K$ is surjective.

(b) By Proposition \ref{CohCom}, $H^i(\Fl, \M)$ is a finitely generated $\Un$-module. On the other hand, it is $\pi$-torsion when $i > 0$ by assumption. Since $\pi$ is central in $\Un$, we deduce that there exists an integer $m$ such that $\pi^m H^i(\Fl,\M) = 0$ for all $i > 0$. 

Now for all $a,b\ge 0$ there is a commutative diagram  \[\xymatrix{ 0 \ar[r] & \M \ar[r]^{\pi^{a+b}}\ar[d]_{\pi^a} & \M \ar[r]\ar[d]_{\id} & \M/\pi^{a+b}\M \ar[r]\ar[d]^{\tau_{a,b}}& 0 \\
                                                                                                              0 \ar[r] & \M  \ar[r]^{\pi^b} & \M \ar[r] & \M/\pi^b\M  \ar[r] & 0
}\]  of sheaves of coherent $\dnl$-modules with exact rows. This induces a commutative diagram on cohomology  \[ \xymatrix{ H^i(\Fl,\M) \ar[r]^{\pi^{a+b}}\ar[d]_{\pi^a} & H^i(\Fl,\M) \ar[d]_{\id}\ar[r] & H^i(\Fl,\M/\pi^{a+b}\M) \ar[r]\ar[d]^{H^i(\tau_{a,b})} & H^{i+1}(\Fl,\M) \ar[d]_{\pi^a}  \\ 
	                      H^i(\Fl,\M) \ar[r]^{\pi^b} & H^i(\Fl,\M)  \ar[r] & H^i(\Fl,\M/\pi^{b}\M) \ar[r] & H^{i+1}(\Fl,\M)  } \] with exact rows for each $i\ge 0$.  If $a\ge m$ then the last vertical arrow is zero, since $\pi^mH^{i+1}(\Fl,\M)=0$. Thus the image of $H^i(\tau_{a,b})$ is the image of $H^i(\Fl,\M)$ in $H^i(\Fl,\M/\pi^{b}\M)$.

It follows that the projective system $H^i(\Fl,\M/\pi^b\M)$ satisfies the Mittag-Leffler condition for each $i\ge 0$ and so by \cite[Proposition 0.13.3.1]{EGAIII}  together with Lemma \ref{GoodFilt} that $H^i(\Fl,\hM)\cong \invlim H^i(\Fl,\M/\pi^b\M)$ for all $i>0$. Taking the projective limit of the columns of the cohomology diagram for $i>0$ and using the fact that the maps in the first and last columns are zero for $a\ge m$ we then obtain isomorphisms $H^i(\Fl,\M)\cong H^i(\Fl,\hM)$, whence \[ H^i(\Fl, \hM_K) \cong H^i(\Fl, \hM)\otimes_R K \cong H^i(\Fl, \M)\otimes_R K = 0\]
whenever $i > 0$, as claimed.
\end{proof}

\subsection{Integral and dominant weights}
\label{BBVanishing}

We assume from now on that $\mb{G}$ is semisimple and simply-connected, and that $K$ is a field of characteristic zero.

Let $h_1,\ldots, h_l \in \fr{h}$ be the simple coroots corresponding to the simple roots in $\fr{h}^\ast_K$ given by the adjoint action of $\mb{H}$ on $\fr{g}/\fr{b}$, let $\omega_1, \ldots, \omega_l \in \fr{h}^\ast_K$ be the corresponding system of fundamental weights and let $\rho = \omega_1 + \ldots + \omega_l$. Thus $\omega_i(h_j) = \delta_{ij}$ for all $i,j$ and any $\mu \in \fr{h}^\ast_K$ can be written in the form $\mu = \sum_{i=1}^l \mu(h_i) \omega_i$. Since $K$ is a field of characteristic zero, $\fr{h}^\ast_K$ contains an isomorphic copy $\mathbb{Z}\omega_1 \oplus \cdots \oplus \mathbb{Z} \omega_l$ of the weight lattice $\Lambda$ of the group $\mb{H}$. Since $\fr{h}$ is spanned by the $h_i$ over $R$ --- see \cite[\S II.1.6, \S II.1.11]{Jantzen} --- our space of twists $\Hom_R(\pi^n\fr{h}, R)$ can be naturally identified with $\pi^{-n} R \otimes_{\mathbb{Z}} \Lambda$.

Recall that a weight $\mu \in \fr{h}_K^\ast$ is said to be \emph{integral} if $\mu \in \Lambda \subseteq \fr{h}^\ast_K$ or equivalently $\mu(h_i) \in \mathbb{Z}$ for all $i$. An integral weight $\mu$ is said to be \emph{dominant} if $\mu(h_i) \geq 0$ for all $i$. Following \cite{BB2} we extend this notion to non-integral weights as follows: an arbitrary weight $\mu \in \fr{h}_K^\ast$ is \emph{dominant} if $\mu(h)\notin \{-1,-2,-3,\cdots \}$ for any positive coroot $h\in \fr{h}$. Finally, we will say that $\lambda$ is \emph{$\rho$-dominant} if $\lambda + \rho$ is dominant.

\begin{thm} If $\lambda$ is $\rho$-dominant, then $H^i(\Fl, \hdnlK(s)) = 0$ for all $i \geq 1$ and all integers $s$.
\end{thm}
\begin{proof} Apply Theorem \ref{UnComplete}(b) and paragraph (iii) of the proof of \cite[Th\'eor\`eme Principal]{BB}, noting that our sheaf $\D^\lambda_K$ is their sheaf $\D_{\lambda+\rho}$. Our ground field $K$ is not algebraically closed, but this part of the proof of the Beilinson--Bernstein Theorem does not require this assumption.
\end{proof}

\begin{cor}If $\lambda$ is $\rho$-dominant, every coherent $\hdnlK$-module is $\Gamma$-acyclic.
\end{cor}
\begin{proof} Let $\M \in \coh(\hdnlK)$ and let $d = \dim \Fl$. Using Theorem \ref{Gen}, choose a resolution
\[ \cdots \to\mathcal{P}_d \stackrel{f_d}{\to} \mathcal{P}_{d-1} \stackrel{f_{d-1}}{\to} \cdots \stackrel{f_1}{\to} \mathcal{P}_0 \stackrel{f_0}{\to} \M \to 0\]
where each $\mathcal{P}_i$ is a direct sum of right twists of $\hdnlK$. Now if $\M_i = \im f_i$, then for each $i \geq 1$ the long exact sequence of cohomology together with Theorem \ref{BBVanishing} shows that
\[H^i(\Fl, \M) = H^{i+1}(\Fl, \M_1) = \cdots = H^{i+d}(\Fl, \M_d) = 0\]
by \cite[Theorem III.2.7]{Hart}.\end{proof}
We now start working towards computing the global sections of $\hdnl$.

\subsection{Restrictions on the prime $p$}\label{VeryGoodp}
Let $p$ be the characteristic of the residue field $k$ of $R$. Recall that $p$ is said to be \emph{bad} for an irreducible root system $\Phi$ if
\begin{itemize}
\item $p=2$ when $\Phi = B_l$, $C_l$ or $D_l$;
\item $p=2$ or $3$ when $\Phi = E_6, E_7, F_4$ or $G_2$
\item $p=2,3$ or $5$ when $\Phi = E_8$.
\end{itemize}
We say that $p$ is \emph{bad} for $\mb{G}$ if it is bad for some irreducible component of the root system of $\mb{G}$, and we say that $p$ is a \emph{good} prime for $\mb{G}$ if $p$ is not bad. Finally $p$ is said to be \emph{very good} for $\mb{G}$ if $p$ is good and no irreducible component of the root system of $\mb{G}$ is of type $A_{mp - 1}$ for some integer $m \geq 1$.

\textbf{We assume from now on that $p$ is a very good prime for $\mb{G}$}.

\subsection{Rings of invariants}\label{RingsofI}
Recall that we have chosen a Cartan subgroup $\mb{T}$ of $\mb{G}$ in \S \ref{HCmap}, and that it induces a root space decomposition $\fr{g} = \fr{n} \oplus \fr{t} \oplus \fr{n}^+$ of $\fr{g}$. The Weyl group $\mb{W}$ of $\mb{G}$ acts naturally on $\fr{t}$ and hence on the symmetric algebra $\Sy{t}$. On other hand, the adjoint action of $\mb{G}$ on $\fr{g}$ extends to an action on $\U{g}$ by ring automorphisms. This action preserves the filtration on $\U{g}$ and induces the adjoint action of $\mb{G}$ on $\gr \U{g} \cong \Sy{g}$. Let
\[ \psi : \Si{g}{G} \to \Sy{t}\]
be the composition of the inclusion $\Si{g}{G} \hookrightarrow \Sy{g}$ with the projection $\Sy{g} \to \Sy{t}$ along the decomposition $\Sy{g} = \Sy{t} \oplus \left( \fr{n} \Sy{g} + \Sy{g}\fr{n}^+\right)$. By \cite[Theorem 7.3.7]{Dix}, the image of $\psi$ is contained in $\Sy{t} \cap S(\fr{t}_K)^{\mb{W}} = \Si{t}{W}$, and $\psi$ is injective.

Let $\Ui{g}{G}$ denote the subalgebra of $\mb{G}$-invariants of $\U{g}$. Since taking $\mb{G}$-invariants is left exact, there is a natural inclusion
\[\iota : \gr (\Ui{g}{G}) \to \Si{g}{G}\]
of graded rings. Inspired by the ideas contained in \cite[\S 9.6]{Ja2}, we can now compute the associated graded ring of $\Ui{g}{G}$.

\begin{prop} The rows of the diagram
\[\xymatrix{ 0 \ar[r] & \gr ( \Ui{g}{G} ) \ar[r]^{\pi}\ar[d]_{\iota} & \gr (\Ui{g}{G}) \ar[r]\ar[d]_{\iota} & \gr (\Uik{g}{G}) \ar[r]\ar[d]_{\iota_k} & 0 \\0 \ar[r] & \Si{g}{G} \ar[r]^{\pi}\ar[d]_{\psi} & \Si{g}{G} \ar[r]\ar[d]_{\psi} & \Sik{g}{G} \ar[r]\ar[d]_{\psi_k} & 0
\\0 \ar[r] & \Si{t}{W} \ar[r]^{\pi} & \Si{t}{W} \ar[r] & \Sik{t}{W} \ar[r] & 0
}\]
are exact and each vertical map is an isomorphism.\end{prop}
\begin{proof} We view this diagram as a sequence of complexes $C^\bullet \stackrel{\iota^\bullet}{\to} D^\bullet \stackrel{\psi^\bullet}{\to} E^\bullet$. It is easy to check that each complex is exact except possibly in the right-most non-zero entry. Since $p$ is very good for $\mb{G}$ and $\mb{G}$ is simply-connected, it follows from \cite[Corollaire du Th\'eor\`eme 2]{Dema} that $E^\bullet$ is actually exact.

Now $\psi$ is injective by \cite[Theorem 7.3.7]{Dix} and $\psi_k$ is an isomorphism by \cite[Theorem 4(i)]{KW} since $p$ is good, so $\psi^\bullet \circ \iota^\bullet$ is also injective. Consider the short exact sequence of complexes $0 \to C^\bullet \stackrel{\psi^\bullet \circ \iota^\bullet}{\longrightarrow} E^\bullet \to F^\bullet \to 0$ where $F^\bullet := \coker(\psi^\bullet \circ \iota^\bullet)$. Since $E^\bullet$ is exact and $H^0(C^\bullet) = H^1(C^\bullet) = 0$, the long exact sequence of cohomology shows that $H^0(F^\bullet) = H^2(F^\bullet) = 0$ and that there is an isomorphism $H^1(F^\bullet) \tocong H^2(C^\bullet)$.

Since $\psi_K \circ \iota_K : \gr ( \ugK^{\mb{G}_K}) \to S(\fr{t}_K)^{\mb{W}_K}$ is an isomorphism by \cite[Theorem 7.3.7]{Dix}, we see that $F^0 = F^1  = \coker(\psi\circ \iota)$ is $\pi$-torsion. But since $H^0(F^\bullet) = 0$, the sequence $0 \to F^0 \stackrel{\pi}{\to} F^1 $ is exact whence $F^0 = F^1 = 0$. Hence $H^1(F^\bullet) = H^2(C^\bullet) = 0$ and the top row $C^\bullet$ of the diagram is exact.

Finally, since $\psi^\bullet \circ \iota^\bullet : C^\bullet \to E^\bullet$ is now an isomorphism in all degrees except possibly 2, it must be an isomorphism by the Five Lemma. The result follows because $\psi^\bullet$ and $\iota^\bullet$ are both injections.\end{proof}

\begin{cor} $\gr (\Ui{g}{G})$ is isomorphic to a polynomial algebra over $R$ in $l$ variables.
\end{cor}
\begin{proof} Since $\psi \circ \iota$ is a graded isomorphism by the Proposition, this follows from \cite[Th\'eor\`eme 3]{Dema} and \cite[\S I.2.10(3)]{Jantzen}.
\end{proof}

\subsection{Global sections of $\hdnlK$}\label{GlobSec}
It follows from Corollary \ref{RingsofI} that $\Ui{g}{G}$ is itself a commutative polynomial algebra over $R$ in $l$ variables. Hence
\[\hugGnK := \widehat{(\Ui{g}{G})_{n,K}}\]
is a commutative Tate algebra in $l$ variables. 

The commutative square in Lemma \ref{HCmap} consists of deformable $R$-algebras. Applying the deformation functor to it we obtain another commutative square of deformable $R$-algebras
\[\xymatrix{\UG \ar[d]\ar[r]^{\phi_n} & \U{t}_n \ar[d]^{(j \circ i)_n} \\
\Un \ar[r]_{U(\varphi)_n} & \dnt .
} \]
We view the $\U{h}_n$-module $R_\lambda$ as a $\UG$-module via restriction along the map $(i \circ \phi)_n : \UG \to \U{h}_n$, and let $K_\lambda := R_\lambda \otimes_R K$ be the corresponding $\hugGnK$-module. We make the following definitions:
\begin{itemize}
\item $\unl :=  \Un \otimes_{\UG} R_\lambda $,
\item $\hunl := \invlim  \unl/ \pi^a \unl$, and
\item $\hunlK := \hunl \otimes_R K$.
\end{itemize}
Because the diagram commutes, the map $U(\varphi)_n \otimes (j \circ i)_n : \Un \otimes \U{t}_n \to \dnt$ factors through $\UG$ and we obtain algebra homomorphisms
\[\begin{array}{lllll}\varphi^{\lambda}_n &:& \unl &\to& \dnl,  \\

\widehat{\varphi^{\lambda}_n} &:& \hunl &\to& \hdnl,\quad\mbox{and} \\

\widehat{\varphi^\lambda_{n,K}} &:& \hunlK &\to& \hdnlK.
\end{array}\]
It is not immediately clear whether $\unl$ is a deformable $R$-algebra as it could \emph{a priori} have $\pi$-torsion. Presumably $\gr \U{g}$ is a free module over $\gr \left(\U{g}^{\mb{G}}\right)$, which would imply that $\unl$ is deformable. However we will not need to prove this.

\begin{thm} 
\be \hfill
\item $\hunlK \cong \hugnK \otimes_{\hugGnK} K_\lambda$ is an almost commutative affinoid $K$-algebra.
\item The map $\widehat{\varphi^\lambda_{n,K}} : \hunlK \to \Gamma(\Fl,\hdnlK)$ is an isomorphism of complete doubly filtered $K$-algebras.
\item There is an isomorphism $\Syk{g} \otimes_{ \Sik{g}{G} } k \tocong \Gr(\hunlK)$.
\ee\end{thm}\begin{proof}
(a) By Corollary \ref{RingsofI} and Lemma \ref{Defs}, $\UG$ is Noetherian. Hence there is an $R$-algebra isomorphism $\hunl \tocong \hunt \otimes_{\hunG} R_\lambda$ by Lemma \ref{ComplCent}. So $\hunlK$ is an almost commutative affinoid $K$-algebra, being a quotient of $\hugnK$. 

(b), (c) Let $\{U_1,\ldots, U_m\}$ be an open cover of $\Fl$ by open affines that trivialise the torsor $\xi$; thus the special fibre $\bFl$ is covered by the special fibres $U_{i,k}$. For part (b), it will be enough to prove that the sequence
\[C^\bullet: \quad 0 \to \hunlK \stackrel{\widehat{\varphi^\lambda_{n,K}}}{\longrightarrow} \bigoplus_{i=1}^m \hdnlK(U_i) \to \bigoplus_{i < j} \hdnlK(U_i \cap U_j)\]
is exact. Since $C^\bullet$ is a complex in the category of complete doubly filtered $K$-algebras, it is enough to show that $\Gr(C^\bullet)$ is exact.

By Corollary \ref{GrComp}, there is a commutative diagram with exact rows
\[\xymatrix{ 0 \ar[r] & \gr (\Ui{g}{G}) \ar[r]^{\pi}\ar[d] & \gr (\Ui{g}{G}) \ar[r]\ar[d] & \Gr(\hugGnK) \ar[r]\ar[d] & 0 \\
0 \ar[r] & \gr (\U{g}) \ar[r]^{\pi} & \gr (\U{g}) \ar[r] & \Gr(\hugnK) \ar[r] & 0 .}
\]
Since $\gr (\U{g}) = \Sy{g}$, Proposition \ref{RingsofI} induces a commutative square
\[\xymatrix{ \Gr(\hugGnK) \ar[r]\ar[d] & \Sik{g}{G} \ar[d]\\ \Gr(\hugnK) \ar[r] & \Syk{g} }\]
where the horizontal maps are isomorphisms and the vertical maps are inclusions. Since $\Gr(K_\lambda)$ is the trivial $\Gr(\hugGnK)$-module $k$, we obtain a natural surjection 
\[ \Syk{g} \otimes_{ \Sik{g}{G} } k \cong \Gr(\hugnK) \otimes_{\Gr(\hugGnK)} \Gr(K_\lambda) \twoheadrightarrow \Gr(\hunlK)\]
which fits into the commutative diagram
\[\xymatrix{  0 \ar[r] & \Syk{g} \otimes_{ \Sik{g}{G} } k \ar[r]\ar[d] & \bigoplus\limits_{i=1}^m \O(T^\ast U_{i,k}) \ar[r] & \bigoplus\limits_{i < j} \O(T^\ast (U_{i,k} \cap U_{j,k})) \\
0 \ar[r] & \Gr(\hunlK) \ar[r] & \bigoplus\limits_{i=1}^m \Gr(\hdnlK(U_i)) \ar[r]\ar[u] & \bigoplus\limits_{i < j} \Gr(\hdnlK(U_i \cap U_j))\ar[u] }
\]
where the bottom row is $\Gr(C^\bullet)$ and the top row appeared in the proof of \cite[Proposition 3.4.1]{BMR1} and is induced by the moment map $T^\ast \Fl_k \to \frk{g}^\ast$. It was shown in \emph{loc. cit.} that under the assumption that $p$ is very good for $\mb{G}$, the top row is exact. The second and third vertical arrows are isomorphisms by Proposition \ref{ComplCent}(a). An elementary diagram chase now shows that $\Gr(C^\bullet)$ is exact,  proving (b), and also that the first vertical arrow is an isomorphism, proving (c).
\end{proof}

\subsection{The localisation functors}\label{LocFun}
Recall the localisation functor from $\S$\ref{EnhLoc}:
\[\Loc : \md{\hugnK}\to \md{\hdntK}.\]
For each $\lambda\in \Hom_R(\pi^n \fr{h},R)$ we also have a functor \[ \Loc^\lambda :  \md{\hunlK} \to \md{\hdnlK},\] given by $M\mapsto \hdnlK\otimes_{\hunlK}M$ which we will also call a localisation functor. Since $\hdnlK$ is a quotient of $\hdntK$ by Proposition \ref{ComplCent}(c), we can and will view $\Loc^\lambda(M)$ as a $\hdntK$-module.

\begin{lem} For any finitely generated $\hunlK$-module $M$, there is a natural surjection of $\hdntK$-modules $\Loc(M) \twoheadrightarrow \Loc^\lambda(M)$.
\end{lem}
\begin{proof}
By Theorem \ref{GlobSec}(a), there is an isomorphism
\[\hdntK \otimes_{\hugnK} \hunlK \cong \hdntK \otimes_{\hugnK} \left(\hugnK \otimes_{\hugGnK} K_\lambda\right) \cong \hdntK \otimes_{\hugGnK} K_\lambda\]
of sheaves of complete doubly filtered $K$-algebras. There is also the isomorphism
\[\hdnlK \cong \hdntK \otimes_{\huhnK} K_\lambda\]
of sheaves of complete doubly filtered $K$-algebras by Proposition \ref{ComplCent}(c). These isomorphisms fit together into a commutative diagram
\[\xymatrix{ \hdntK \otimes_{\hugnK} \hunlK \ar[d]_{\cong} \ar@{.>>}[r] & \hdnlK \ar[d]^{\cong}\\ \hdntK \otimes_{\hugGnK} K_\lambda \ar@{>>}[r] & \hdntK \otimes_{\huhnK} K_\lambda}\]
and the obvious surjective horizontal arrow in the second row induces the dotted surjective arrow in the first row. This proves the result in the case when $M = \hunlK$; in the general case, pick a presentation $F_1 \to F_0 \to M \to 0$ of $M$ where $F_1$ and $F_0$ are finitely generated free $\hunlK$-modules, and apply the Five Lemma. \end{proof}

\subsection{The equivalence of categories} \label{BBThm}
Recall that a weight $\lambda \in \fr{h}^\ast_K$ is said to be \emph{regular} if its stabilizer under the action of $\mb{W}$ is trivial. Recall also that we're assuming that our ground field $K$ has characteristic zero.

\begin{prop} Let the weight $\lambda \in \Hom_R(\pi^n \fr{h}, R)$ be such that $\lambda + \rho$ is dominant. Then $\Fl$ is coherently $\hdnlK$-acyclic. If $\lambda+\rho$ is also regular then $\Fl$ is $\hdnlK$-affine.
\end{prop}

\begin{proof} Let $\M$ be a coherent $\hdnlK$-module. By Corollary \ref{BBVanishing}, $\M$ is $\Gamma$-acyclic. Because we may view $\M$ as a coherent $\hdntK$-module, $\M(X)$ is finitely generated over $\hugnK$ by Proposition \ref{Bry1}. Thus $\M(X)$ is a coherent $\hdnlK(X)$-module and so $\Fl$ is coherently $\hdnlK$-acyclic.

Suppose now that $\lambda+\rho$ is regular. By Theorem \ref{Gen} and Proposition \ref{ComplCent}(c) we can find a surjection $\mathcal{F} \twoheadrightarrow \M$ where $\mathcal{F}$ is a direct sum of sheaves of the form $\hdnlK(s_i)$ for some integers $s_i$.  Since each $\dlK(s_i)$ is a coherent $\mathcal{D}_K^{\lambda}$-module, it is generated by its global sections by paragraph (iv) of the proof of \cite[Th\'eor\`eme Principal]{BB}. Hence each $\hdnlK(s_i)$ is generated by its global sections by Theorem \ref{UnComplete}(a) and we can therefore find a surjection $\mathcal{F}_0 \to \M$ where $\mathcal{F}_0$ is a direct sum of copies of $\hdnlK$. Applying the same argument to the kernel of this surjection gives a presentation $\mathcal{F}_1 \to \mathcal{F}_0 \to \M \to 0$. Since $\Gamma$ is exact and $\Loc^\lambda$ is right exact there is a commutative diagram \[\xymatrix{ \mathcal{F}_1 \ar[r] & \mathcal{F}_0 \ar[r] & \M \ar[r]& 0 \\
\Loc^\lambda(\Gamma(\mathcal{F}_1)) \ar[r]\ar[u] & \Loc^\lambda(\Gamma(\mathcal{F}_0)) \ar[r]\ar[u] & \Loc^\lambda(\Gamma(\M)) \ar[r]\ar[u] & 0}\] with exact rows and with the first two vertical maps isomorphisms. Thus it follows from the Five Lemma that the canonical map $\Loc^\lambda(\Gamma(\M))\to \M$ is an isomorphism so $\M$ is generated by its global sections and $\Fl$ is $\hdnlK$-affine as required.\end{proof} 
 We can finally prove Theorem \ref{BBIntro}.

\begin{thm} Let the weight $\lambda \in \Hom_R(\pi^n \fr{h}, R)$ be such that $\lambda + \rho$ is dominant and regular. Then the functors $\Loc^\lambda$ and $\Gamma$ are mutually inverse equivalences of abelian categories between $\coh(\hunlK)$ and $\coh(\hdnlK)$. 

If $\lambda+\rho$ is dominant but not regular then $\Loc^\lambda$ and $\Gamma$ still induce mutually inverse exact equivalences of abelian categories between $\coh(\hunlK)$ and $\coh(\hdnlK)/\ker \Gamma$.
\end{thm}
 \begin{proof} This follows from immediately from the Proposition and Proposition \ref{CohDaffine}.\end{proof}

\begin{cor} Suppose that $\lambda$ is $\rho$-dominant, let $M$ be a finitely generated $\hunlK$-module and let $\M = \Loc^\lambda(M)$. Then
\[\beta(\Ch(\M)) = \Ch(M).\]
\end{cor}
\begin{proof} We have $\Ch(\Loc^\lambda(M))\subseteq \Ch(\Loc(M))$ by Lemma \ref{LocFun} and Corollary \ref{CharVar}(b). Now $M \cong \Gamma( \Loc^\lambda(M))$ by the Theorem, so
\[ \Ch(M) = \Ch( \Gamma( \Loc^\lambda(M) ) \subseteq \beta \Ch( \Loc^\lambda (M)) \subseteq \beta \Ch( \Loc(M)) \subseteq \Ch(M)\]
by Propositions \ref{Bry1} and \ref{CharLocM}.
\end{proof}

\section{Bernstein's Inequality}
\label{BernsteinIneq}

In this section we continue to assume that $R$ is a complete discrete valuation ring with uniformizer $\pi$, residue field $k$ and field of fractions $K$ of characteristic zero.

\subsection{Affinoid Weyl algebras and symplectic forms}
\label{AffWeyl}
Suppose that $V$ is a free $R$-module of finite rank equipped with an alternating bilinear form $\omega$ and recall from Example \ref{ExAC}(d) the definition of the enveloping algebra $\Env{V}$ of $(V,\omega)$.  As we remarked in $\S$\ref{Defs}, this is a deformable $R$-algebra. Thus given a alternating form $\omega$ on a free $R$-module of finite rank we may form the almost commutative affinoid $K$-algebra $\AEn{V}:=\widehat{\Env{V}_{n,K}}$.

If $S$ is an $R$-algebra then we write $V_S$ for the free $S$-module $S\otimes_R V$. Similarly, given an $R$-bilinear form $\phi\colon V\times V\rightarrow R$, we write $\phi_S$ for the $S$-bilinear form on $V_S$ obtained by $S$-linearly extending $\phi$.

\begin{defn} We say that an $R$-bilinear form $\omega$ on a free $R$-module $V$ of finite rank is \emph{symplectic} if $\omega_K$ and $\omega_k$ are symplectic forms on $V_K$ and $V_k$ respectively; that is if $\omega_K$ and $\omega_k$ are non-degenerate alternating forms.
\end{defn}

\begin{defn} If $\omega$ is a symplectic form on $V$ we call $\AEn{V}$ an \emph{affinoid Weyl algebra}. \end{defn}

If $A=R[x_1,\ldots,x_m,\partial_1,\ldots,\partial_m]$ is the $m$th Weyl algebra over $R$, then $A$ is the enveloping algebra of the standard symplectic form on $R^{2m}$. Moreover, since every symplectic form on $R$ is equivalent to the standard one, up to isomorphism every enveloping algebra of a symplectic form arises in this way.

\textbf{For the remainder of this section we fix a free $R$-module $V$ of rank $2m$ and a symplectic form $\omega$ on $V$.}

Whenever $W$ is a free summand of $V$, $\omega$ will restrict to an alternating form on $W$ that by abuse of notation we will also call $\omega$.

\subsection{Grassmannians and $\perp$} \label{Grass}

\begin{defn} For each $0\leq t\leq 2m$, we define $\Grass_t(V)$ to be the set of free summands of $V$ as an $R$-module of rank $t$. Similarly we define $\Grass_t(V_k)$ to be the set of $k$-subspaces of $V_k$ of dimension $t$. \end{defn}

\begin{prop} The natural map $\Grass_t(V)\to \Grass_t(V_k)$ given by $W\mapsto W_k$ is a surjection.
\end{prop}

\begin{proof} Since $R$ is $\pi$-adically complete we may apply the idempotent lifting lemma to $M_n(R)\to M_n(k)$.
\end{proof}

\begin{defn} If $W$ is a free summand of $V$ we define \[ W^\perp=\{ v\in V\mid \omega(v,w)=0\mbox{ for all }w\in W\}.\]
\end{defn}
Similarly, if $W$ is a subspace of $V_k$ we define $W^\perp \leq V_k$ by the same formula.
\begin{lem}\hfill
\be
\item If $W\in \Grass_t(V)$ then $W^\perp\in \Grass_{2m-t}(V)$;
\item $(W^\perp)_k= (W_k)^\perp$ for all $W \in \Grass_t(V)$.
\ee
\end{lem}

\begin{proof}
For (a) it suffices to prove that $W^\perp$ is a free summand of $V$, its rank then follows from the rank-nullity theorem since $\omega$ is non-degenerate. Since $R$ is a discrete valuation ring it is enough to show that $V / W^\perp$ is $\pi$-torsion-free. But we know $\omega(\pi v,w)=\pi \omega(v,w)$ so for $v \in V$, $\pi v\in W^\perp$ if and only if $v\in W^\perp$.

For (b), $(W^\perp)_k$ is easily seen to be contained in $(W_k)^\perp$. Since $\omega_k$ is non-degenerate, $\dim (W_k)^\perp = 2m - t$ by the rank-nullity theorem, and $\dim (W^\perp)_k = 2m - t$ by part (a). The result follows.
\end{proof}

Recall that each $\Grass_t(V_k)$ is an irreducible algebraic variety, when equipped with the Zariski topology.

\begin{thm} For each $0\leq t\leq 2m$ the map $\perp\colon\Grass_t(V_k)\to \Grass_{2m-t}(V_k)$ that sends $W_k\to W_k^\perp$ is a homeomorphism.
\end{thm}

\begin{proof} Since $\omega_k$ is non-degenerate it defines a perfect pairing $V_k\times V_k\to k$ and so identifies $V_k$ and $V_k^\ast$. Then the result is well-known, see \cite[\S 2.8]{Lafforgue} for example.
\end{proof}

\subsection{Simplicity of affinoid Weyl algebras} \label{simple}

We now consider how $V$ acts by derivations on $\AEn{V}$.

\begin{lem} For $v\in V$, let $\epsilon(v) := \omega(v,-)$ be its image in $V^\ast := \Hom_R(V,R)$, let $\partial_v \in V_k^\ast$ be the image of $\epsilon(v)$ modulo $\pi$ and suppose that $\partial_v \neq 0$. Let $d_v$ be the $R$-derivation of $\Env{V}_n$ given by \[ d_v \colon r  \mapsto  \frac{vr-rv}{\pi^n}. \]
Then
\be
\item $d_v$ extends to a derivation $d_v$ of $\AEn{V}$.
\item $\Gr(d_v)$ is the unique $k$-derivation of $\Sym V_k$ that extends $\partial_v$.
\item The $K$-linear endomorphism $\frac{d_v^m}{m!}$ of $\AEn{V}$ preserves the $R$-lattice $F_0 \AEn{V}$ for all $m \geq 0$, and $\Gr\left(\frac{d_v^m}{m!}\right)$ acts as the $m$th divided power of $\partial_v$ on $\Sym V_k$.
\ee
\end{lem}

\begin{proof} All parts follow from some straightforward calculations.
\end{proof}

The following result is a special case of \cite[Proposition 1.4.6]{Pangalos}. We provide a proof for the convenience of the reader.

\begin{thm} The ring $\AEn{V}$ is simple. \end{thm}

\begin{proof} Since $\AEn{V}$ is a complete doubly filtered $K$-algebra, it suffices to prove that if $I\neq 0$ is an ideal then $\Gr (I)\subseteq \Sym V_k$ contains $1$.

Since $\omega_k$ is a non-degenerate form on $V_k$, for each $\partial\in V_k^\ast$ there is $v\in V$ such that the reduction of $\omega(v,-)$ mod $\pi$ induces $\partial$. Moreover, by the Lemma, $\Gr\left( \frac{d_v^m}{m!}\right)$ is the $m$th divided power of $\partial$ on $\Sym V_k$. It is easy to verify that there are no non-trivial ideals in $\Sym V_k$ invariant under all of these differential operators.
\end{proof}

\subsection{Bernstein's inequality}
\label{BernIneq}
\begin{thm} Suppose that $V$ is a free $R$-module of rank $2m$ and $\omega$ is a symplectic form on $V$. If $M$ is a non-zero finitely generated $\AEn{V}$-module then every irreducible component $X$ of $\Ch(M)$ satisfies \[ \dim X \ge m.\] \end{thm}

It is straightforward to construct modules that attain this bound.

\begin{proof} We may find a sequence of submodules $0=M_0\subset M_1\subset \cdots\subset M_r=M$ such that $M_i/M_{i-1}$ is pure for all $i\geq 1$. By Proposition \ref{CharVarAff}, $\Ch(M)=\bigcup \Ch(M_i/M_{i-1})$ so we may reduce to the case that $M$ is pure. Moreover by Theorem \ref{CharVarAff} we know that in this case every irreducible component of $\Ch(M)$ has the same dimension. So it suffices to show that $\dim \Ch(M)\geq m$.

Suppose that $\dim \Ch (M) < m$.  By base changing to the completion of the maximal unramified extension of $K$ and applying Proposition \ref{basechange2}, we may assume that $k$ is an infinite field. Choose a good double filtration on $M$. Now
\[ \mathcal{X} := \{ W\in \Grass_m(V_k)\mid \Gr(M)\mbox{ is over finitely generated over }\Sym W\} \]
is a non-empty and Zariski open subset in $\Grass_m(V_k)$ by the Generic Noether Normalization Lemma \cite[Remark 3.4.4]{GreuelPfister}. So $\mathcal{X}^\perp$ is again open and non-empty by Theorem \ref{Grass}. Since $\Grass_m(V_k)$ is irreducible, we deduce that $\mathcal{X}\cap\mathcal{X}^\perp$ is non-empty.

Using Proposition \ref{Grass}, choose $W\in \Grass_m(V)$ such that $W_k\in \mathcal{X}\cap \mathcal{X}^\perp$. Then $(W^\perp)_k \in \mathcal{X}$ by Lemma \ref{Grass}(b), so $M$ is finitely generated over both $\AEn{W}$ and $\AEn{W^\perp}$ by Lemma \ref{GoodDouble}(a). Choose a finite generating set $X$ for $M$ as a $\AEn{W^\perp}$-module. Now elements of $\AEn{W^\perp}$ act as $\AEn{W}$-endomorphisms of $M$, so \[ \Ann_{\AEn{W}}(M)=\bigcap_{x\in X} \ann_{\AEn{W}}(x).\]

But $\dim \Ch(M) < m =\dim \Gr \AEn{W}$, so $M$ must be torsion as a $\AEn{W}$-module and each term in the intersection is non-zero. Since also $\AEn{W}$ is a Noetherian domain, it follows from \cite[Theorem 2.1.15]{MCR} that $\Ann_{\AEn{W}}(M)\neq 0$. But $ \Ann_{\AEn{W}}(M)\subseteq \Ann_{\AEn{V}}(M)$ and so the latter is a non-zero ideal of $\AEn{V}$. By Theorem \ref{simple}, $\Ann_{\AEn{V}}(M)=\AEn{V}$ and $M=0$ as required.
\end{proof}

\begin{cor} If $A=R[x_1,\ldots,x_m,\partial_1,\ldots,\partial_m] = \D(\mathbb{A}^m)$ is the $m$th Weyl algebra equipped with the order filtration and $M$ is a non-zero finitely generated $\widehat{A_{n,K}}$-module, then every irreducible component $X$ of $\Ch(M)$ satisfies $\dim X\geq m$.
\end{cor}

\begin{proof}  As in the proof of the Theorem above we may reduce to the case that $M$ is pure and so assume that every irreducible component of $\Ch(M)$ has the same dimension. So again it suffices to show that $\dim \Ch(M) \geq m$.

If $n$ is even, there is an isomorphism of $R$-algebras $\alpha : \Env{V}_{n/2} \to A_n$, which induces an isomorphism $\widehat{\alpha} : K_\omega\langle V\rangle_{n/2} \to \widehat{A_{n,K}} $ of complete filtered rings. Although $\widehat{\alpha}$ is \emph{not} an isomorphism of complete doubly filtered algebras, $\Gr(\widehat{A_{n,K}})$ and $\Gr(K_\omega\langle V\rangle_{n/2})$ are both polynomial algebras over $k$ in $2m$ variables (with different gradations). Letting $N$ be the restriction of $M$ to a $K_\omega\langle V\rangle_{n/2}$-module via $\h{\alpha}$, we see that
\[ \dim \Ch(M) = 2m - j(M) = 2m - j(N) = \dim \Ch(N) \geq m\]
by Theorem \ref{CharVarAff} and the Theorem above.

If $n$ is odd, first we base-change to $K' = K(\sqrt{\pi})$ and let $A' = R' \otimes_R A$ where $R' = R[\sqrt{\pi}]$. Then $K' \otimes_K \widehat{A_{n,K}}\cong\widehat{A'_{2n,K'}}$ by Lemma \ref{basechange2}(c). By Proposition \ref{basechange2} this has not changed $\dim \Ch(M)$ and we have returned to the case when $n$ is even.
\end{proof}

\subsection{A global version of Bernstein's inequality}

Recall the notation of \S \ref{TDO}.

\begin{thm}\label{GlobalBernIneq}
If $\M$ is a non-zero coherent $\hdnlK$-module, then every irreducible component $X$ of $\Ch(\M)$ satisfies $\dim X \geq \dim \Fl$.
\end{thm}

\begin{proof} Let $m=\dim \Fl$. By the proof of Lemma \ref{AlgGps}(c), the Weyl translates $U_w$ of the big cell in $\Fl$ each trivialise the torsor $\xi \colon \Ex{\Fl}\to \Fl$ and are isomorphic to $\mathbb{A}^m$. Hence Theorem \ref{hdntKaff} and Lemma \ref{CharVar} tell us that for each $w$ in the Weyl group $\mb{W}$, $\M(U_w)$ is a finitely generated $\hdntK(U_w)$-module and $\Ch(\M)\cap \Ex{T^\ast U_w}=\Ch(\M(U_w))$. Because the $U_w$ cover $\mathcal{B}$, it thus suffices to prove that every irreducible component of $\Ch(\M(U_w))$ has dimension at least $m$ for every $w\in \mb{W}$ such that $\M(U_w)\neq 0$.

By Proposition \ref{ComplCent}(a), $\hdnlK(U_w)$ is isomorphic to $\widehat{\D(\mathbb{A}^m)_{n,K}}$ as an almost commutative affinoid $K$-algebra. The result now follows from  Corollary \ref{BernIneq}.
\end{proof}\section{Quillen's Lemma}
Recall that Quillen proved in \cite{Quillen} that the endomorphism ring of a simple module over an almost commutative $k$-algebra is algebraic over $k$. We will generalise this to show that the endomorphism ring of a simple module over an almost commutative affinoid $K$-algebra $A$ is algebraic over $K$ provided that $\gr_0 A$ is commutative and Gorenstein. In particular this will apply to our rings $\widehat{U(\fr{g})_{n,K}}$ for $n>0$. 

In this section, $K$ will be a complete discrete valuation field of arbitrary characteristic and $A$ will denote a complete filtered $K$-algebra such that $F_0A$ is an $R$-lattice in $A$ and the slice $\gr_0 A$ is a finitely generated commutative $k$-algebra. We have in mind almost commutative affinoid $K$-algebras with commutative slice but our proofs do not need to use the filtration on the slice except to guarantee that the slice is finitely generated as a $k$-algebra. 

\subsection{Regular $\varphi$-lattices and Quillen's Lemma}\label{RegLat}
Let $M$ be a finitely generated $A$-module and $\varphi\in \End_A(M)$. We say that $\varphi$ is \emph{simple} if every non-zero element of the $K$-algebra generated by $\varphi$ acts invertibly on $M$. Clearly every non-zero $\varphi \in \End_A(M)$ is simple whenever $M$ is a simple $A$-module by Schur's Lemma, but there are other interesting examples. We will prove that every simple endomorphism is algebraic over $K$.

Suppose now that $\varphi$ is a simple endomorphism of a finitely generated $A$-module $M$. We let $K(\varphi)$ be the subfield of $\End_A(M)$ generated by $K$ and $\varphi$ and note that we can view $M$ as an $A-K(\varphi)$-bimodule.

\begin{defn} We say that an $R$-lattice $N$ in $M$ is an \emph{$F_0A$-lattice} if it is a finitely generated $F_0A$-submodule of $M$. We say that an $F_0A$-lattice $N$ in $M$ is a \emph{regular $\varphi$-lattice} if the subring
\[B:=\{\theta\in K(\varphi)\mid \theta(N)\subseteq N\}\]
of $K(\varphi)$ is a discrete valuation ring and an $R$-lattice in $K(\varphi)$.\end{defn}

We begin along similar lines to Quillen's original proof. However we can only do this if $M$ has a regular $\varphi$-lattice.

\begin{lem} Suppose that $M$ has a regular $\varphi$-lattice $N$. Then the residue field of $B=\{\theta\in K(\varphi)\mid \theta(N)\subseteq N\}$ is an algebraic extension of $k$.
\end{lem}

\begin{proof} Let $\tau$ be a uniformiser of $B$ and note that the residue field $k' := B/\tau B$ of $B$ acts faithfully on $N/\tau N$ by the definition of $B$. In particular, $N/ \tau N$ is non-zero. Since $B$ is an $R$-lattice in $K(\varphi)$, $\pi$ is not a unit in $B$. Hence $\pi \in \tau B$ and $N/\tau N$ is a finitely generated $\gr_0 A = F_0 A / \pi F_0A$-module.

Let $s \in k'$ and consider the $k$-subalgebra $k[s]$ of the field $k'$ generated by $s$. Because $\gr_0 A$ is commutative by assumption and $N/\tau N$ is a finitely generated $(\gr_0 A)[s]$-module, by the Generic Flatness Lemma \cite[Lemme IV.6.9.2]{EGAIV2}, we can find a non-zero element $f \in k[s]$ such that $(N/\tau N)_f$ is a free $k[s]_f$-module. Now every non-zero element of $k'$ acts invertibly on $N/ \tau N$ because $k'$ is a field; hence every non-zero element of $k[s]_f$ acts invertibly on the non-zero free $k[s]_f$-module $(N / \tau N)_f$. This forces $k[s]_f$ to be a field. Because $k[s]_f$ is a finitely generated $k$-algebra, it must be algebraic over $k$ by the Nullstellensatz.
\end{proof}

\begin{prop} Suppose that $M$ has a regular $\varphi$-lattice $N$. Then the residue field of $B:=\{\theta\in K(\varphi)\mid \theta(N)\subseteq N\}$ is a \emph{finite} algebraic extension of $k$.
\end{prop}

\begin{proof} Once again let $\tau$ be a uniformiser of $B$. We have already seen that the residue field $k'$ of $B$ must be algebraic over $k$ and that $k'$ acts by automorphisms on $N/\tau N$ and so may be identified with a subfield of $\End_{\gr_0A}(N/\tau N)$.

Let $P\in \Spec(\gr_0 A)$ be a minimal prime over $\Ann_{\gr_0 A}(N/\tau N)$. Then if $k(P)$ is the residue field of $(\gr_0A)_P$ and $X:=k(P)\otimes_{\gr_0 A}N / \tau N$, there is a $k$-algebra homomorphism $\End_{\gr_0 A}(N / \tau N)\to \End_{k(P)}(X)$. Thus we may identify $k'$ with a subfield of $\End_{k(P)}(X)$; a matrix ring over $k(P)$ since $X$ is a finite dimensional $k(P)$-vector space.

Since $\gr_0 A$ is a finitely generated $k$-algebra, $k(P)$ is a finitely generated field extension of $k$. Now $k'k(P)$ is commutative finite dimensional $k(P)$-algebra, since it is a subspace of $\End_{k(P)}(X)$ and since $k(P)$ lies in the centre of $\End_{k(P)}(X)$. Thus if $Q$ is a prime ideal in $k'k(P)$ then $L=k'k(P)/Q$ is an integral domain and a finite dimensional $k(P)$-vector space. Thus $L$ is a finitely generated field extension of $k$. It follows from \cite[Proposition III.6]{Lang58} that every subextension, and in particular the image of $k'$ in $L$, is a finitely generated field extension of $k$. But $k'$ is isomorphic to its image in $L$ and the result follows.
\end{proof}

\begin{cor} Let $A$ be a complete filtered $K$-algebra such that $F_0A$ is an $R$-lattice in $A$ and the slice $\gr_0 A$ is a finitely generated commutative $k$-algebra. Suppose that $M$ is a finitely generated $A$-module and $\varphi\in \End_A(M)$ is simple. If $M$ has a regular $\varphi$-lattice, then $\varphi$ is algebraic over $K$.
\end{cor}
\begin{proof}
 Suppose that $N$ is a regular $\varphi$-lattice in $M$. Then the residue field $k'$ of $B:=\{\theta\in K(\varphi)\mid \theta(N)\subseteq N\}$ is a finite extension of $k$ by the Proposition. Since $B$ is a discrete valuation ring with maximal ideal $\tau B$, the $\tau$-adic filtration on $B$ is separated and therefore $B$ is finitely generated $R$-module by \cite[Theorem I.5.7]{LVO}. Because $B$ is an $R$-lattice in $K(\varphi)$, we deduce that $K(\varphi)$ is a finite dimensional $K$-vector space.
\end{proof}

We thank Qing Liu \cite{MO} for giving an example which shows that it is possible to find a discrete valuation on the function field $K(t)$ whose residue field is algebraic (and necessarily infinite dimensional) over $k$.

\subsection{Microlocalisation}\label{MicroQuillen}
We will next show that if $A$, $M$ and $\varphi$ satisfy our conditions and $M$ satisfies one extra hypothesis, then it has a regular $\varphi$-lattice. We begin this process by working microlocally.

We fix a complete filtered $K$-algebra $A$ with commutative slice, a finitely generated $A$-module $M$ and a simple $\varphi\in\End_A(M)$. We fix an $F_0A$-lattice $F_0M$ in $M$ and let $P_1,\ldots, P_r$ be the distinct minimal primes in $\gr_0 A$ above $\Ann_{\gr_0 A}(\gr_0 M)$. Let
\[T := \gr_0A \backslash \bigcup_{i=1}^r P_i.\]
Since $\gr A = (\gr_0 A)[s, s^{-1}]$ by Lemma \ref{DoubleFilt}, $T$ is a multiplicatively closed subset in $\gr A$ consisting of homogeneous elements of degree zero, so we can consider the microlocalisation of $A$ at $T$:
\[Q := Q_T(A)\] as in \S\ref{Micro}.

\begin{lem} The slice $\gr_0 Q_T(M)$ of $Q_T(M)$ is an Artinian $\gr_0 Q$-module, and $Q_T(M)$ is an Artinian $Q$-module.
\end{lem}
\begin{proof} Some product of the $P_i$s annihilates $\gr_0 M$, so we can find a finite chain of $(\gr_0 A)_T$-submodules of $(\gr_0 M)_T$ with each subquotient isomorphic to $(\gr_0 A / P_i)_T$ for some $i$. By our choice of $T$, $(\gr_0 A / P_i)_T$ is the residue field of the local ring of $\gr_0 A$ at $P_i$, so $(\gr_0 M)_T$ has finite length.

It follows that $\gr Q_T(M) \cong (\gr_0 M)_T[s,s^{-1}]$ has the descending chain condition on graded $\gr Q$-modules. Since $Q_T(M)$ is complete with respect to its filtration, it follows from \cite[Proposition I.7.1.2]{LVO} that it is an Artinian $Q$-module.
\end{proof}

Now by functoriality of microlocalisation, every $A$-module endomorphism of $M$ extends to a $Q$-module endomorphism of $Q_T(M)$. We can therefore view $Q_T(M)$ as a $Q-K(\varphi)$-bimodule. We fix a choice of a simple $Q-K(\varphi)$-bimodule quotient $V$ of $Q_T(M)$, and note the Lemma implies that $V$ is an Artinian $Q$-module.

\subsection{Finding a maximal lattice preserver}\label{MaxLatPres}
We retain the notation of the previous section.

\begin{defn} Let $L$ be an $R$-lattice in $V$. We say that $L$ is an \emph{$F_0Q$-lattice} if it is a finitely generated $F_0Q$-submodule of $V$. We say that a subring $B$ of $K(\varphi)$ is a \emph{lattice preserver} if $BL \subseteq L$ for some $F_0Q$-lattice $L$ in $V$.
\end{defn}

\begin{lem} $V$ has at least one $F_0Q$-lattice. For any $F_0Q$-lattice $L$ in $V$, $L / \pi L$ has finite length and $L$ has Krull dimension 1 as a $F_0Q$-module.
\end{lem}
\begin{proof} Let $L_0$ be the image of $F_0Q_T(M)$ in $V$. Since $F_0Q$ is Noetherian and $\pi$-adically complete, the proof of Proposition \ref{GoodDouble}(a) shows that $L_0$ is an $R$-lattice in $V$ and is therefore a $F_0Q$-lattice. Since $L_0/ \pi L_0$ is a quotient of $\gr_0 Q_T(M)$, it has finite length by Lemma \ref{MicroQuillen}. If $L$ is another $F_0Q$-lattice in $V$ then by the arguments in the proof of \cite[Proposition 1.1.2]{Ginz}, $L/ \pi L$ also has finite length; indeed the class of $L/ \pi L$ in the Grothendieck semigroup of $\gr_0 Q$-modules equals the class of $L_0/\pi L_0$. The last statement now follows from \cite[Proposition I.7.1.2]{LVO}.
\end{proof}

We fix an $F_0Q$-lattice $L_0$ in $V$ and let $\mathcal{L}$ be the set of $F_0Q$-lattices in $V$ contained in $L_0$ but not contained in $\pi L_0$. Let $\mathcal{P}$ denote the set of lattice preservers in $K(\varphi)$. Notice that every lattice preserver preserves a lattice in $\mathcal{L}$ since for every finitely generated $F_0Q$-submodule $L$ of $V$ there is an integer $a$ such that $\pi^aL\in\mathcal{L}$.

\begin{prop} The set $\mathcal{P}$ has a maximal element.
\end{prop}
\begin{proof} By the Lemma, $\mathcal{P}$ is non-empty because it contains the subring of $K(\varphi)$ consisting of elements that preserve $L_0$. By Zorn's Lemma, it will be enough to prove that $\mathcal{P}$ is chain complete.

Let $\{B_\alpha\}_{\alpha\in \mathcal{A}}$ be a chain in $\mathcal{P}$. For each $\alpha\in \mathcal{A}$ let $L_\alpha$ be the largest $F_0Q$-lattice in $\mathcal{L}$ such that $B_\alpha L_\alpha\subseteq L_\alpha$. $L_\alpha$ exists because $L_0$ is a Noetherian $F_0Q$-module.

Now if $B_\alpha\subseteq B_\beta$ then $B_\alpha L_\beta\subseteq L_\beta$ so $L_\beta\subseteq L_\alpha$ by the maximality of $L_\alpha$. Hence the $L_\alpha$ form a descending chain in $\mathcal{L}$.  We claim that $L_\infty := \bigcap L_\alpha$ is also in $\mathcal{L}$. Since $L_0$ has Krull dimension $1$ as a $F_0Q$-module by the Lemma, the chain $\{L_\alpha\}_{\alpha \in \mathcal{A}}$ has deviation at most $1$ and so any well-ordered subchain has order-type $<\omega^2$. By passing to a final segment of the chain $\{L_\alpha\}_{\alpha \in \mathcal{A}}$ we may assume that its order-type is either $\omega$ or a singleton. In the latter case $L_\infty =L_\alpha$ for some $\alpha$ so $L_\infty \in \mathcal{L}$, and in the former case the claim follows from Lemma \ref{Latt}.

Finally the subring $B_\infty$ of $K(\varphi)$ consisting of elements that preserve $L_\infty$ is an upper bound for $\{B_\alpha\}_{\alpha\in \mathcal{A}}$  in $\mathcal{P}$: if $\alpha \in \mathcal{A}$ then $B_\alpha L_\infty \subseteq B_\beta L_\beta \subseteq L_\beta$ for all $\beta \geq \alpha$, so $B_\alpha L_\infty \subseteq \bigcap_{\beta \geq \alpha} L_\beta = L_\infty$ and $B_\alpha \subseteq B_\infty$.
\end{proof}

\subsection{Theorem}\label{MaxDVR} Every maximal lattice preserver is a discrete valuation ring.
\begin{proof} Let $B$ be a maximal element of $\mathcal{P}$, and choose $L\in\mathcal{L}$ such that $BL\subseteq L$. Note that the maximality of $B$ forces it to be equal to $\{\theta \in K(\varphi) : \theta L \subseteq L\}$.

First we show that $B$ is local. We pick $f\in B$ and show that one of $f$ and $1-f$ must be a unit in $B$. Since $L$ is $\pi$-adically complete and since each $L/ \pi^a L$ has finite length as a $F_0Q$-module by Lemma \ref{MaxLatPres}, it follows from Fitting's Lemma (see the proof of \cite[Theorem I.10.4]{Feit}) that $L$ decomposes as $U\oplus W$ where
\[ U=\{u\in L\mid \lim_{i\to \infty}f^i(u)=0\} \]
and
\[ W=\bigcap_{i\geq 0} f^iL. \]
Because $B$ is commutative, both $U$ and $W$ are $B$-modules as well as $F_0Q$-modules. Thus $U_K \oplus W_K$ is a decomposition of $V = L_K$ into a direct sum of $Q$-$K(\varphi)$-bimodules. But $V$ is a simple bimodule by construction (see $\S \ref{MicroQuillen}$), so either $U=0$ or $W=0$. If $U=0$, then $f$ is injective and $W=L$ so $f$ is surjective; thus $f^{-1}\in B$. Otherwise $W=0$ and $U=L$, so $\lim_{n\to \infty}f^n=0$ and $1-f$ is a unit in $B$.  Hence $B$ is a local ring with maximal ideal $\fr{m}$, say.

Next, we will show that $\fr{m}$ is invertible as a fractional ideal. To that end define $\fr{m}^{-1}=\{x\in K(\varphi)\mid x\fr{m}\subseteq B\}$. We will show that $\fr{m}^{-1}\fr{m}$ is $B$; certainly it is contained in $B$ and contains $\fr{m}$ since $1\in\fr{m}^{-1}$. Since $\fr{m}$ is a maximal ideal in $B$ it thus suffices to show that $\fr{m}^{-1}\fr{m}\neq\fr{m}$. Now $\fr{m}L$ is a $F_0Q$-submodule of $L$ and so is finitely generated as such since $F_0Q$ is Noetherian. Thus we may find $f_1,\ldots,f_u\in\fr{m}$ such that $\fr{m}L=\sum_{i=1}^u f_iL$. Let $J$ be the ideal in $B$ generated by $f_1,\ldots,f_u$. The argument above shows that each $f_i$ acts topologically nilpotently on $L$, so we can find an integer $m \geq 1$ such that $J^mL\leq \pi L$. Choose $m$ to be the least such; then we can find $x\in J^{m-1}$ such that $xL \nsubseteq \pi L$. Thus $\frac{x}{\pi}$ is not in $B$ but $\frac{x}{\pi}\fr{m}L\leq L$ and therefore $\frac{x}{\pi}\in \fr{m}^{-1}$. If $\fr{m}^{-1}\fr{m}=\fr{m}$ then $\frac{x}{\pi}\fr{m}L\leq \fr{m}L$, and the subring $B'$ of $K(\varphi)$ generated by $B$ and $\frac{x}{\pi}$ preserves $\fr{m} L$. But then $B' \in \mathcal{P}$ and $B'$ strictly contains $B$ which contradicts the maximality of $B$.

Now every invertible fractional ideal in an integral domain is a finitely generated projective module of rank one by the Dual Basis Lemma \cite[Lemma 3.5.2]{MCR}. Since $B$ is local it follows that $\fr{m}$ is a principal ideal generated by $\tau$, say. Moreover we've seen that $\tau$ acts topologically nilpotently so the $\fr{m}$-adic filtration of $B$ is separated. It follows that for every non-zero element $x$ of $B$ there is a non-negative integer $n$ such that $x\in (\tau^n)\backslash (\tau^{n+1})$. Thus $x=y\tau^n$ for some unit $y\in B$. Therefore $B$ is a discrete valuation ring.
\end{proof}

\subsection{Dimension theory}\label{Gabber}
We retain the notation of the previous subsections and we also impose the additional condition that $\gr_0 A$ is Gorenstein. We will need this so that we can apply Gabber's maximality principle. We recall the statement of this now.

\begin{thm}[Gabber's Maximality Principle] Let $C$ be an Auslander--Gorenstein ring and let $X$ be a finitely generated pure $C$-module contained in a $C$-module $Y$ (not necessarily finitely generated) such that every finitely generated submodule of $Y$ is pure. Then $Y$ contains a unique largest finitely generated submodule $Z$ containing $X$ such that \[j_C(Z/X)\geq j_C(X) + 2 \]\end{thm}
\begin{proof} See \cite[Theorem 1.14]{Bj89}.\end{proof}

We will also need the following preparatory results.

\begin{lem} Suppose that $\gr_0 A$ is Gorenstein and $N$ is a finitely generated $F_0A$-module. Then
\be\item $F_0A$ is Auslander--Gorenstein.
\item If $N$ is $\pi$-torsion-free then $j_{F_0A}(N)=j_A(N_K)$.
\item If $L$ is  an $F_0A$-lattice in $M$ and $N$ is a submodule of $M/L$ such that $Q_T(N)=0$ then $j_{F_0A}(N) \geq j_{F_0A}(L) + 2$.
\ee
\end{lem}
\begin{proof} Part (a) follows from Proposition \ref{AG}(b) since $\gr F_0A\cong(\gr_0 A)[s]$.

(b) First notice that $\Ext^i_A(N_K,A)\cong \Ext^i_{F_0A}(N,F_0A)_K$ for each $i\ge 0$. Thus $j:=j_{F_0A}(N)\leq j_A(N_K)$ and to prove the equality we must show that $\Ext^j_{F_0A}(N,F_0A)$ is not $\pi$-torsion. Since $d_{F_0A}$ is finitely partitive by Proposition \ref{AG}(a), we have that $j_{F_0A}(N/\pi N)>j$. By considering the long exact sequence for $\Ext_{F_0A}(-,F_0A)$ associated to the short exact sequence $0\to N\to N\to N/\pi N\to 0$ we deduce that $\pi\colon \Ext^j_{F_0A}(N,F_0A)\to \Ext^j_{F_0A}(N,F_0A)$ is an injection.

(c)  Since $N \subseteq M/L$ is finitely generated, it is contained in $\pi^{-m}L$ for some $m \geq 1$. By considering the filtration on $N$ induced by $L < \pi^{-1}L < \cdots < \pi^{-m}L$ we may reduce to the case when $m = 1$, so that $\pi N = 0$. Now $N$ is a $\gr_0 A$-submodule of $\pi^{-1}L / L \cong L / \pi L$, so $\Supp(N)$ is contained in $\Supp(L / \pi L)$. Since $L$ and $F_0M$ are both $F_0A$-lattices in $M$, $\Supp(L / \pi L) = \Supp(F_0M / \pi F_0M)$ by \cite[Chapter III, Lemma 4.1.9]{LVO}, so $L/\pi L$ is annihilated as a $\gr_0A$-module by some product of the minimal primes $P_1,\ldots, P_r$ above $\Ann_{\gr_0A}(\gr_0 M)$. On the other hand, $Q_T(N) = 0$ and $\pi N = 0$ together imply that $N$ is a $T$-torsion $\gr_0 A$-module by Corollary \ref{Micro}, so $tN = 0$ for some $t \in T$.

Since $\gr_0A$ is Gorenstein and $t \notin P_i$ for all $i$,
\[j_{\gr_0 A}(\gr_0A / \langle P_i, t \rangle) > j_{\gr_0A}(\gr_0A / P_i)\quad\mbox{for all}\quad i\]by Proposition \ref{AG}(a). Choose a finite filtration of $\pi^{-1}L / L$ by $\gr_0A$-submodules where each subquotient is killed by some $P_i$. The restriction of this filtration to $N \subseteq \pi^{-1}L / L$ then shows that $j_{\gr_0A}(N) > \min_i j_{\gr_0A}(\gr_0 A / P_i) = j_{\gr_0A}(L / \pi L)$.

We can now apply the Rees Lemma \cite[Lemma 1.1]{ASZ2} twice to obtain
\[j_{F_0A}(N) = j_{\gr_0 A}(N) + 1 \geq j_{\gr_0 A}(L / \pi L) + 2 = j_{F_0A}(L / \pi L) + 1.\]
Because $F_0A$ is Auslander-Gorenstein by (a) and $d_{F_0A}$ is finitely partitive by Proposition \ref{AG}(a),
\[j_{F_0A}(L / \pi L) \geq j_{F_0A}(L) + 1\]
and the result follows.
\end{proof}\subsection{Finding a global regular $\varphi$-lattice}\label{Quillen}
\begin{prop} Suppose that $M$ is a simple $A - K(\varphi)$-bimodule. Then $M$ has a regular $\varphi$-lattice.
\end{prop}
\begin{proof} Since $Q_T(M)$ surjects onto $V$, the natural map $\phi : M \to V$ is non-zero and therefore $\ker \phi$ is a proper $A$-submodule of $M$. Since $V$ is a $K(\varphi)$-module quotient of $Q_T(M)$, $\ker \phi$ is in fact an $A-K(\varphi)$-subbimodule. Therefore by our assumption $\phi$ is an injection and we will use it to identify $M$ with an $A$-submodule of $V$. By Proposition \ref{MaxLatPres} we can find a maximal element $B$ of $\mathcal{P}$. Let $L\in\mathcal{L}$ be an $F_0Q$-lattice preserved by $B$; we will show that $L \cap M$ is the required regular $\varphi$-lattice.

Let $S = \{s \in A : \gr s \in T\}$ be the microlocal Ore set arising from $T \subseteq \gr_0 A$. Alternatively put, $S$ is just the preimage of $T$ in $F_0A$; since $1 \in T$ we see that $1 + \pi F_0A \subseteq S$ and therefore $\pi$ is in the Jacobson radical of $F_0A_S$.

Since $M_S$ is dense in $Q_T(M)$, $(L\cap M_S)+\pi L=L$ and so $L = F_0Q.(L \cap M_S)$. Since also $L \cap M_S = F_0A_S.(L\cap M)$, $L$ is generated by $L \cap M$ as an $F_0Q$-module. Because $L$ is Noetherian as an $F_0Q$-module, we can find a finitely generated $F_0A$-submodule $X$ of $L\cap M$ such that $F_0Q.X = L$. Since $L \cap M$ generates $M$ as a $K$-vector space, by enlarging $X$ if necessary we may assume that $X$ is also an $R$-lattice in $M$.

Let $N$ be the largest $A$-submodule of $M$ with $d_A(N) < d_A(M)$. Then $N$ is a proper characteristic submodule of $M$ and is therefore stable under every element of $K(\varphi) \subseteq \End_A(M)$. Since $M$ is a simple $A - K(\varphi)$-bimodule, $N = 0$ so $M$ is pure. Now $F_0A$ is Auslander--Gorenstein by Lemma \ref{Gabber}(a), so every finitely generated $F_0A$-submodule of $M$ is pure by Lemma \ref{Gabber}(b). Thus $M$ contains a unique largest finitely generated submodule $Z$ containing $X$ such that $j_{F_0A}(Z/X)\geq j_{F_0A}(X) + 2$ by Gabber's Maximality Principle, Theorem \ref{Gabber}. We will show that $L \cap M$ is contained in $Z$ and so is finitely generated over $F_0A$ since $F_0A$ is Noetherian.

Suppose that $X'$ is any finitely generated $F_0A$-submodule of $L\cap M$ containing $X$. Then $Q_T(X')=Q_T(X)=L$; thus $Q_T(X'/X)=0$ and so
\[ j_{F_0A}(X'/X) \geq j_{F_0A}(X) + 2\]
by Lemma \ref{Gabber}(c). Therefore $X' \leq Z$. This means that every finitely generated submodule of $L\cap M$ is contained in $Z$ and so $L\cap M$ is itself a submodule of $Z$ as claimed.

Finally, since $L=F_0Q.(L\cap M)$,
\[\{\theta \in K(\varphi) \st \theta(L \cap M) \subseteq L\cap M\} = \{\theta \in K(\varphi) \st \theta(L) \subseteq L \}\]
is a discrete valuation ring by Theorem \ref{MaxDVR}.
\end{proof}

We can finally state and prove our version of Quillen's Lemma.

\begin{cor} Let $A$ be an almost commutative affinoid $K$-algebra with commutative Gorenstein slice. Then every simple endomorphism of every finitely generated $A$-module is algebraic over $K$.
\end{cor}
\begin{proof} As we explained in $\S \ref{RegLat}$, we can view $M$ as an $A-K(\varphi)$-bimodule. Since $M$ is a Noetherian $A$-module, we can find a simple $A-K(\varphi)$-bimodule quotient $\overline{M}$ of $M$; note that the $A$-linear endomorphism $\overline{\varphi}$ of $\overline{M}$ induced by $\varphi$ is still simple. Then by Proposition \ref{Quillen}, $\overline{M}$ has a regular $\overline{\varphi}$-lattice so by Corollary \ref{RegLat}, $\overline{\varphi}$ is algebraic over $K$. But the natural map $K(\varphi) \to K(\overline{\varphi})$ is an isomorphism since $K(\varphi)$ is a field, so $\varphi$ is also algebraic over $K$.\end{proof}

\begin{rmks} \hfill
\be
\item The same proof shows that if $A$ is any complete filtered $K$-algebra such that $F_0A$ is an $R$-lattice in $A$ and $\gr_0A$ is a finitely generated Gorenstein commutative $k$-algebra then the conclusion of the corollary holds.

\item It may be possible to relax the assumption that $\gr_0 A$ is Gorenstein by using the version of Gabber's maximality principle found in \cite{YZ} based around Auslander dualising complexes. However we do not know whether all almost commutative affinoid $K$-algebras have an Auslander dualising complex.
\ee
\end{rmks}

\section{Modules over completed enveloping algebras}
\subsection{Finite dimensional modules}\label{CdimAmod} We begin our study of finitely generated modules over completed enveloping algebras with the following rather general result.

\begin{prop} Let $A$ be a complete doubly filtered $K$-algebra such that $\Gr(A)$ is a connected graded polynomial algebra over $k$, and let $M$ be a finitely generated $A$-module.
\be
\item $M$ is finite dimensional over $K$ if and only if $\dim \Ch(M) = 0$.
\item If $A$ has at least one non-zero finite dimensional module $V$, then $\injdim A = \dim \Gr(A)$ and $d(M) = \dim \Ch(M)$.
\ee\end{prop}
\begin{proof}
(a) Choose a good double filtration $(F_0M, F_\bullet \gr_0M)$ on $M$ using Proposition \ref{GoodDouble}(b). Then $\dim \Ch(M) = 0$ if and only if $\Gr(M)$ is finite dimensional over $k$. Now if $\Gr(M)$ is finite dimensional over $k$ then the double filtration is good for $M$ as a doubly filtered $K$-module and therefore $M$ is finite dimensional over $K$ by Lemma \ref{GoodDouble}(a). Conversely, if $M$ is finite dimensional over $K$ then $F_0M$ has to be finitely generated over $R$ by Proposition \ref{Latt} because it is an $R$-lattice in $M$, so $\gr_0M$ and $\Gr(M)$ are finite dimensional over $k$.

(b) By part (a), $\dim \Ch(V) = 0$ so $\dim \Gr(A) = j_A(V)$ by Theorem \ref{CharVarAff}. Clearly $j_A(V) \leq \injdim A \leq \gldim A$. Now $\pi$ is a central regular element of $F_0A$ contained in the Jacobson radical of $F_0 A$ and $A = (F_0A)_\pi$. Therefore
\[\gldim A \leq \gldim F_0 A - 1 = \gldim \gr_0 A \leq \gldim \Gr(A) = \dim \Gr(A)\]
by \cite[\S 7.4.4, 7.3.7, 7.5.3(iii)]{MCR} and \cite[Corollary I.7.2.2]{LVO}. The first statement follows, and we obtain the second from Theorem \ref{CharVarAff}.
\end{proof}

Now let $\fr{g}$ be an $R$-Lie algebra which is free of finite rank as an $R$-module, and let $A$ denote the almost commutative affinoid $K$-algebra $\hugnK$. Then $\Gr(A) \cong \Syk{g}$ is commutative and Gorenstein with $\dim \Gr(A) = \dim \frk{g}$, and we always have the trivial $A$-module $K = A / A\fr{g}_K$ which is one-dimensional over $K$. Thus we obtain the following

\begin{cor} Let $M$ be a finitely generated $A = \hugnK$-module. Then $d(M) = \dim \Ch(M)$ and $d(M) = 0$ if and only if $M$ is finite dimensional over $K$.\end{cor}

Note the Proposition fails for the affinoid Weyl algebras of $\S \ref{AffWeyl}$ because these never have any non-zero modules which are finite dimensional over $K$ by Bernstein's Inequality, Theorem \ref{BernIneq}.

\label{ModCompEnv}
\subsection{Finite dimensional modules}\label{FDAmod}
We now continue with the notation of $\S\ref{AlgGps}$, with the additional assumptions that the algebraic group $\mb{G}$ is semisimple and that the field $K$ has characteristic zero. Since the usual enveloping algebra $\ugK$ is a $K$-subalgebra of $A = \hugnK$, we can view every $A$-module as a $\ugK$-module by restriction.

\begin{prop} Restriction induces an equivalence of abelian categories between finite dimensional $A$-modules and finite dimensional $\ugK$-modules.
\end{prop}
\begin{proof} Let $V$ be a finite dimensional $\ugK$-module. By Weyl's Theorem \cite[Theorem 1.6.3]{Dix}, $V$ is a direct sum of simple $\ugK$-submodules, and each simple submodule has a highest weight by \cite[Proposition 7.2.1(i)]{Dix}. Now the proof of \cite[Theorem 27.1(b)]{Hum} shows that we can find an $R$-lattice $L$ in $V$ which is stable under $\U{g}$. Hence it is also stable under $\U{g}_n$. Since $L$ is finitely generated over $R$, it is $\pi$-adically complete and is therefore an $F_0A = \h{\U{g}_n}$-module. Hence $V$ is also an $A$-module, so the restriction functor is essentially surjective on objects. This functor is clearly faithful, so it remains to show that it is full.

Let $V, W$ be two finite dimensional $A$-modules and let $f : V \to W$ be a $\ugK$-module homomorphism. Choose $\fr{g}$-stable $R$-lattices $L \subseteq V$ and $M \subseteq W$. Since $L$ is finitely generated over $R$, $\pi^m f(L)\subseteq M$ for some integer $m$ so $f$ is continuous. Since $\ugK$ is dense in $A$ it follows that $f$ is actually an $A$-module homomorphism.
\end{proof}

Recall the set of integral dominant weights $\Lambda^+ = \mathbb{N} \omega_1 \oplus \cdots \oplus \mathbb{N} \omega_l \subseteq \fr{h}^\ast$ from $\S\ref{BBVanishing}$, and the isomorphism $i : \fr{t} \tocong \fr{h}$ from \S \ref{HCmap}.

\begin{cor} Every finite dimensional $A$-module is semisimple. For each $\lambda \in \Lambda^+$ there is a unique up to isomorphism simple finite dimensional $A$-module $L(\lambda)$ with highest weight $\lambda \circ i$, and all finite dimensional simple $A$-modules are of this form.
\end{cor}
\begin{proof} This follows from \cite[$\S$1.6.3, 7.1.11, 7.2.2]{Dix}.\end{proof}

We could have also constructed a $\fr{g}$-stable $R$-lattice in each $L(\lambda)$ by considering the co-ordinate ring $\O(\Ex{\Fl})$ of the basic affine space $\Ex{\Fl} = \mb{G} / \mb{N}$ (see $\S\ref{AlgGps}$). This is a $\Lambda^+$-graded $\fr{g}$-stable subring of the usual representation ring $\mathcal{O}(\Ex{\Fl}_K) \cong \bigoplus_{\lambda \in \Lambda^+} L(\lambda)$, so its homogeneous components give the required $\fr{g}$-stable lattices.

\subsection{The centre}\label{Centre}
Recall the Harish-Chandra homomorphism $\phi : \U{g}^{\mb{G}} \longrightarrow \U{t}$ from $\S \ref{HCmap}$. This is a morphism of deformable $R$-algebras, and applying the deformation and $\pi$-adic completion functors we obtain the deformed Harish-Chandra homomorphism
\[\h{\phi_{n,K}} : \hugGnK \longrightarrow \hutnK\]
which we will denote by $\h{\phi} : Z \to \Ex{Z}$ in an attempt to alleviate the notation. Now the Weyl group $\mb{W}$ of $\mb{G}$ acts on $\fr{t}^\ast_K$ by
\[w \bullet \lambda = w(\lambda + \rho') - \rho',\]
where $\rho' := i^\ast(\rho) = \rho \circ i \in \fr{t}^\ast$ denotes the image of $\rho \in \fr{h}^\ast$ under the dual isomorphism $i^\ast : \fr{h}^\ast \tocong \fr{t}^\ast$. In fact, $\rho'$ is equal to half the sum of the $\mb{T}$-roots on $\fr{n}^+$. If we view $\U{t}_K$ as an algebra of polynomial functions on $\fr{t}^\ast_K$, we get a corresponding `dot'-action of $\mb{W}$ on $\U{t}_K$. This action preserves the $R$-subalgebra $\U{t}_n$ of $\U{t}_K$ and therefore extends to a natural `dot'-action of $\mb{W}$ on $\Ex{Z} = \hutnK$.

\begin{prop} Suppose that $p$ is a very good prime for $\mb{G}$. \be
\item The algebra $Z$ is contained in the centre of $A$.
\item The map $\h{\phi}$ is injective, and its image is the ring of invariants $\Ex{Z}^{\mb{W} \bullet}$.
\item The algebra $\Ex{Z}$ is free of rank $|\mb{W}|$ as a module over $\Ex{Z}^{\mb{W} \bullet}$.
\item $\Ex{Z}^{\mb{W} \bullet}$ is isomorphic to a Tate algebra $K\langle S_1,\ldots, S_l \rangle$ as a complete doubly filtered $K$-algebra.
\ee\end{prop}
\begin{proof} (a) The algebra $\U{g}_K^{\mb{G}}$ is central in $\U{g}_K$ by \cite[Lemma 23.2]{Hum}, so it is also contained in the centre of $A$ since $\U{g}_K$ is dense in $A$. But $\U{g}_K^{\mb{G}}$ is dense in $Z$, so $Z$ is also central in $A$.

(b) By the classical result of Harish-Chandra \cite[Theorem 7.4.5]{Dix}, $\phi$ sends $\U{g}^{\mb{G}}_K$ onto $\U{t}^{\mb{W} \bullet}_K$, so $\h{\phi}(Z)$ is contained in $\Ex{Z}^{\mb{W} \bullet}$. This algebra is complete and doubly filtered, and $\Gr(\Ex{Z}^{\mb{W}\bullet})$ can be naturally identified with $\Syk{t}^{\mb{W}_k}$.

Consider the morphism of complete doubly filtered $K$-algebras $\alpha : Z \to \Ex{Z}^{\mb{W}}$ induced by $\h{\phi}$. Its associated double graded map $\Gr(\alpha) : \Gr(Z) \to \Gr(\Ex{Z}^{\mb{W}})$ can naturally be identified with the isomorphism $\psi_k : \Syk{g}^{\mb{G}_k} \tocong \Syk{t}^{\mb{W}_k}$ by Corollary \ref{GrComp} and Proposition \ref{RingsofI}.  Hence $\Gr(\alpha)$ is an isomorphism and therefore by completeness $\alpha$ is also an isomorphism.

(c) By \cite[Th\'eor\`eme 2(c)]{Dema}, $\Syk{t}$ is a free graded $\Syk{t}^{\mb{W}_k}$-module of rank $|\mb{W}|$. It follows from Lemma \ref{GoodDouble}(a) that $\Ex{Z}$ is finitely generated over $Z$, and it is easy to see that in fact it's free of rank $|\mb{W}|$.

(d) By \cite[Corollaire du Th\'eor\`eme 3]{Dema}, $\Syk{t}^{\mb{W}_k}$ is a polynomial algebra in $l$ homogeneous generators over $k$. Fix some (double) lifts $s_1,\ldots,s_l \in \U{t}_n^{\mb{W} \bullet}$ of these generators, and define an $R$-algebra homomorphism $R[S_1,\ldots, S_l] \to \Ex{Z}^{\mb{W}\bullet}$ by sending $S_i$ to $s_i$. This extends to an isomorphism $K \langle S_1,\ldots, S_l \rangle \to \Ex{Z}^{\mb{W} \bullet}$ of complete doubly filtered $K$-algebras.
        \end{proof}
It was shown in \cite[Theorem 5.2.1]{ArdThesis} that in fact $Z$ is the whole centre of $A$ when $n=0$. We plan to show in a later paper that this is true for any $n \geq 0$.

\textbf{From now on we will assume that $n > 0$ and that $p$ is a very good prime for $\mb{G}$}.

\subsection{$Z$-locally finite modules}\label{CentralChar}
Let $M$ be an $A$-module. Since $Z$ is central in $A$, the action of $Z$ on $M$ induces an $K$-algebra homomorphism
\[\chi_M : Z \to \End_A(M).\]
which we call \emph{the central character} of $M$.
\begin{defn} We say that $M$ is \emph{$Z$-locally finite} if $\dim_K Z.m < \infty$ for all $m \in M$.
\end{defn}
It is easy to see that if $M$ is finitely generated over $A$ then $M$ is $Z$-locally finite if and only if $\dim_K \im \chi_M < \infty$. It is also clear that $Z$-locally finite modules are closed under taking submodules, quotient modules and extensions.

We are now ready to prove Theorem \ref{QuillenIntro} from the introduction.

\begin{thm} Let $M$ be a simple $A$-module. Then $\im \chi_M$ is a finite field extension of $K$, so $M$ is $Z$-locally finite.
\end{thm}
\begin{proof} By Schur's Lemma, $\End_A(M)$ is division ring. It is algebraic over $K$ by Corollary \ref{Quillen}, since $n > 0$ by assumption. So $\im \chi_M$ is an integral domain which is algebraic over $K$; it is therefore a field and $\ker \chi_M$ is a maximal ideal of $Z$. But $Z \cong K\langle S_1,\ldots, S_l\rangle$ is a Tate algebra by Proposition \ref{Centre} and every maximal ideal of $Z$ has finite codimension over $K$ by \cite[Theorem 3.2.1(5)]{FvdPut}.
\end{proof}

Let $M$ be a finitely generated $A$-module. By applying the Theorem to a simple factor module of $M$, we see that $M$ has a non-zero $Z$-locally finite quotient. In fact, a stronger statement is true.

\begin{prop}
Let $M$ be a finitely generated $A$-module with $d(M) \geq 1$. Then $M$ has a $Z$-locally finite quotient $N$ such that $d(N) \geq 1$.
\end{prop}
\begin{proof}
Since $d(M) = \dim \Ch(M) \geq 1$ by Corollary \ref{CdimAmod}, $\gr_0 M$ is infinite-dimensional over $k$. We can therefore find an element $f \in \fr{g}_k$ such that $(\gr_0 M)_f \neq 0$. Hence $Q_f(M)$ is a finitely generated \emph{non-zero} module over the microlocalisation $Q_f(A)$, and we may choose some simple quotient $W$ of $Q_f(M)$ as a $Q_f(A)$-module.

The degree zero part $F_0Q_f(A)$ is an $R$-lattice in $Q_f(A)$ and the slice $\gr_0 Q_f(A)$ is isomorphic to the localisation $(\gr_0A)_f \cong \Syk{g}[t] / \langle tf - 1 \rangle$. So $\gr_0 Q_f(A)$ is a finitely generated Gorenstein commutative $k$-algebra, which means that it is possible to apply Corollary \ref{Quillen} to the algebra $Q_f(A)$. The central subalgebra $Z$ of $A$ is still central in $Q_f(A)$ and therefore acts on $W$ by $Q_f(A)$-module endomorphisms. Therefore $W$ is $Z$-locally finite by the proof of the Theorem above.

Now let $N$ be the image of $M$ in $W$, and suppose for a contradiction that $d(N) = 0$. Then $N$ is finite dimensional over $K$ by Proposition \ref{CdimAmod}, so we can find a $\fr{g}$-stable lattice $L$ inside $N$ by $\S \ref{FDAmod}$. Since $n > 0$, $\pi^n \fr{g} L \leq \pi L$, so $\gr_0A = \Syk{g}$ acts on $\gr_0N = L / \pi L$ through its augmentation. Since $f \in \fr{g}_k$, it follows that $Q_f(N) = 0$. This is a contradiction, because $Q_f(N)$ surjects onto the simple $Q_f(A)$-module $W$. Thus $N$ is the required $Z$-locally finite quotient of $M$ which satisfies $d(N) \geq 1$.
\end{proof}

\subsection{Base change}\label{BaseC}
We want to apply Theorem \ref{BBThm} and Corollary \ref{LocFun} to our finitely generated $Z$-locally finite $A$-module $M$. However $\im \chi_M$ could be strictly bigger than $K$; also we need to produce a $\hunlK$-module for some appropriate weight $\lambda \in \fr{h}^\ast_K$. We will solve both problems by passing to a finite field extension of $K$.

For any finite field extension $K'$ of $K$ with ring of integers $R'$, let $\pi' \in R'$ be a uniformizer and let $e$ be the ramification index of $K'$ over $K$, so that $\pi R' = \pi'^e R'$. Let $\mb{G}' = R' \times_R \mb{G}$, $\mb{H}' := R' \times_R \mb{H}$, $\fr{g}' = R' \otimes_R \fr{g}$, $\fr{t}' = R' \otimes_R \fr{t}$ and $\fr{h}' = R' \otimes_R \fr{h}$ be the corresponding base-changed objects. Since $\fr{t}$ has finite rank over $R$, we will identify $\fr{t}'^\ast := \Hom_{R'}(\fr{t}', R')$ with $R' \otimes_R \fr{t}^\ast$. The isomorphism $i : \fr{t} \tocong \fr{h}$ extends to an isomorphism $i : \fr{t}' \tocong \fr{h}'$.

\begin{lem} Let $K'/K$ be a finite field extension. Then
\be
\item $K' \otimes_K Z \cong \widehat{ U(\fr{g}')_{ne,K'}^{\mb{G}'} }$.
\item $K' \otimes_K \Ex{Z} \cong \widehat{ U(\fr{t}')_{ne,K'} }$.
\ee
\end{lem}
\begin{proof} We know that $U(\fr{g}')^{\mb{G}'} \cong R' \otimes_R \Ui{g}{G}$ by \cite[\S I.2.10(3)]{Jantzen}. Now both parts follow from Lemma \ref{basechange2}(c).
\end{proof}

Let $A' := K' \otimes_K A$ and note that $A' \cong \h{U(\fr{g}')_{en,K'}}$ by Lemma \ref{basechange2}(c). Recall the central quotients $\h{\mathcal{U}^\lambda_{en,K'}}$ of $\h{U(\fr{g}')_{en,K'}}$ from $\S \ref{GlobSec}$ for each weight $\lambda \in \pi'^{-ne} \fr{h}'^\ast$. 

\begin{thm}
Let $M$ be a finitely generated $Z$-locally finite $A$-module. Then there exists a finite extension $K'$ of $K$ with ramification index $e$, a weight $\lambda \in \pi'^{-ne}\fr{h}'^\ast$ and a finitely generated $\h{\mathcal{U}^\lambda_{en,K'}}$-module $N$ such that $d(M) = d(N)$.
\end{thm}
\begin{proof}Choose a submodule $M_0$ of $M$ maximal subject to having $d(M/M_0) = d(M)$. Replacing $M$ by $M / M_0$ we may assume that $M$ is \emph{$d$-critical} in the sense that $d(M/M') < d(M)$ for any non-zero proper $A$-submodule $M'$ of $M$. In particular, $M$ must be \emph{$d$-pure}: $d(M') = d(M)$ for any non-zero submodule $M'$ of $M$.

Let $P := \ker \chi_M = \Ann_Z(M)$ and suppose that $xy \in P$ for some $x,y \in Z$. If $x \notin P$ then $xM$ is a non-zero submodule of $M$, so $d(xM) = d(M)$. Because $xyM = 0$, multiplication by $x$ induces an $A$-module surjection $M / yM \twoheadrightarrow xM$, whence $d(M / yM) \geq d(M)$. This is only possible if $yM = 0$ since $M$ is $d$-critical. Hence $P$ is a prime ideal in $Z$; since $M$ is $Z$-locally finite, $P$ is in fact maximal. 

Next, $\Ex{Z}$ is a finitely generated $Z$-module via $\widehat{\phi}$  by Proposition \ref{Centre}, so $\Ex{Z}$ is an integral extension of $Z$. Thus by \cite[Corollary 5.9, Theorem 5.10]{AMac} for example, we may find a maximal ideal $\fr{m}$ of $\Ex{Z}$ with $\widehat{\phi}^{-1}(\fr{m})=P$ and define $K':=\Ex{Z}/\fr{m}$, a finite extension of $K$. Extend the natural surjection $\Ex{Z} \twoheadrightarrow K'$ to a $K'$-algebra homomorphism $\theta : K'\otimes_K \Ex{Z} \to K'$. By the Lemma, $K' \otimes_K \Ex{Z} \cong \widehat{ U(\fr{t}')_{ne,K'} }$ is a Tate algebra, so $\theta$ sends the power-bounded subset $\pi'^{ne} \fr{t}'$ of $K'\otimes_K \Ex{Z}$ to the ring of integers $R'$ in $K'$ and we can find an element $\lambda \in \pi^{-ne}\fr{h}'^\ast$ such that $\lambda \circ i$ is the restriction of $\theta$ to $\pi'^{ne} \fr{t}'$.

Now $N := K' \otimes_{Z/P} M$ is a finitely generated $A'$-module and Lemma \ref{basechange1} tells us that $d_{A'}(N) = d_A(N) = d_A(M)$. By the Lemma above, $K' \otimes_K Z \cong \widehat{ U(\fr{g}')_{ne,K'}^{\mb{G}'} } $, and this algebra acts on $N$ via $\lambda \circ i \circ (1 \otimes \widehat{\phi})$ by construction. It now follows from Theorem \ref{GlobSec}(a) that $N$ is a finitely generated $\h{\mathcal{U}^{\lambda}_{en,K'}}$-module as required.  
\end{proof}

\subsection{Using the $W$-action}\label{WAct} Theorem \ref{BaseC} tells us that after making an appropriate base change, we may assume that our finitely generated $Z$-locally finite $A$-module has a $K$-rational central character $\lambda \circ i \circ \widehat{\phi}$. However Corollary \ref{LocFun} requires $\lambda$ to be $\rho$-dominant;  our next result shows that we may achieve this by using the action of the Weyl group.

\begin{lem} For any weight $\mu \in \fr{h}^\ast_K$ there exists $w \in \mb{W}$ such that $w(\mu)$ is dominant.
\end{lem}
\begin{proof} Let us define a binary relation $\geq$ on $\fr{h}^\ast_K$ by $\lambda \geq \mu$ if and only if $\lambda - \mu$ is a linear combination of positive roots with non-negative integer coefficients. Since $K$ has characteristic zero, this is a partial order on $\fr{h}^\ast_K$. Since $\mb{W}$ is finite, we can find an element $\lambda = w(\mu)$ in the $\mb{W}$-orbit of $\mu$ which is maximal with respect to this ordering. If $\lambda(h) \in \{-1,-2,\cdots \}$ for some positive coroot $h\in \fr{h}$ then taking $\alpha$ to be the corresponding positive root, we have
\[ s_{\alpha}(\lambda) = \lambda - \lambda(h) \alpha\]
lies in the $\mb{W}$-orbit of $\mu$ and $s_{\alpha}(\lambda) > \lambda$, which contradicts the maximality of $\lambda$. So $\lambda(h)\notin \{-1,-2,\cdots\}$ for any positive coroot $h\in \fr{h}$ and hence $\lambda = w(\mu)$ is dominant.
\end{proof}\subsection{Springer fibres}\label{SprFib}
We will assume throughout $\S \ref{SprFib} - \S \ref{MinNilp}$ that the field $k$ is algebraically closed. We will identify the $k$-points of the scheme $\fr{g}^\ast = \Spec(\Sym_R \fr{g})$ with the dual of the $k$-vector space $\frk{g}$; thus $\fr{g}^\ast(k) = \frk{g}^\ast$. We will also abuse notation and denote the map on $k$-points $f(k) : X(k) \to Y(k)$ induced by a morphism of $R$-schemes $f : X \to Y$ simply as $f : X(k) \to Y(k)$. With these notations, the diagram from $\S\ref{NotnDefCompCh}$ on the level of $k$-points looks as follows:
\[\xymatrix{ & \Ex{T^\ast \Fl}(k) \ar[dl]^{\tau} \ar[dr]_{\beta} & \\
                           \Fl(k) & & \fr{g}^\ast(k) .
}\]
We are interested in the \emph{Springer fibres}, which by definition are the sets $\beta^{-1}(y)$ as $y$ runs over $\frk{g}^\ast$; these are algebraic varieties over $k$.

Let $G$,$B$,$N$ denote the sets of $k$-points of $\mb{G}$,$\mb{B}$,$\mb{N}$ respectively, and let us identify the set $\Fl(k)$ of $k$-points of the flag $R$-scheme $\Fl$ with $G/B = \{gB : g \in G\}$. The group $G$ acts on $\frk{g}$ and on $\frk{g}^\ast$ via the adjoint and coadjoint actions, respectively; if $S$ is a subset of $\frk{g}$, let $S^\perp = \{\lambda \in \frk{g}^\ast : \lambda(S) = 0\}$ denote its annihilator in $\frk{g}^\ast$. Note that $\frk{n} \subseteq \frk{b} \subseteq \frk{g}$ are the Lie algebras of the algebraic groups $N \subseteq B \subseteq G$.

\begin{lem} Let $y \in \frk{g}^\ast$. Then $\tau \beta^{-1} (y)$ is equal to $\{g B \in \Fl(k) : y \in (g.\frk{n})^\perp\}$.
\end{lem}
\begin{proof} Let $gB \in \Fl(k)$. The geometric fibre of the morphism of vector bundles $\varphi : \O_\Fl \otimes \fr{g} \to \Ex{\T}_\Fl$ used in $\S\ref{EnhMom}$ to define the enhanced moment map $\beta$ is just the action map $\frk{g} \to T_{gN} (G/N)$ of $\frk{g}$ on the homogeneous space $G/N$ at the point $gN \in G/N$. Because the action map is surjective by \cite[Proposition II.6.7]{Borel}, this tangent space can be naturally identified with $\frk{g} / g.\frk{n}$. It follows that the restriction of $\beta$ to $\tau^{-1}(gB) = \Ex{T^\ast_{gB} (G/B)}$ is the dual of this action map, which we will identify with the inclusion $(g. \frk{n})^\perp \hookrightarrow \frk{g}^\ast$.  With these identifications in place, it is now clear that $gB \in \tau \beta^{-1}(y)$ if and only if $y \in (g.\frk{n})^\perp$.
\end{proof}

\subsection{Nilpotent orbits}\label{NilpOrb}
We define the \emph{nilpotent cone} in $\frk{g}^\ast$ as the set of zeros of $\mb{G}_k$-invariant polynomials in $\Syk{g} = \O(\frk{g}^\ast)$ with no constant term: $\N^\ast = V( \Sik{g}{G}_+ )$. Thus
\[\mathcal{O}(\N^\ast) = \Syk{g} \otimes_{ \Sik{g}{G} } k.\]
The nilpotent cone $\N = V( S(\frk{g}^\ast)^{\mb{G}_k}_+)$ in $\frk{g}$ is defined similarly. Since we're assuming that the characteristic of $k$ is very good for $\mb{G}$, there is a non-degenerate $G$-invariant bilinear form on $\frk{g}$ which induces a $G$-equivariant isomorphism $\kappa : \frk{g}^\ast \to \frk{g}$ (see \cite[\S3.1.2]{BMR1}). This isomorphism maps $\N^\ast$ onto $\N$.

The nilpotent cone $\N$ is a union of $G$-orbits in $\frk{g}$ called the \emph{nilpotent orbits}. The corresponding $G$-orbits in $\N^\ast$ are called the \emph{coadjoint nilpotent orbits}. It turns out that these are very closely connected with Springer fibres. The next result is well-known, but we give the proof for the benefit of the reader.
\begin{prop} For any $y \in \N^\ast$, we have $\dim \beta^{-1}(y) \leq \dim \Fl - \frac{1}{2} \dim G.y$.
\end{prop}
\begin{proof} Note first that $\dim \tau \beta^{-1}(y) = \dim \beta^{-1}(y)$ for all $y \in \frk{g}^\ast$ because the map $\tau$ is clearly injective on the Springer fibre $\beta^{-1}(y)$. Since we've been assuming from $\S \ref{BBVanishing}$ onwards that $G$ is simply-connected, a result of Springer \cite[Corollary 9.3.4]{BaRi} tells us that there is a $G$-equivariant isomorphism $\eta : \N \to \mathcal{U}$, where $\mathcal{U} \subseteq G$ is the variety of unipotent elements.

Let $u = \eta(\kappa(y)) \in \mathcal{U}$. Since $\kappa( \frk{n}^\perp) = \frk{b}$, Lemma \ref{SprFib} implies that
\[\tau \beta^{-1}(y) = \{gB \in G/B : \kappa(y) \in g. \frk{b}\} = \{gB \in G/B : u\in gBg^{-1}\}.\]
Thus $\tau\beta^{-1}(y)$ is the set of fixed points $(G/B)_u$ of the action of $u$ on $G/B$, and it follows from \cite[Theorem 3.5(a)]{Stein} that
\[ \dim \tau \beta^{-1}(y) = \dim (G/B)_u \leq \frac{1}{2}(\dim C_G(u) - l)\]
where $C_G(u)$ is the centralizer of $u$ in $G$ and $l$ is the rank of $G$. The result is now clear because $\dim C_G(u) = \dim G - \dim G.u = 2 \dim \Fl + l - \dim G.y$.
\end{proof}

\begin{rmks}
\be \hfill
\item It is possible to give a slightly more direct, but longer, proof of this result mimicking the proof of \cite[Theorem 3.5]{Stein} (see also \cite[Theorem 6.8]{Hum2}), using the \emph{Steinberg variety} of triples
\[ \{ (\fr{b}_1, \fr{b}_2, z) \in \Fl(k) \times \Fl(k) \times G.\kappa(y) : z \in \fr{b}_1 \cap \fr{b}_2\}\]
where we now think of $\Fl(k)$ as the set of $G$-conjugates of $\frk{b}$ in $\frk{g}$.
\item In fact, under our assumptions on $p$, equality always holds in the Proposition. This is the `Dimension Formula', originally a conjecture of Grothendieck, and was proven by Steinberg \cite{Stein} as a consequence of the Bala-Carter classification of unipotent classes. The book \cite{Hum2} gives a good overview of this subject; see also \cite{Premet} and \cite{DouRoh}.
\ee\end{rmks}

\subsection{The minimal non-zero nilpotent orbit}\label{MinNilp}
Let $\fr{g}_{\mathbb{C}}$ be the complex semisimple Lie algebra with the same root system $\Phi$ as our group $\mb{G}$, and let $G_{\mathbb{C}}$ denote the adjoint complex algebraic group associated with $\fr{g}_{\mathbb{C}}$. It is known \cite[Remark 4.3.4]{CM} that there is a unique non-zero nilpotent $G_{\mathbb{C}}$-orbit in $\fr{g}_{\mathbb{C}}^\ast$ of minimal dimension, called the \emph{minimal nilpotent orbit}. The dimension of this orbit is even integer, since each coadjoint $G_{\mathbb{C}}$-orbit is a symplectic manifold.

\begin{defn} We let $r$ denote half the dimension of the minimal nilpotent orbit:
\[r := \frac{1}{2} \min \{ \dim G_{\mathbb{C}}.y : 0 \neq y \in \fr{g}_{\mathbb{C}}\}.\]
\end{defn}
\noindent The values of $r$ are well-known. We took the following table from \cite[\S 1.6]{Smith}:
\[\begin{array}{c|cccccccccr} \Phi & A_l & B_l & C_l & D_l & E_6 & E_7 & E_8 & F_4 & G_2 \\ \hline \dim G & l^2+2l & 2l^2 + l & 2l^2 + l & 2l^2 - l & 78 & 133& 248& 52& 14 \\ r & l & 2l - 2 & l & 2l - 3 & 11 & 17 & 29 & 8 & 3 . \end{array}\]
Pommerening proved that structure of nilpotent coadjoint orbits in good characteristic is the same as over $\mathbb{C}$. However we will only need to know the following consequence of the Bala-Carter-Pommerening classification of nilpotent orbits:

\begin{prop} For any non-zero $y \in \N^\ast$, $\frac{1}{2}\dim G.y \geq r$.
\end{prop}
\begin{proof} By \cite[\S 9.2.1]{BaRi}, $\N$ is the set of nilpotent elements in $\frk{g}$. In view of the $G$-equivariant isomorphism $\kappa : \frk{g}^\ast \to \frk{g}$, the required inequality now follows from \cite[Theorems  2.6 and 2.7]{Premet}.
\end{proof}

\subsection{A lower bound for $d(M)$}\label{LowerBound}
We now return to our original setting, dropping the assumption that the field $k$ is algebraically closed. We can finally state and prove the analogue of Smith's Theorem for modules over $\pi$-adically completed enveloping algebras.

\begin{thm} Suppose that $n > 0$. Let $M$ be a finitely generated $\hugnK$-module with $d(M) \geq 1$. Then $d(M) \geq r$.
\end{thm}
\begin{proof} By Proposition \ref{CentralChar}, we may assume that $M$ is $Z$-locally finite. By passing to a finite field extension of $K$ if necessary and applying Theorem \ref{BaseC}, we may further assume that $M$ is a $\hunlK$-module for some $\lambda \in \fr{h}^\ast_K$. Since $\lambda \circ (i \circ \widehat{\phi}) = (w \bullet \lambda) \circ (i \circ \widehat{\phi})$ for any $w \in \mb{W}$ by Proposition \ref{Centre}(b), we may also assume that $\lambda$ is $\rho$-dominant by applying Lemma \ref{WAct}. Thus $\Gr(M)$ is a $\Gr(\hunlK) \cong \Syk{g} \otimes_{ \Sik{g}{G} } k$-module by Theorem \ref{GlobSec}(c) and if $\M := \Loc^\lambda(M)$ is the corresponding coherent $\hdnlK$-module then $\beta(\Ch(\M)) = \Ch(M)$ by Corollary \ref{LocFun}.

Let $\overline{k}$ be an algebraic closure of $k$ and let $X, Y$ denote the $\overline{k}$-points of $\Ch(\M)$ and $\Ch(M)$, respectively. These are algebraic varieties over $\overline{k}$ such that $\dim X = \dim \Ch(\M)$ and $\dim Y = \dim \Ch(M)$. Moreover, $Y \subseteq \N^\ast$ because $\Gr(M)$ is annihilated by $\Sik{g}{G}_+$, and $\beta : \Ex{T^\ast \Fl}(\overline{k}) \to \fr{g}^\ast(\overline{k})$ maps $X$ onto $Y$.

Let $f : X \to Y$ be the restriction of $\beta$ to $X$. Since $\dim Y = d(M) \geq 1$ by Corollary \ref{CdimAmod}, we can find a non-zero smooth point $y \in Y$. Since $f$ is surjective, we can find a smooth point $x \in f^{-1}(y)$. Considering the map $df_x : T_{X,x} \to T_{Y,y}$ induced by $f$ on Zariski tangent spaces shows that \[\dim Y + \dim f^{-1}(y) \geq \dim T_{X,x} .\]
Now $\dim T_{X,x}\geq \dim \Fl$ by Bernstein's Inequality, Theorem \ref{GlobalBernIneq}. Hence
\[\begin{array}{lllll} d(M) = \dim \Ch(M) = \dim Y &\geq& \dim \Fl - \dim \beta^{-1}(y)  \\
                                                 &\geq& \frac{1}{2} \dim G.y  \\ & \geq & r \end{array}\]
by Propositions \ref{NilpOrb} and \ref{MinNilp}.
\end{proof}

We believe that we can show that the bound in this theorem is best possible as it is in the classical result of Smith \cite[\S 3.10]{Smith}. However we will leave this for another paper.

\section{Microlocalisation of Iwasawa algebras}\label{MicIwa}
\subsection{Completed group rings}
We now specialize further and assume that the uniformizer $\pi$ of our complete discrete valuation ring $R$ is the prime number $p$. We make no assumptions about the residue field $k$ of $R$, except that it is of characteristic $p$. We let $v : R \to \mathbb{Z} \cup \{ \infty \}$ be the discrete valuation of $R$, normalized by $v(p) = 1$. Note that these assumptions imply that $R$ contains a canonical copy of the $p$-adic integers $\Zp$.

Whenever $G$ is a profinite group, we will denote its completed group ring with coefficients in $R$ by
\[RG := R[[G]] = \invlim R[G/N].\]
Here $N$ runs over all the open normal subgroups $N$ of $G$. When $G$ is a compact $p$-adic Lie group, we call $RG$ the \emph{Iwasawa algebra} of $G$ with coefficients in $R$.

\subsection{Uniform pro-$p$ groups and their Lie algebras}\label{UnifLie}
We refer the reader to \cite[\S 4]{DDMS} for the definition of uniform pro-$p$ groups, and briefly recall their main properties here. We fix the uniform pro-$p$ group $G$ of dimension $d := \dim G$ and a minimal topological generating set $\{g_1,\ldots,g_d\}$ for $G$ until the end of $\S\ref{MicIwa}$. Thus each element of $G$ can be written uniquely in the form $g_1^{\lambda_1} \cdots g_d^{\lambda_d}$ for some $\lambda_1,\ldots,\lambda_d \in \Zp$. It is shown in \cite[Theorem 4.30]{DDMS} that the operations
\[\begin{array}{rcl} \lambda \cdot x &=& x^\lambda,\\

 x+y &=&\lim\limits_{i\to\infty}(x^{p^i}y^{p^i})^{p^{-i}},\\

 [x,y]&=&\lim\limits_{i\to\infty}(x^{-p^i}y^{-p^i}x^{p^i} y^{p^i})^{p^{-2i}}\end{array}\]
define on the set $G$ the structure of a Lie algebra over $\Zp$. We will denote this Lie algebra by $L_G$. It is known that $L_G$ is a \emph{powerful Lie algebra}, in the sense that it is free of finite rank as a module over $\Zp$ and $[L_G, L_G]$ is contained in $p L_G$ (and $[L_G, L_G] \subseteq 4L_G$ if $p = 2$). Regardless of whether $p$ is odd or even, we note that $[p^{-1} L_G, p^{-1}L_G]$ is always contained in $p^{-1}L_G$. This motivates the following
\begin{defn} Let $G$ be a uniform pro-$p$ group. We define \emph{$R$-Lie algebra} associated with $G$ to be
\[\frac{1}{p}RL_G := R \otimes_{\Zp} \left(\frac{1}{p}L_G\right).\]
\end{defn}
It is clear that $\frac{1}{p}RL_G$ is an $R$-Lie algebra which is free over $R$ of rank $d$.  We can also consider the \emph{associated graded group} to $G$:
\[\gr G := \bigoplus_{i=0}^\infty \gr_i G = \bigoplus_{i=0}^\infty G^{p^i} / G^{p^{i+1}}.\]
Since $G$ is uniform, each graded piece $\gr_i G$ is a finite elementary abelian $p$-group. Moreover by \cite[Lemma 4.10]{DDMS}, the $p$-power map induces a bijection $\gr_i G \to \gr_{i+1}G$ of abelian $p$-groups, so in fact $\gr G$ is naturally a graded module over the polynomial ring $\Fp[t]$ where $t$ acts by raising elements to their $p$-th powers:
\[t \cdot g G^{p^{i+1}} = g^p G^{p^{i+2}}\]
for all $g \in G^{p^i}$. This module is free of rank $d$. Since $L_{G^{p^i}} = p^iL_G$ for all $i\geq 0$, we can use \cite[Corollary 4.15]{DDMS} to identify $\gr G$ with
\[\gr L_G := \bigoplus_{i=0}^\infty  p^i L_G / p^{i+1} L_G\]
and therefore $\gr G$ carries the structure of a graded $\Fp[t]$-Lie algebra.  By tensoring $\gr G$ with $k$ over $\Fp$, we obtain the following

\begin{lem} Let $G$ be a uniform pro-$p$ group, let $\fr{h}$ be its associated $R$-Lie algebra and let $\fr{h}_k = \fr{h} / p \fr{h}$. Then there is a natural isomorphism of graded $k[t]$-Lie algebras
\[ \gr G \otimes_{\Fp} k \cong t  \fr{h}_k[t].\]
In particular, $\gr G$ is abelian whenever $\fr{h}_k$ is abelian.\end{lem}

\subsection{The $\fr{m}$-adic filtration on $RG$}\label{madicFilt}
Let $b_i = g_i - 1 \in R[G]$, and write
\[ \mb{b}^\alpha = b_1^{\alpha_1} \cdots b_d^{\alpha_d} \in R[G]\]
for any $d$-tuple $\alpha \in \mathbb{N}^d$. Then it follows from the proof of \cite[Theorem 7.20]{DDMS} that $RG$ can be naturally identified with the set of non-commutative formal power series in $b_1,\ldots, b_d$ with coefficients in $R$:
\[RG = \left\{ \sum_{\alpha\in\mathbb{N}^d} \lambda_\alpha \mb{b}^\alpha \st \lambda_\alpha \in R\right\}.\]
Let $\fr{m} = \ker(RG \to k)$ be the unique maximal ideal of $RG$, and let $\deg : RG \to \mathbb{Z} \cup \{\infty\}$ be the degree function corresponding to the $\fr{m}$-adic filtration on $RG$; thus $\deg(x) = a$ precisely when $x \in \fr{m}^a \backslash \fr{m}^{a+1}$. We can now state the fundamental result due to Lazard.

\begin{thm} The group ring $R[G]$ is dense in $RG$, and
\[\deg \left( \sum_{\alpha\in\mathbb{N}^d} \lambda_\alpha \mb{b}^\alpha \right) = \min \left\{ v(\lambda_\alpha) + |\alpha| \st \alpha \in\mathbb{N}^d\right\}.\]
The degree filtration on $RG$ is complete and the associated graded ring $\gr RG$ is isomorphic to the enveloping algebra of the $k[t]$-Lie algebra $\gr G \otimes_{\Fp} k$:
\[\gr RG \cong U(\gr G \otimes_{\Fp} k).\]\end{thm}
\begin{proof} The group ring $R[G]$ is dense in $RG$ because it contains all sums of the form $\sum_{\alpha \in \mathbb{N}^d} \lambda_\alpha \mb{b}^\alpha$ where only finitely many coefficients are non-zero. The displayed formula for the degree of an element in $RG$ follows from \cite[Theorem 7.5]{DDMS}, and the completeness of the degree filtration on $RG$ follows from the fact that $R$ is $\pi$-adically complete.

Define a function $\omega : G \to \mathbb{Z}^+ \cup \{\infty\}$ by $\omega(g) = i + 1$ if $g \in G^{p^i} \backslash G^{p^{i+1}}$ and $\omega(1) = \infty$. Then $\omega$ is a $p$-valuation on $G$ in the sense of \cite[Definition II.2.1.2]{Laz1965}, and the associated graded Lie algebra $\gr G$ of $G$ with respect to $\omega$ in the sense of \cite[\S II.1.1.72]{Laz1965} coincides with the Lie algebra $\gr G$ defined above, because $[x,y] \equiv x^{-1}y^{-1}xy \mod G^{p^2}$ for any $x,y \in G \backslash G^p$ by \cite[Lemma 4.28]{DDMS}. The last assertion of the Theorem now follows from \cite[Theorem III.2.3.3]{Laz1965}.\end{proof}

\begin{cor} $\gr RG$ is a Noetherian domain.\end{cor}
\begin{proof} By Lemma \ref{UnifLie}, $\gr G\otimes_{\Fp} k$ is isomorphic to $t \fr{h}_k[t]$. This is a free $k[t]$-module of rank $d$. Now apply the Poincar\'e-Birkhoff-Witt Theorem. \end{proof}

\subsection{The microlocalisation of $RG$ at $\gr p$}\label{MicIwaGrp}
We can now make the connection between Iwasawa algebras and almost commutative affinoid algebras. Since $\fr{h} := \frac{1}{p}RL_G$ is a $R$-Lie algebra which is free of finite rank as an $R$-module, we know from Example \ref{ExAC}(c) that $\U{h}$ is an almost commutative $R$-algebra. Hence we may form its $p$-adic completion in the manner of $\S\ref{GrComp}$:
\[\h{\U{h}_K} = \left(\invlim \U{h} / p^r \U{h}\right) \otimes_R K.\]
On the other hand, the set of powers of $\gr p$ in $\gr RG$ is multiplicatively closed and consists of homogeneous central elements, and $\gr RG$ is Noetherian by Corollary \ref{madicFilt}, so we can consider the corresponding microlocal Ore set $S$ from $\S\ref{Micro}$:
\[S := \{x \in RG \st \gr x = (\gr p)^a \quad\mbox{for some}\quad a \geq 0\} = \bigcup_{a\geq 0} \left(p^a + \fr{m}^{a+1}\right) \subseteq RG.\]

\begin{thm} Let $G$ be a uniform pro-$p$ group and let $\fr{h}=\frac{1}{p}RL_G$ be its associated $R$-Lie algebra. Then the microlocalisation of $RG$ at $\gr p$ is isomorphic as a complete $\mathbb{Z}$-filtered ring to the almost commutative affinoid $K$-algebra $\h{\U{h}_K}$:
\[Q_{\gr p}(RG) \cong \h{\U{h}_K}.\]
\end{thm}
\begin{proof}The exponential series $\exp(u)$ converges to a unit in the algebra $A := \h{\U{h}_K}$ whenever $u \in L_G \subseteq p\fr{h}$. Now the Campbell-Hausdorff formula \cite[Theorem 6.28]{DDMS} shows that
\[u \ast v := \log (\exp(u) \exp(v))\]
is an element of $L_G$ for all $u, v \in L_G$, and $\psi : u \mapsto \exp(u)$ is an isomorphism from the uniform pro-$p$ group $G$ to the subgroup $\exp(L_G)$ of the group of units of $A$ by the proof of \cite[Theorem 9.10]{DDMS}. We thus obtain an $R$-algebra homomorphism
\[ \psi : R[G] \to A\]
such that $\psi(u) = \exp(u)$ for all $u \in G$. Let $\{u_1,\ldots, u_d\}$ be the $R$-basis for $\fr{h}$ corresponding to the topological generating set $\{g_1,\ldots,g_d\}$ of $G$. In these coordinates, the map $\psi$ satisfies
\[\psi(\mb{b}^\alpha) = (e^{pu_1} - 1)^{\alpha_1} \cdots (e^{pu_d} - 1)^{\alpha_d} \equiv p^{|\alpha|} \mb{u}^\alpha \mod p^{|\alpha| + 1} F_0A\]
where the monomial $\mb{u}^\alpha = u_1^{\alpha_1} \cdots u_d^{\alpha_d}$ is computed inside the enveloping algebra $\U{h} \subseteq A$.  Now let $x = \sum_{\alpha\in\mathbb{N}^d} \lambda_\alpha \mb{b}^\alpha \in R[G]$ be a non-zero element and let $m = \deg(x)$. Then $m = \min \left\{ v(\lambda_\alpha) + |\alpha| \st \alpha \in\mathbb{N}^d\right\}$ by Theorem \ref{madicFilt} and
\[\psi(x) \equiv  p^m \sum_{v(\lambda_\alpha) = m - |\alpha|} \left(\frac{\lambda_\alpha}{p^{v_p(\lambda_\alpha)}}\right) \mb{u}^\alpha \mod p^{m+1}F_0A.\]
Now the slice $\gr_0 A$ of $A$ is isomorphic to $\Uk{h}$ and the images of the monomials $\mb{u}^\alpha$ in $\gr_0 A$ are $k$-linearly independent by the Poincar\'e-Birkhoff-Witt Theorem. Hence
\[\deg(\psi(x)) = \deg(x)\quad\mbox{for all}\quad x \in R[G]\]
and the homomorphism $\psi$ is strictly filtered. Since $A$ is complete and since $R[G]$ is dense in $RG$ by Theorem \ref{madicFilt}, $\psi$ extends to a strictly filtered ring homomorphism
\[\psi : RG \to A\]
which sends $\gr p \in \gr RG$ to $s = \gr p \in \gr A$. Now $\gr A \cong U(\fr{h}_k)[s, s^{-1}]$ by Lemma \ref{DoubleFilt}, so $s$ is a homogeneous unit in $\gr A$ and hence $\psi$ extends to a strictly filtered ring homomorphism
\[\psi_K : Q_t(RG) \to A\]
by the universal property of algebraic microlocalisation. Since $\psi_K$ is strictly filtered, it must be injective. Now
\[\gr Q_t(RG) \cong (\gr RG)_t\]
by Lemma \ref{Micro}, and the above computation shows that
\[\begin{array}{lll} \gr(\psi_K)(t) &=& s, \quad\mbox{and}\\

\gr(\psi_K)(\gr(b_i)) &=& su_i \quad\mbox{for all}\quad i=1,\ldots, d.\end{array}\]
Hence the image of $\gr(\psi_K)$ contains the generators of $\gr A$ as an algebra over $\gr K \cong k[t,t^{-1}]$, so $\gr(\psi_K)$ is surjective. Hence $\psi_K$ is also surjective because $Q_t(RG)$ and $A$ are complete.
\end{proof}

\subsection{Remarks}\label{RemsLaz}
\be
\item Theorem \ref{MicIwaGrp} is essentially due to Lazard, since it appears in a different language as \cite[\S IV.3.2.5]{Laz1965} and is proved there for the larger class of $p$-saturated groups. The completed enveloping algebra $\h{\U{h}}$ is shown to be isomorphic to the `saturation' $\Sat(R[G])$ of the valued group ring $R[G]$.
\item The algebra $A$ is heavily used in the foundational chapter 6 of the book \cite{DDMS} under the name $\Qp[[G]]$. It underpins the entire development of Lie theory for compact $p$-adic analytic groups in that book.
\item We believe that our way of phrasing Theorem \ref{MicIwaGrp} is new in the literature. The algebra $A$ appears as the `largest' distribution algebra $D_{1/p}(G,K)$ in the paper \cite{ST} by Schneider and Teitelbaum.
\item One advantage of the viewpoint we give in this paper is that the theory of algebraic microlocalisation tells us which modules are killed by the base change functor associated to the ring homomorphism $RG \to A$: these are precisely the $S$-torsion modules.
\ee

\subsection{Crossed products}\label{CrossProd} The subgroup $G^{p^n}$ of $G$ is uniform by \cite[Theorems 3.6(i), 4.5]{DDMS} and the group $H_n := G / G^{p^n}$ is finite by \cite[Proposition 1.16(iii)]{DDMS}. Consider the completed group ring $RG^{p^n}$ and let $\fr{m}_n$ be its unique maximal ideal. The group $G$ acts on $RG^{p^n}$ by conjugation, this action preserves the $\fr{m}_n$-adic filtration and fixes $p$. Hence it preserves the corresponding microlocal set
\[S_n := \bigcup_{a \geq 0} (p^a + \fr{m}_n^{a+1})\]
in $RG^{p^n}$ and induces an action of $G$ by ring automorphisms on the microlocalisation
\[\mathcal{U}_n := Q_{\gr p} (RG^{p^n}).\]
We will write this action on the left: thus $x \mapsto { }^gx$ is the automorphism induced by $g \in G$ on $\mathcal{U}_n$. We now define a multiplication on the tensor product $\mathcal{U}_n \otimes_{RG^{p^n}}RG$ by setting
\[ (x\otimes g) \cdot (y \otimes h) = x (^gy) \otimes gh.\]
for $x,y \in \mathcal{U}_n$ and $g,h \in G$.
\begin{prop} Let $\fr{h}=\frac{1}{p}RL_G$ be the associated $R$-Lie algebra of $G$.
\be
\item $\mathcal{U}_n \cong \h{\U{h}_{n,K}}$ is an almost commutative affinoid $K$-algebra.
\item $\mathcal{U}_n \otimes_{RG^{p^n}} RG$ is a crossed product $\GhuntK$ of $\h{\U{h}_{n,K}}$ with $H_n = G/G^{p^n}$.
\item The inclusion $RG^{p^n} \hookrightarrow \mathcal{U}_n$ extends to a natural inclusion $RG \hookrightarrow \GhuntK$.
\item $\GhuntK$ is a flat right $RG$-module.
\item If $M$ is a finitely generated $RG$-module, then $(\GhuntK) \otimes_{RG} M = 0$ if and only if $M$ is $S_n$-torsion.
\ee \end{prop}
\begin{proof}
(a) The $R$-Lie algebra associated with the uniform pro-$p$ group $G^{p^n}$ is clearly $p^n \fr{h}$, so $\mathcal{U}_n$ is isomorphic to $\h{U(p^n\fr{h})_K}$ by Theorem \ref{MicIwaGrp}. But this is just $\h{\U{h}_{n,K}}$ by definition.

(b), (c) Recall \cite[\S 1.5.8]{MCR} that a \emph{crossed product} of a ring $S$ by a group $H$ is an associative ring $S\ast H$ which contains $S$ as a subring and contains a set of units $\overline{H} = \{\overline{h} : h \in H\}$, isomorphic as a set to $H$, such that
\begin{itemize}
\item $S\ast H$ is a free right $S$-module with basis $\overline{H}$,
\item for all $x,y \in H$, $\overline{x}S = S\overline{x}$ and $\overline{x}\cdot\overline{y}S = \overline{xy}S$.
\end{itemize}
Such a crossed product determines an \emph{action} $\sigma : H \to \Aut(S)$ and a \emph{twisting} $\tau$ by the rules
\[\begin{array}{rcl} \sigma(x)(s) &=& \overline{x}\hspace{1mm}^{-1} s \hspace{1mm} \overline{x} \\
\overline{x}\hspace{1mm}\overline{y} &=& \overline{xy} \hspace{1mm} \tau(x,y)\end{array}\]
for all $x, y \in H$ and $s \in S$. Here $\tau(x,y) \in S^\times$ for all $x,y \in H$. It turns out that $\sigma$ defines a group homomorphism $H \to \Out(S)$ and $\tau$ is a $2$-cocycle $\tau : H \times H \to S^\times$ for the action of $H$ on $S^\times$ via $\sigma$. Conversely, starting with a ring $S$, a group $H$, a group homomorphism $\sigma : H \to \Out(S)$ and a $2$-cocycle $\tau : H \times H \to S^\times$, one can construct an associative ring $S \ast_{\sigma,\tau} H$ which is a crossed product of $S$ by $H$, having the prescribed action and twisting --- see \cite{Pass}.

Now $RG$ is a crossed product of $RG^{p^n}$ with $H_n$ defined by some action $\sigma : H_n \to \Out(RG^{p^n})$ and twisting $\tau : H_n \times H_n \to (RG^{p^n})^\times$. Units in $RG^{p^n}$ are units in $\mathcal{U}_n$, and ring automorphisms of $RG^{p^n}$ extend to ring automorphisms of $\mathcal{U}_n$ because the $\fr{m}_n$-adic filtration on $RG^{p^n}$ is canonical. Furthermore, inner automorphisms extend to inner automorphisms, so we obtain an action $\sigma' : H_n \to \Out(\mathcal{U}_n)$ and a twisting $\tau' : H_n \times H_n \to \mathcal{U}_n^\times$ which is still a $2$-cocycle. Thus we can form the crossed product $\mathcal{U}_n \ast_{\sigma', \tau'} H_n$ which equals $\mathcal{U}_n \otimes_{RG^{p^n}} RG$ as a set, and a ring homomorphism $RG \to \mathcal{U}_n \ast_{\sigma', \tau'} H_n$ which extends the inclusion of $RG^{p^n}$ into $\mathcal{U}_n$. It is clear that the multiplication in this crossed product agrees with the one defined above.

(d) For any $RG$-module $M$, there is an isomorphism of left $\mathcal{U}_n$-modules:
\[(\GhuntK) \otimes_{RG} M =         \mathcal{U}_n \otimes_{RG^{p^n}} RG \otimes_{RG} M \cong \mathcal{U}_n \otimes_{RG^{p^n}} M.\]
Restriction of modules is exact, and $\mathcal{U}_n = Q_{\gr p}(RG^{p^n})$ is a flat $RG^{p^n}$-module by Lemma \ref{Micro}

(e) This follows from Lemma \ref{Micro} and the displayed isomorphism above.
\end{proof}

\begin{rmk} It can probably be shown that the crossed product $\GhuntK$ is isomorphic to the distribution algebra $D_{\sqrt[p^n]{1/p}}(G,K)$, but we will not need this isomorphism.\end{rmk}

\subsection{Re-valuation of $p \in RG$}
\label{FiltZpG}
We now reinterpret the work of Schneider and Teitelbaum \cite[\S 4]{ST}. The main idea is to define degree functions
\[\deg_w : RG \to \mathbb{R} \cup \{\infty\}\]
for any real number $w \geq 1$, such that
\[\deg_w(p) = w \quad\mbox{and}\quad \deg_w(b_i) = 1 \quad\mbox{for all}\quad i = 1,\ldots, d.\]

\begin{defn} For any real number $w \geq 1$, define $\deg_w : RG \to \mathbb{R} \cup \{\infty\}$ by
\[\deg_w \left( \sum_{\alpha\in\mathbb{N}^d} \lambda_\alpha \mb{b}^\alpha \right) = \min \left\{ w \cdot v(\lambda_\alpha) + |\alpha| \st \alpha \in\mathbb{N}^d\right\}\]
with the understanding that this minimum value is $\infty$ if all the $\lambda_\alpha$ are zero.
\end{defn}
Thus $\deg_1$ is the degree function associated to the $\fr{m}$-adic filtration on $RG$.
\begin{lem}For any $w \geq 1$, $\deg_w$ is a degree function in the sense of $\S \ref{DegFun}$.
\end{lem}
\begin{proof} When translated to the language of norms, $\deg_w$ corresponds to the norm $||\cdot||_{p^{-1/w}}$ defined in \cite[p. 160]{ST}. Then the result follows from \cite[Proposition 4.2]{ST}.\end{proof}

Morally, as $w \to \infty$ the element $p \in RG$ approaches $0$, so the filtrations approach $\mod p$ Iwasawa algebra $kG := RG / p RG$, equipped with its $\overline{\fr{m}}$-adic filtration.

Let $\gr^w RG$ be the associated graded ring of $RG$ with respect to the associated $\mathbb{R}$-filtration, and let $X_i = \gr^w b_i\in \gr^w_1 RG$ be the principal symbols of the topological generators $b_i$ of $RG$. The associated graded ring $\gr^w R$ of $R$ with respect to $\deg_w$ is isomorphic to the polynomial ring $k[t_w]$ with $t_w$ in degree $-w$ --- see $\S\ref{DegFun}$ for our conventions.

\begin{prop}
\begin{enumerate}[{(}a{)}] \hfill
\item $\gr^w RG$ is isomorphic to the polynomial ring $k[t_w, X_1,\ldots, X_d]$ as $k[t_w]$-modules.
\item If $w > 1$ then this is an isomorphism of graded rings.
\end{enumerate}
\end{prop}
\begin{proof} Apply \cite[Lemma 4.3]{ST} and the remarks immediately before this Lemma.\end{proof}

\subsection{The restriction of $\deg_{p^n}$ to $RG^{p^n}$}
\label{Open}
Recall that $\{g_1,\ldots,g_d\}$ is a minimal topological generating set for the uniform pro-$p$ group $G$. By \cite[Theorem 3.6(iii)]{DDMS}, $\{g_1^{p^n}, \ldots, g_d^{p^n}\}$ is a minimal topological generating set for the open uniform subgroup $G^{p^n}$ of $G$, so
\[RG^{p^n} = \left\{ \sum_{\alpha\in\mathbb{N}^d} \lambda_\alpha \mb{b}_n^\alpha \st \lambda_\alpha \in R\right\}\]
where $\mb{b}_n^\alpha := (g_1^{p^n} - 1)^{\alpha_1} \cdots (g_d^{p^n} - 1)^{\alpha_d}$. We can now calculate the restriction of the filtration $\deg_{p^n}$ on $RG$ to its subalgebra $RG^{p^n}$. The next result is essentially \cite[Proposition 6.2]{Schmidt}, but we give a proof for the convenience of the reader.

\begin{prop} Let $\fr{m}_n$ be the maximal ideal of $RG^{p^n}$.  If $x \in \fr{m}_n^a \backslash \fr{m}_n^{a+1}$ for some integer $a\geq 0$, then $\deg_{p^n}(x) = p^na$.
\end{prop}
\begin{proof} The polynomial $(1 + X)^{p^n} - (1 + X^{p^n})$ is divisible by $p$ and has no constant term. Therefore
\[ (1 + b_i)^{p^n} \equiv 1 + b_i^{p^n} \mod p\fr{m}_0\]
for all $i = 1,\ldots, d$. Since $\deg_{p^n}(\fr{m}_0) \geq 1$ and $\deg_{p^n}(p) = p^n$ by definition, we see that
\[ g_i^{p^n} - 1 = b_i^{p^n} + \epsilon_i\]
for some $\epsilon_i \in RG$ with $\deg_{p^n}(\epsilon_i) > p^n$. It follows that
\[ \mb{b}_n^\alpha = (g_1^{p^n} - 1)^{\alpha_1} \cdots (g_d^{p^n} - 1)^{\alpha_d} = b_1^{p^n \alpha_1} \cdots b_d^{p^n \alpha_d} + \epsilon_\alpha = \mb{b}^{p^n\alpha} + \epsilon_\alpha\]
for some $\epsilon_\alpha \in RG$ with $\deg_{p^n}(\epsilon_\alpha) > p^n |\alpha|$. Thus
\[\deg_{p^n}(\fr{m}_n^j) \geq p^nj\quad\mbox{for all}\quad j \geq 0.\]
Now as $x \in \fr{m}_n^a \backslash \fr{m}_n^{a+1}$, we can write
\[ x \equiv \sum_{\alpha \in T} \lambda_\alpha \mb{b}_n^\alpha \mod \fr{m}_n^{a+1}\]
for some non-empty set $T$ of indices $\alpha$ satisfying $v(\lambda_\alpha) = a - |\alpha|$ for all $\alpha \in T$. Because $\deg_{p^n}(\fr{m}_n^{a+1}) \geq p^n(a+1) > p^na$,
\[x = \sum_{\alpha \in T} \lambda_\alpha \mb{b}^{p^n\alpha} + x'\]
for some $x'$ with $\deg_{p^n}(x') > p^n a$. Since $v(\lambda_\alpha) = a - |\alpha|$ for all $\alpha \in T$, we see that
\[\deg_{p^n}\left(\sum_{\alpha \in T} \lambda_\alpha \mb{b}^{p^n\alpha}\right) = \min\{ p^n v(\lambda_\alpha) + p^n |\alpha| \st \alpha \in T\} = p^n a \]
by the definition of $\deg_{p^n}$. Hence $\deg_{p^n}(x) = p^na$ as required.
\end{proof}

Let $S(w)$ be the microlocal Ore set in $RG$ associated to the $\deg_w$ filtration and the powers of $t_w$ in $\gr^w RG$:
\[ S(w) = \{ x \in RG \st \gr^w(x) = t_w^a \quad\mbox{for some} \quad a\geq 0\},\]
and recall the microlocal Ore set $S_n \subseteq RG^{p^n}$ from $\S\ref{CrossProd}$.
\begin{cor} \be \hfill
\item The restriction of the $\deg_{p^n}$ filtration to $RG^{p^n}$ is a re-scaling of the $\fr{m}_n$-adic filtration on $RG^{p^n}$.
\item The image of $\gr^{p^n} RG^{p^n}$ inside $\gr^{p^n} RG$ is precisely $k[t_{p^n}][X_1^{p^n}, \ldots, X_d^{p^n}]$, with all generators in degree $p^n$.
\item For any $n \geq 0$, $S_n$ is contained in $S(p^n)$.
\ee\end{cor}
\begin{proof} The first two statements are clear. Let $p^a + x \in S_n$ for some $a \geq 0$ and $x \in \fr{m}_n^{a+1}$. Then
\[\deg_{p^n}(x) \geq p^n(a+1) > p^na\]
by Proposition \ref{Open} so $p^a + x \in S(p^n)$.
\end{proof}

\subsection{The filtration on $kG$}
\label{FiltFpG}
Write $\overline{x}$ for the image of $x \in RG$ in the completed group ring $kG := RG / p RG$. We will abuse notation and write $\overline{\mb{b}^\alpha} = \mb{b}^\alpha$, so that $kG$ is in bijection with the set of non-commutative formal power series in $b_1,\ldots, b_d$ with coefficients in $k$:
\[kG = \left\{ \sum_{\alpha\in\mathbb{N}^d} \lambda_\alpha \mb{b}^\alpha \st \lambda_\alpha \in k\right\}.\]
Let us define $\od : kG \to \mathbb{R} \cup \{\infty\}$ as follows:
\[\od \left( \sum_{\alpha\in\mathbb{N}^d} \mu_\alpha \mb{b}^\alpha \right) = \min \left\{ |\alpha| \st \mu_\alpha \neq 0\right\}.\]
Let $\fr{m} = \ker(kG \to k)$ be the maximal ideal of $kG$; then clearly
\[\od (x) = \left\{ \begin{array}{l}a \quad \mbox{ if } x \in \fr{m}^a \backslash \fr{m}^{a+1} \\ \infty \quad \mbox{ if } x = 0\end{array} \right. \]
so $\od$ is the usual degree function associated with the $\fr{m}$-adic filtration on $kG$.

The degree functions $\deg_w$ on $RG$ and $\od$ on $kG$ are related as follows.

\begin{lem} Let $x \in RG$ be such that $\overline{x} \neq 0$. Then
\be
\item $\od(\overline{x}) \geq \deg_w(x)$ for any $w \geq 1$.
\item If $y\in RG$ is such that $\overline{y} = \overline{x}$, and if $w \geq \od(\overline{x})$, then
\[\deg_w(y) = \od(\overline{x}).\]
\ee
\end{lem}
\begin{proof}(a) Write $x = \sum_{\alpha\in\mathbb{N}^d} \lambda_\alpha \mb{b}^\alpha$ for some $\lambda_\alpha \in R$. Since $\overline{x}$ is non-zero, $\od(\overline{x}) = |\beta|$ for some $\beta \in \mathbb{N}^d$ such that $v(\lambda_\beta) = 0$. Then
\[ \od(\overline{x}) = |\beta| \geq \min \left\{ w \cdot v(\lambda_\alpha) + |\alpha| \st \alpha \in\mathbb{N}^d\right\} = \deg_w(x).\]

(b) Let $\alpha$ be such that $|\alpha| < |\beta|$. By definition of $\od$, $p$ divides $\lambda_\alpha$, and the coefficient $\mu_\alpha$ of $\mb{b}^\alpha$ in $y$ differs from $\lambda_\alpha$ by a multiple of $p$. Since $w \geq \od(\overline{x}) = |\beta|$ by assumption,
\[w \cdot v(\mu_\alpha) + |\alpha| \geq w \geq |\beta|\]
for any such $\alpha$. On the other hand
\[w \cdot v(\mu_\alpha) + |\alpha| \geq |\beta|\]
is trivially true whenever $|\alpha| \geq |\beta|$, so
\[\deg_w(y) = \min \{w\cdot v(\mu_\alpha) + |\alpha|\} \geq |\beta| = \od(\overline{x}).\]
But we showed in (a) that $\od(\overline{x}) = \od(\overline{y}) \geq \deg_w(y)$.
\end{proof}

\subsection{Good generating sets}
\label{GGS}
Since the $\fr{m}$-adic filtration on $kG$ is complete, and the associated graded ring is Noetherian, it is well known that the Rees ring
\[ \Ex{kG} = \bigoplus_{j \in \mathbb{Z}} \fr{m}^j t^{-j} \subset kG[t, t^{-1}]\]
is Noetherian (where as always $\fr{m}^j = kG$ if $j \le 0$.)

Let $J$ be a left ideal of $kG$. Then the Rees ideal
\[ \Ex{J} = \bigoplus_{j \in \mathbb{Z}} (J \cap \fr{m}^j) t^{-j}  \triangleleft_l \Ex{ kG}\]
is finitely generated over $\Ex{kG}$. Let $z_1 t^{-d_1}, \cdots, z_\ell t^{-d_\ell}$ be a homogeneous generating set for $\Ex{J}$, for some $z_i \in J$ with $\od(z_i) = d_i$; then $\{z_1,\ldots, z_\ell\}$ is a \emph{good generating set} for $J$:
\[ J \cap \fr{m}^n = \sum_{i=1}^\ell \fr{m}^{n - d_i}z_i  \quad \mbox{for all}\quad n \in \mathbb{Z}.\]
We record this as a Lemma:
\begin{lem} Let $x \in J$. Then there exist $r_i \in kG$ such that $x = \sum_{i=1}^\ell r_iz_i$ and
\[ \od(r_i) \ge \od(x) - d_i\]
for all $i = 1,\ldots, \ell$.
\end{lem}

Here is the faithful flatness result of Schneider and Teitelbaum from the point of view of non-commutative algebra.
\subsection{Theorem}\label{MainTech} Let $J$ be a left ideal of $RG$ such that $RG / J$ is $p$-torsion-free. Then there exists $w_0 > 1$, depending only on $J$, such that $\gr^w(RG / J)$ is $t_w$-torsion-free for all $w \geq w_0$.
\begin{proof}
Suppose that $t_w \cdot X \in \gr^w J$ for some homogeneous element $X \in \gr^w RG$; we will show if $w$ is large enough then we can find $x' \in J$ such that $X = \gr^w x'$, which implies that $X \in \gr^w J$.

Let $\{z_1,\ldots, z_\ell\}$ be a good generating set for the left ideal $\overline{J}$ of $k G$, as in $\S \ref{GGS}$. Let $d_i = \od(z_i)$ and define
\[w_0 := \max\{ d_1, \ldots, d_\ell\}.\]
We fix lifts $a_i \in J$ for these generators and note that if $w \geq w_0$ then
\[\deg_w(a_i) = d_i \quad\mbox{for all}\quad i=1,\ldots, \ell\]
by Lemma \ref{FiltFpG}(b). Choose some $x \in R G$, possibly not in $J$, such that $\gr^w x = X$. Since $\gr^w p \cdot \gr^w x \in \gr^w J$,
\[px + z \in J\]
for some $z \in RG$ with $\deg_w(z) > \deg_w(px)$. So $\overline{z} \in \overline{J}$ and by Lemma \ref{GGS} we can find $r_1,\ldots, r_\ell \in k G$ such that $\overline{x} = \sum_{i=1}^\ell r_i z_i$ and
\[\od(r_i) \geq \od(\overline{x}) - d_i \quad\mbox{for all} \quad i = 1,\ldots, \ell.\]
Choose lifts $s_i \in RG$ of $r_i \in k G$ satisfying $\deg_w(s_i) = \od(r_i)$. Then
\[z = \sum_{i=1}^\ell s_i a_i + py\]
for some $y \in RG$, and
\[\begin{array}{lllllll} \deg_w\left(\sum_{i=1}^\ell s_i a_i \right) &\geq& \min \{ \deg_w(s_i) + \deg_w(a_i) \} &=& && \\
&= & \min \{ \od(r_i) + d_i \} &\geq& \od(\overline{z}) &\geq& \deg_w(z)\end{array}\]
by Lemma \ref{FiltFpG}(a). Hence
\[\deg_w(py) = \deg_w\left(z - \sum_{i=1}^\ell s_i a_i\right) \geq \deg_w(z) > \deg_w(px).\]
Since $t_w$ is not a zero-divisor in $\gr^w RG$, we deduce that
\[\deg_w(y) > \deg_w(x).\]
Since each $a_i$ lies in $J$,
\[px + py = px + z - \sum_{i=1}^\ell s_i a_i \in J.\]
But $RG / J$ is $p$-torsion-free, so $x' := x + y \in J$; since $\deg_w(y) > \deg_w(x)$ we deduce that $X = \gr^w x = \gr^w x' \in \gr^w J$.
\end{proof}

\begin{cor} Let $M$ be a finitely generated $p$-torsion-free $RG$-module. Then there exists $n_0 \in \mathbb{N}$ such that $M$ is $S_n$-torsion-free for all $n \geq n_0$.
\end{cor}
\begin{proof} Let $KG = K \otimes_R RG$. It is enough to prove that the $KG$-module $M_K=K\otimes_R M$ is $S_n$ torsion-free for all $n$ sufficiently large. We can find a finite composition series for $M_K$ consisting of $KG$-submodules such that each composition factor is cyclic; then $M_K$ is $S_n$-torsion-free provided all the composition factors are $S_n$-torsion-free. Thus we may reduce to the case where $M$ is a cyclic $RG$-module, $M\cong RG/J$, say.

 By Theorem \ref{MainTech}, there exists $w_0 > 1$ such that $\gr^w (RG / J)$ is $t_w$-torsion-free whenever $w \geq w_0$. Hence $RG/J$ is $S(w)$-torsion-free if $w \geq w_0$ by Lemma \ref{Micro}.

Choose $n_0$ such that $p^{n_0} > w_0$, and let $n \geq n_0$.  Then $p^n \geq w_0$ and $S_n \subseteq S(p^n)$ by Corollary \ref{Open}(c), so $M$ is $S_n$-torsion-free.
\end{proof}

\subsection{The congruence kernels}\label{CongKern}
Recall the algebraic $R$-group scheme $\mb{G}$ introduced in $\S \ref{AlgGps}$. We now make the following additional assumptions on our uniform pro-$p$ group $G$:
\begin{itemize}
\item The Lie algebra $\fr{g}$ of the algebraic group $\mb{G}$ and the Lie algebra $L_G$ of $G$ satisfy $p^{m+1} \fr{g} = R \otimes_{\Zp}L_G$ for some integer $m \geq 0$,
\item $\mb{G}$ is semisimple,
\item $p$ is a very good prime for $\mb{G}$ in the sense of $\S \ref{VeryGoodp}$.
\end{itemize}
Thus $p^m\fr{g}$ is the $R$-Lie algebra $\frac{1}{p}RL_G$ associated to the uniform pro-$p$ group $G$ in the sense of $\S\ref{UnifLie}$. It follows from the discussion in \cite[\S 7]{SchVen} that the so-called \emph{$(m+1)$-st congruence kernel}
\[G = \ker(\mb{G}(\Zp) \to \mb{G}(\Zp/p^{m+1}\Zp))\]
of the group of $\Zp$-points $\mb{G}(\Zp)$ satisfies these conditions. We now combine the earlier results of this paper in order to obtain some information about the representation theory of the central localisation
\[ KG := K \otimes_R RG = K \otimes_R R[[G]].\]
of the completed group ring $R[[G]]$. Equivalently, we study the category of $p$-torsion-free $RG$-modules. Perhaps our methods are applicable to a slightly wider class of compact $p$-adic analytic groups, but it may help the reader to keep these congruence kernels in mind.

\subsection{Modules over $KG$}
Note that $RG$ is an $R$-lattice in $KG$ and the slice $\gr_0 KG \cong kG$ is a complete $\mathbb{Z}$-filtered ring when equipped with its $\fr{m}$-adic filtration --- see $\S \ref{FiltFpG}$. Thus $KG$ is itself a complete doubly filtered $K$-algebra with $\Gr(KG) = \gr kG \cong k[x_1,\ldots, x_d]$ and the theory of $\S \ref{DoubleFilt} - \S \ref{CharVarAff}$ applies, although it is not an almost commutative affinoid $K$-algebra because the filtration on its slice is negative.

\begin{lem} $KG$ is Auslander-Gorenstein with $\injdim KG = \dim G$. Let $M$ be a finitely generated $KG$-module. Then $d(M) = \dim \Ch(M)$ and $d(M) = 0$ if and only if $M$ is finite dimensional over $K$.\end{lem}
\begin{proof} The first part is originally due to Otmar Venjakob \cite[Theorem 3.29]{V}, but also follows from Theorem \ref{CharVarAff}. Since $KG$ always has the trivial module $K$ which is one-dimensional over $K$, the second part follows from Proposition \ref{CdimAmod}.
\end{proof}

We can now put all the pieces together and prove the main result of our paper. Recall the integer $r$ defined in $\S \ref{MinNilp}$, and the algebras $\mathcal{U}_n = Q_{\gr p} (RG^{p^n})$ defined in $\S \ref{CrossProd}$, and note that in our current notation,
\[ \mathcal{U}_n \cong \widehat{U(\fr{g})_{n+m, K}}\]
by Theorem \ref{MicIwaGrp}.
\begin{thm} Let $M$ be a finitely generated $KG$-module which is infinite dimensional over $K$. Then $d(M) \geq r$.
\end{thm}
\begin{proof} Let $j = j(M)$; then the $KG$-module $N := \Ext^j_{KG}(M,KG)$ is finitely generated and nonzero. Pick any finitely generated $RG$-submodule and $R$-lattice $F_0N$ in $N$. By Corollary \ref{MainTech}, $N$ is $S_n$-torsion-free for some integer $n$ which we may as well assume to be positive. Let $B = \mathcal{U}_n \ast H_n$; then $B \otimes_{KG} N = B \otimes_{RG} F_0N$ is non-zero by Proposition \ref{CrossProd}(e). Now  $B$ is a flat $RG$-module by Proposition \ref{CrossProd}(d) so $B$ is also a flat $KG$-module and $\injdim KG = \injdim B$ by \cite[Lemma 5.4]{ArdBro2006}, Proposition \ref{CdimAmod} and the Lemma, so we may apply Proposition \ref{basechange1} to deduce
\[d_{KG}(M) = d_B(B \otimes_{KG} M).\]
Since $M' := B \otimes_{KG} M$ is a finitely generated $B$-module and $B$ is a finitely generated $\mathcal{U}_n$-module, $M'$ is a finitely generated $\mathcal{U}_n$-module. Finally, $d_B(M') = d_{KG}(M) \geq 1$ by the Lemma because $M$ is infinite dimensional over $K$, and $d_B(M') = d_{\mathcal{U}_n}(M')$ by \cite[Lemma 5.4]{ArdBro2006}. Hence we may apply Theorem \ref{LowerBound} to deduce that $d_{KG}(M) = d_{\mathcal{U}_n}(M') \geq r$.
\end{proof}

We can now prove Theorem \ref{main} from the Introduction.

\begin{proof}[Proof of Theorem \ref{main}] Suppose that $K$, $G$ and $M$ are as in the statement of Theorem \ref{main}. By restricting from $KG$ to $\Qp G$ and applying Lemma \ref{basechange1} we may assume that $K = \Qp$. Since the Lie algebra of $G$ is split semisimple, we can find an open uniform subgroup $H$ of $G$ that satisfies the assumptions of $\S\ref{CongKern}$. Choose an open normal subgroup $N$ of $G$ contained in $H$; then $d_{KG}(M) = d_{KN}(M) = d_{KH}(M)$ by \cite[Lemma 5.4]{ArdBro2007} and the result follows from the Theorem above.
\end{proof}

\section{Finite dimensional $KG$-modules}

In this section we study finite dimensional $KG$-modules. The results and proof techniques in $\S\ref{FDKGmod}$-\ref{Unique} are similar to those for distribution algebras found in Prasad's appendix to \cite{ST1}. It is not clear to us how to deduce our results directly from those found there.

\subsection{Lie modules and Artin modules}\label{FDKGmod}
We continue with the notation of $\S\ref{CongKern}$, but drop the restriction on $p$.

\begin{defn} We let $\mathcal{M}$ denote the category of $KG$-modules which are finite dimensional over $K$, and all $KG$-module homomorphisms. We let $\mathcal{L}$ denote the full subcategory of $\mathcal{M}$ consisting modules obtained from finite dimensional $\mathcal{U}_0$-modules by restriction; we tentatively call objects in $\mathcal{L}$ \emph{Lie modules}. We say that $V \in \mathcal{M}$ is an \emph{Artin module} if some open subgroup of $G$ acts trivially on $V$, and let $\mathcal{A}$ denote the full subcategory consisting of Artin modules.
\end{defn}

We will denote an Artin module $V$ by the corresponding representation $\rho : G \to GL(V)$. Note that if $V$ is an abstract $K[G]$-module such that some open normal subgroup $U$ of $G$ acts trivially, then $V$ is automatically an Artin module because the action of $K[G]$ on $V$ factors through $K[G/U]$ and the completed group ring $KG$ surjects onto $K[G/U]$. We will use this observation without further mention in what follows.

\begin{prop} Let $M \in \mathcal{M}$ and suppose that $M$ is $S_n$-torsion-free. Then the natural map
\[\eta_M : M \longrightarrow (\mathcal{U}_n \ast H_n) \otimes_{KG} M\]
is an isomorphism of $KG$-modules.
\end{prop}
\begin{proof}Suppose first that $n = 0$ and write $S = S_0$. Choose a finitely generated $RG$-submodule and $R$-lattice $N$ in $M$, and choose a good filtration $F_\bullet N$ on $N$ for the $\fr{m}$-adic filtration on $RG$. Then $\mathcal{U}_0 \otimes_{KG} M = Q_{\gr p}(RG) \otimes_{RG} N$ is the microlocalisation $Q_{\gr p}(N)$ of $N$ at $\gr p$, which in turn is isomorphic to the completion of $N_S := (RG)_S \otimes_{RG} N$ with respect to a certain filtration $F_\bullet (N_S)$ on this $RG_S$-module --- see $\S \ref{Micro}$.  The filtration $F_\bullet (N_S)$ is good, and the filtration on $(RG)_S$ is Zariskian by Lemma \ref{Micro}, so $F_\bullet N_S$ is separated by \cite[Corollary I.5.5]{LVO}.

Let $L = F_0(N_S)$; then because $p^i \in F_{-i} (RG)_S$ for all $i$ we see that $F_i(N_S) = p^{-i}L$ for all $i$. Since $F_\bullet N_S$ is separated, $L$ is an $R$-lattice in $N_S$ and $Q_{\gr p}(N)$ is the completion of $N_S$ with respect to this lattice.

Now because $p \in S$, the partial localisation $M = N_K$ of $N$ is contained in $N_S$. Let $s \in S$. Since $M = N_K$ is $S$-torsion-free by assumption, $s$ acts injectively on $M$. Because $M$ is finite dimensional over $K$, the action of $s$ is actually surjective, and therefore by the universal property of localisation the natural map $M \to N_S$ is an isomorphism. So $N_S$ is finite dimensional over $K$ and hence $L$ is finitely generated over $R$ by Proposition \ref{Latt}. But finitely generated $R$-modules are already $p$-adically complete, so the natural map $N_S \to Q_{\gr p}(N)$ is an isomorphism.

Returning to the general case, let $M_n =  M$ denote the restriction of $M$ to $KG^{p^n}$. Then there is a commutative diagram of $KG^{p^n}$-modules
\[\xymatrix{ M \ar[r]^-{\eta_M} \ar[d] & (\mathcal{U}_n \ast H_n)\otimes_{KG} M \ar[d] \\ M_n \ar[r] & \mathcal{U}_n \otimes_{KG^{p^n}} M_n }\]
where the left column is the identity map and the right column is the isomorphism $(\mathcal{U}_n \ast H_n) \otimes_{KG} M = (\mathcal{U}_n \otimes_{KG^{p^n}} KG) \otimes_{KG} M \tocong \mathcal{U}_n \otimes_{KG^{p^n}} M_n$ given by the definition of the crossed product $\mathcal{U}_n \ast H_n$. Since the bottom row is a bijection by the case $n = 0$ applied to the $KG^{p^n}$-module $M_n$, $\eta_M$ is a bijection and the result follows.
\end{proof}

\begin{thm}\hfill \be\item The category $\mathcal{M}$ is semisimple: every submodule $W$ of $V \in \mathcal{M}$ has a complement.
\item $\mathcal{U}_0 \otimes_{KG} - $ is an equivalence of categories between $\mathcal{L}$ and the category of finite dimensional $U(\fr{g}_K)$-modules.
\ee \end{thm}
\begin{proof} (a) By Corollary \ref{MainTech}, $V$ is $S_n$-torsion-free for some $n$. Let us identify $V$ with $(\mathcal{U}_n \ast H_n) \otimes_{KG} V$ using the Proposition; then $V$ is a finite dimensional $\mathcal{U}_n \ast H_n$-module and $W$ is a $\mathcal{U}_n \ast H_n$-submodule. By Corollary \ref{FDAmod}, we can find a $\mathcal{U}_n$-linear projection $\sigma$ from $V$ onto $W$. Since we're working over a field of characteristic zero, the average $\sigma' := \frac{1}{|H_n|} \sum_{h \in H_n} \overline{h} \sigma \overline{h}^{-1}$ of the $H_n$-conjugates of $\sigma$ is a $\mathcal{U}_n \ast H_n$-linear projection of $V$ onto $W$ --- see \cite[Lemma 1.1]{Pass} for more details. Now $\ker \sigma'$ is a $KG$-stable complement to $W$ in $V$.

(b) Let $\mathcal{C}$ denote the category of finite dimensional $\mathcal{U}_0$-modules. The restriction functor $\mb{R} = \Hom_{KG}(\mathcal{U}_0, -)$ is right adjoint to the base-change functor $\mb{L} = {\mathcal{U}_0\otimes_{KG} -}$, and sends $\mathcal{C}$ to $\mathcal{L}$ by definition. If $V = \mb{R}W \in \mathcal{L}$, then $S_0$ acts invertibly on $V$ because $S_0$ consists of units in $\mathcal{U}_0$. Therefore $V$ is $S_0$-torsion-free and the counit of the adjunction $\eta_V : V \to \mb{R}\mb{L}V$ is an isomorphism by the Proposition, which implies that $\mb{L}$ sends $\mathcal{L}$ to $\mathcal{C}$. For $W \in \mathcal{C}$, the unit of the adjunction $\epsilon_W : \mb{L}\mb{R} W \to W$ satisfies $\mb{R}(\epsilon_W) \circ \eta_{\mb{R}W} = 1_{\mb{R}W}$ by the unit-counit equation. So $\mb{R}(\epsilon_W)$ is an isomorphism because $\eta_{\mb{R}W}$ is an isomorphism. Therefore $\epsilon_W$ is bijective and hence an isomorphism, so $\mb{L} : \mathcal{L} \to \mathcal{C}$ is an equivalence of categories. But we already know from Corollary \ref{FDAmod} that $\mathcal{C}$ is equivalent to the category of finite dimensional $U(\fr{g}_K)$-modules via restriction along the inclusion $U(\fr{g}_K) \hookrightarrow \mathcal{U}_0$.
\end{proof}

\begin{cor} Let $V \in \mathcal{L}$ be simple. Then $\End_{KG}(V) = K$.\end{cor}
\begin{proof} Apply part (b) of the Theorem and \cite[Proposition 7.1.4(iv)]{Dix}.
\end{proof}
We will see shortly that $\mathcal{M}$ is built from $\mathcal{L}$ and $\mathcal{A}$ in a very precise way: in fact, $\mathcal{M}$ is the tensor product of $\mathcal{L}$ and $\mathcal{A}$ in the sense of \cite[$\S$5]{Del}.

\subsection{A factorization theorem}\label{Factor}
Recall that a module $V$ is said to be \emph{isotypic} if $V \cong W^s$ for some simple module $W$ and some $s>0$.

\begin{lem} Let $M \in \mathcal{M}$ be a simple module which is $S_n$-torsion-free. Then $M' := (\mathcal{U}_n \ast H_n) \otimes_{KG} M$ is an isotypic $\mathcal{U}_n$-module.\end{lem}
\begin{proof} The group $G$ acts by conjugation on the algebra $\mathcal{U}_n$ and this action fixes the Harish-Chandra centre $Z_n := \widehat{U(\fr{g})^{\mb{G}}_{n+m, K}}$ of $\mathcal{U}_n = \widehat{U(\fr{g})_{n+m, K}}$, because this conjugation action is induced by the adjoint representation of $G$ on $\fr{g}_K$. Let $W$ be a non-zero simple $\mathcal{U}_n$-submodule of $M'$ and let $g \in G$; then $gW$ is another simple $\mathcal{U}_n$-submodule with the same action of $Z_n$. But up to isomorphism, there is only one simple finite dimensional $\mathcal{U}_n$-module with this action of $Z_n$, so $gW \cong W$ as a $\mathcal{U}_n$-module.

Let $N$ be the image of $(\mathcal{U}_n \ast H_n)\otimes_{\mathcal{U}_n} W$ in $M'$. This is a non-zero $KG$-submodule of $M'$ by construction. But $M'$ is a simple $KG$-module by Proposition \ref{FDKGmod} so $N = M'$. Now $N$ is a finite sum of modules of the form $gW$ as a $\mathcal{U}_n$-module; since finite dimensional $\mathcal{U}_n$-modules are semisimple by Corollary \ref{FDAmod}, it follows that $M' = N$ is an isotypic $\mathcal{U}_n$-module.
\end{proof}

\begin{thm} Let $M $ be a simple finite dimensional $KG$-module. Then $M \cong V \otimes \rho$ for some simple $V \in \mathcal{L}$ and some simple $\rho \in \mathcal{A}$.\end{thm}
\begin{proof} Using Corollary \ref{MainTech}, choose $n$ such that $M$ is $S_n$-torsion-free; then $M' := (\mathcal{U}_n \ast H_n) \otimes_{KG} M$ is an isotypic $\mathcal{U}_n$-module by the Lemma, so $M' \cong W^t$ for some simple, finite dimensional $\mathcal{U}_n$-module $W$. By Corollary \ref{FDAmod}, we may assume that $W$ is the restriction to $\mathcal{U}_n$ of a simple finite dimensional $\mathcal{U}_0$-module. Let $V = \mb{R}W$ be the corresponding object in $\mathcal{L}$; note that $V$ is simple by Theorem \ref{FDKGmod}(b). Consider the vector space
\[ \rho := \Hom_{KG^{p^n}}(V, M).\]
Because $V$ and $M$ are finite dimensional over $K$, $\rho$ is also finite dimensional over $K$. Moreover the inclusion $W \hookrightarrow M'$ gives by restriction to $KG^{p^n}$ a non-zero element of $\rho$ so $\rho \neq 0$. The rule
\[ (g.f)(v) = gf(g^{-1}v),\quad \quad g \in G, v \in V, f \in \rho\]
defines an action of $G$ on $\rho$ which by definition is trivial on $G^{p^n}$; thus $\rho \in \mathcal{A}$ is an Artin representation of $G$. Now $KG$ acts diagonally on the tensor product $V \otimes \rho$ and there is a natural $KG$-module map
\[ \theta : V \otimes \rho \to M\]
given by evaluation: $\theta(v\otimes f) = f(v)$. The map $\theta$ is non-zero because $\rho \neq 0$; since $M$ is simple it follows that $\theta$ is surjective. But $\dim_K M = \dim_K M' = t \dim_K W$ by Proposition \ref{FDKGmod}, and
 \[\dim_K \rho = \dim_K \Hom_{KG^{p^n}}(W, W^t) = t \dim_K \Hom_{KG^{p^n}}(W, W) = t\]
by Corollary \ref{FDKGmod} applied to the group $G^{p^n}$. Hence $\dim_K V \otimes \rho = t \dim_KW$ also and $\theta$ is an isomorphism. Finally, if $\rho'$ is a non-zero submodule of $\rho$ then $V \otimes \rho'$ is a non-zero submodule of $V\otimes \rho$. Because $V \otimes \rho \cong M$ is simple, this submodule must be the whole of $V \otimes \rho$ and a dimension count shows that $\rho' = \rho$.
\end{proof}

\subsection{Uniqueness of factorization}
\label{Unique}
\begin{lem} Let $\rho$ be an Artin representation which is trivial on $G^{p^n}$ and let $V \in \mathcal{L}$ be simple. Then the natural map
\[ \psi : \rho \to \Hom_{KG^{p^n}}(V, V\otimes \rho)\]
given by $\psi(x)(v) = v \otimes x$ is an isomorphism of $KG$-modules.
\end{lem}
\begin{proof} Since $G^{p^n}$ is normal in $G$, the group $G$ acts on $\Hom_{KG^{p^n}}(V, V\otimes \rho)$ via $(g.f)(v) = g.f(g^{-1}v)$. The action of $G^{p^n}$ is trivial, so $\Hom_{KG^{p^n}}(V, V\otimes \rho)$ is an Artin representation of $G$.

It is straightforward to verify that $\psi$ is an injective $KG$-module homomorphism. Because $\rho$ is trivial on $G^{p^n}$, the restriction of $V \otimes \rho$ to $KG^{p^n}$ is a direct sum of $\dim_K \rho$ copies of the restriction of $V$, so
\[\dim_K \Hom_{KG^{p^n}}(V, V \otimes \rho) = \dim_K \rho \cdot \dim_K\End_{KG^{p^n}}(V).\]
It now follows from Theorem \ref{FDKGmod}(b) that $V$ is still simple as a $KG^{p^n}$-module, so $\End_{KG^{p^n}}(V) = K$ by Corollary \ref{FDKGmod}. Hence $\rho$ is an isomorphism.
\end{proof}

We can now give a partial classification of simple finite dimensional $KG$-modules for compact $p$-adic analytic groups $G$ satisfying the hypotheses of $\S \ref{CongKern}$.
\begin{thm}Let $V \in \mathcal{L}$ and $\rho \in \mathcal{A}$ be simple.
\be
\item The module $V \otimes \rho \in \mathcal{M}$ is simple.
\item If $V \otimes \rho \cong V' \otimes \rho'$ for some simple $V' \in \mathcal{L}$ and $\rho' \in \mathcal{A}$, then $V \cong V'$ and $\rho \cong \rho'$.
\ee \end{thm}
\begin{proof}(a) By the proof of \cite[Lemma 4.4(c)]{ArdWad2006}, there is a $K$-linear isomorphism
\[ \Hom_{KG}(V \otimes \rho, V \otimes \rho) \cong \Hom_{KG}(V,V \otimes \rho \otimes \rho^\ast)\]
where $\rho^\ast$ denotes the Artin representation dual to $\rho$. But
\[ \Hom_{KG}(V,V \otimes \rho \otimes \rho^\ast) = \left(\Hom_{KG^{p^n}}(V,V \otimes \rho \otimes \rho^\ast)\right)^G \cong (\rho \otimes \rho^\ast)^G \cong \Hom_{KG}(\rho, \rho)\]
by the Lemma. So $\dim_K \End_{KG}(V \otimes \rho) = \dim_K \End_{KG}(\rho)$. Now there is a natural ring homomorphism
\[\End_{KG}(\rho) \to \End_{KG}(V \otimes \rho)\]
given by $f \mapsto 1_V \otimes f$. It is easy to see that this map is injective, so it must be an isomorphism by considering dimensions. Hence $\End_{KG}(V \otimes \rho)$ is a division ring by Schur's Lemma.

On the other hand the category  $\mathcal{M}$ is semisimple by Theorem \ref{FDKGmod}(a), so $V \otimes \rho = V_1^{n_1} \oplus \cdots \oplus V_t^{n_t}$ for some simple $V_i \in \mathcal{M}$. Therefore
\[\End_{KG}(V \otimes \rho) \cong M_{n_1}(D_1) \oplus \cdots \oplus M_{n_t}(D_t)\]
is a direct sum of matrix algebras over division rings $D_i = \End_{KG}(V_i)$. Since $\End_{KG}(V \otimes \rho)$ is itself a division ring, this can only happen if $r = 1$ and $n_1 = 1$. Hence $V \otimes \rho = V_1$ is simple.

(b) Choose $n$ large enough so that $\rho$ and $\rho'$ are trivial on $G^{p^n}$; then the restriction of $V \otimes \rho$ to $KG^{p^n}$ is isomorphic to both $V^{\dim_K\rho}$ and $(V')^{\dim_K \rho'}$. Theorem \ref{FDKGmod}(b) now implies that $V \cong V'$ in $\mathcal{L}$, and therefore
\[\rho \cong \Hom_{KG^{p^n}}(V, V \otimes \rho) \cong \Hom_{KG^{p^n}}(V', V' \otimes \rho') \cong \rho'\]
as $KG$-modules by the Lemma.
\end{proof}

Combining Theorem \ref{FDKGmod}(a), Theorem \ref{Factor} and Theorem \ref{Unique} gives the following

\begin{cor}\hfill\be \item There is a bijection between the isomorphism classes of simple objects in $\mathcal{M}$ and pairs $([V], [\rho])$ of isomorphism classes of simple objects $V \in \mathcal{L}$ and $\rho \in \mathcal{A}$.
\item There is an isomorphism of Grothendieck groups
\[K_0(\mathcal{M}) \cong K_0(\mathcal{L}) \otimes_{\mathbb{Z}} K_0(\mathcal{A}).\]
\ee\end{cor}

\subsection{The Grothendieck group of Artin representations}
We finish this paper by giving a description of the Grothendieck group of the category $\mathcal{A}$ of Artin $K$-representations of an arbitrary pro-$p$ group $G$. We still assume that $K$ is a complete discrete valuation field of characteristic zero and uniformizer $p$. Let $K(\mu_{p^\infty})$ be the infinite totally ramified field extension of $K$ obtained by adjoining the set $\mu_{p^\infty}$ of all $p$-power roots of unity to $K$, and let
\[\Gamma = \Gal(K(\mu_{p^\infty}) / K).\]
We fix an isomorphism $\Qp / \Zp \to \mu_{p^\infty}$, which induces an isomorphism $\Aut(\mu_{p^\infty}) \to \Zp^\times$. Since $K$ is unramified, the action of $\Gamma$ on $\mu_{p^\infty}$ is faithful and gives rise to the cyclotomic character $\chi : \Gamma \tocong \Zp^\times$ given explicitly by
\[\sigma(\lambda) = \lambda^{\chi(\sigma)}\quad\mbox{ for all }\quad \sigma \in \Gamma \quad\mbox{and}\quad \lambda \in \mu_{p^\infty}.\]
Imitating \cite[\S 12.4]{Serre}, we define a continuous permutation action of $\Gamma$ on our pro-$p$ group $G$ by the similar rule
\[\sigma . g = g^{\chi(\sigma)}\quad\mbox{ for all }\quad \sigma \in \Gamma \quad\mbox{and}\quad g \in G.\]
If $X$ is a compact totally disconnected topological space, let $C^\infty(X, K)$ denote the $K$-algebra of locally constant $K$-valued functions on $X$. Any continuous action of $\Gamma$ on $X$ induces an action of $\Gamma$ on $C^\infty(X,K)$ by the rule
\[(\sigma . f)(x) = f(\sigma^{-1}.x)\quad\mbox{for all} \quad \sigma \in \Gamma , f \in C^\infty(X,K), x \in X.\]
Now let $\rho : G \to V$ be an Artin $K$-representation of $G$. Then by definition, $\ker \rho$ is an open normal subgroup of $G$, so $\rho(G)$ is a finite $p$-subgroup of $\GL(V)$. Let $\chi_\rho : G \to K$ be the \emph{character} of $\rho$, defined in the usual way by
\[\chi_\rho(g) = \tr \rho(g) \quad\mbox{for all}\quad g \in G. \]
This function is locally constant, being constant on the cosets of $\ker \rho$, so $\chi_\rho \in C^\infty(G,K)$. It is also constant on conjugacy classes of $G$, so $\chi_\rho \in C^\infty(G,K)^G$ where we let $G$ act on itself by conjugation. If $g \in G$ and $\omega_i$ are the eigenvalues of $\rho(g)$, then $\omega_i^{\chi(\sigma)}$ are the eigenvalues of $\rho(\sigma.g)$ for any $\sigma \in \Gamma$. This shows that
\[\chi_\rho(\sigma.g) = \chi_\rho(g) \quad\mbox{for all}\quad \sigma \in \Gamma \quad\mbox{and}\quad g \in G.\]
Thus $\chi_\rho$ is a $\Gamma$-invariant, locally constant, class function on $G$:
\[\chi_\rho \in C^\infty(G, K)^{G \times \Gamma}.\]

\begin{prop} Let $G$ be a pro-$p$ group, and let $\mathcal{A}$ be its category of $K$-linear Artin representations. Then there is a natural $K$-algebra isomorphism
\[\chi : K\otimes_{\mathbb{Z}} K_0(\mathcal{A}) \longrightarrow C^\infty(G, K)^{G \times \Gamma}\]
given by $\chi( \lambda \otimes [\rho] ) = \lambda \chi_\rho$.
\end{prop}
\begin{proof} Let $\mathcal{U}$ denote the set of open normal subgroups of $G$ and for each $U \in \mathcal{U}$, let $G_0(K[G/U])$ denote the Grothendieck group of the category of all finitely generated (hence finite dimensional) $K[G/U]$ modules. For each pair $U, W \in \mathcal{U}$ such that $U \supseteq W$, there is a natural commutative diagram
\[\xymatrix{ K\otimes_{\mathbb{Z}} G_0(K[G/U]) \ar[r]^\chi \ar[d] & C^\infty(G/U, K)^{G \times \Gamma} \ar[d]\\
 K\otimes_{\mathbb{Z}} G_0(K[G/W]) \ar[r]^\chi \ar[d] & C^\infty(G/W, K)^{G \times \Gamma} \ar[d]\\
 K\otimes_{\mathbb{Z}} K_0(\mathcal{A}) \ar[r]^\chi & C^\infty(G, K)^{G \times \Gamma} }\]
where the vertical maps in the left-hand column are obtained by inflation, and those in the right-hand column are obtained by pull-back of functions.  Now $\mathcal{A}$ is the filtered limit of the categories of finite dimensional $K[G/U]$-modules with respect to the inflation functors as $U$ runs over $\mathcal{U}$, so
\[ K\otimes_{\mathbb{Z}} K_0(\mathcal{A}) \cong \lim\limits_{\longrightarrow} K\otimes_{\mathbb{Z}} G_0(K[G/W])\]
by \cite[Lemma II.6.2.7]{WeiK}. Since each locally constant function on $G$ must be constant on the cosets of at least one open normal subgroup $U$, we also have
\[ C^\infty(G, K)^{G \times \Gamma} \cong \lim\limits_{\longrightarrow} C^\infty(G/U, K)^{G \times \Gamma}.\]
Because the top two rows in the diagram are isomorphisms by \cite[Chapter 12,Theorem 25]{Serre}, the bottom row must therefore also be an isomorphism. Finally, direct sum and tensor product gives $K_0(\mathcal{A})$ the structure of a commutative ring, it is straightforward to check that $\chi$ is additive and multiplicative.
\end{proof}

We finally remark that in the case when $G$ is a uniform pro-$p$ group, Kirillov's orbit method provides an explicit bijection between $G$-orbits in the Pontryagin dual of the Lie algebra $L_G$, and isomorphism classes of irreducible complex Artin representations of $G$. This bijection is directly compatible with the above isomorphism $\chi$. See \cite{Jaikin} and \cite{BoSab} for more details.
\bibliography{references}
\bibliographystyle{plain}
\end{document}